\documentclass[10pt, a4paper]{amsart}
\usepackage[utf8]{inputenc}
\usepackage[T1]{fontenc}
\usepackage{amsmath}
\usepackage{amssymb}
\usepackage{amsfonts}
\usepackage{amscd}
\usepackage{stmaryrd}
\usepackage{enumerate}
\usepackage[margin = 1in]{geometry}
\usepackage{amsthm}
\usepackage{mathrsfs}
\usepackage{mathtools}
\usepackage{amssymb}
\usepackage[all]{xy}
\usepackage{xcolor}
\usepackage{graphics}
\usepackage{lscape}
\usepackage{array}
\usepackage{microtype}
\usepackage{setspace}
\usepackage{parskip}
\usepackage{marvosym}
\usepackage{tikz-cd} 
\usetikzlibrary{graphs,decorations.pathmorphing,decorations.markings}
\usepackage{stmaryrd} 
\usepackage{upgreek} 
\usepackage{centernot} 
\usepackage{comment}
\usepackage[shortlabels]{enumitem}
\usepackage{soul} 

\usepackage[style=alphabetic,
            isbn=false,
            doi=false,
            url=false,
            backend=biber,
            maxnames =5, 
            maxalphanames = 4, giveninits = true
           ]{biblatex}
\defbibheading{bibliography}[\bibname]{%
  \section*{References}%
}
\AtEveryBibitem{\clearlist{language}}

\bibliography{BibFM1}

\usepackage{xy}

\makeatletter
\def\l@subsection{\@tocline{2}{0pt}{2.5pc}{5pc}{}}
\makeatother

\numberwithin{equation}{section}
\newtheorem*{theorem*}{Theorem}
\newtheorem*{definition*}{Definition}
\newtheorem*{theorem_A}{Theorem A}

\newtheorem*{corollary_B}{Corollary B}
\newtheorem*{corollary_C}{Corollary C}

\newtheorem{theorem}{Theorem}[section]
\newtheorem{lemma}[theorem]{Lemma}
\newtheorem{proposition}[theorem]{Proposition}
\newtheorem{corollary}[theorem]{Corollary}

\newtheorem{assumption}[theorem]{Assumption}
\newtheorem{remark}[theorem]{Remark}
\theoremstyle{definition}
\newtheorem{definition}[theorem]{Definition}
\newtheorem{example}[theorem]{Example}
\newtheorem{notation}[theorem]{Notation}

\usepackage{hyperref}\hypersetup{colorlinks}

\usepackage{color} 

\definecolor{darkred}{rgb}{1,0,0} 
\definecolor{darkgreen}{rgb}{0,1,0}
\definecolor{darkblue}{rgb}{0,0,1}

\hypersetup{colorlinks,
linkcolor=darkblue,
filecolor=darkgreen,
urlcolor=darkred,
citecolor=darkblue}

\newcommand{\Spec}{\mathsf{Spec}}

\newcommand{\supp}{\mathsf{supp}}
\newcommand{\Sym}{\mathsf{Sym}}

\renewcommand{\ker}{\mathsf{ker}}

\newcommand{\im}{\mathsf{im}}

\newcommand{\Jac}{\mathsf{Jac}}
\newcommand{\Pic}{\mathsf{Pic}}
\newcommand{\Sing}{\mathsf{Sing}}

\newcommand{\chow}{\mathsf{chow}}
\newcommand{\length}{\mathsf{length}}

\newcommand{\PMod}{\mathsf{PMod}}
\newcommand{\Hilb}{\mathsf{Hilb}}
\newcommand{\Flag}{\mathsf{Flag}}

\newcommand{\Higgs}{\mathsf{Higgs}}

\newcommand{\Exc}{\mathsf{Exc}}
\newcommand{\RSing}{\mathsf{RSing}}
\newcommand{\URSing}{\mathsf{URSing}}
\newcommand{\Hyb}{\mathsf{H}}

\newcommand{\p}{\mathsf{p}}
\newcommand{\q}{\mathsf{q}}
\renewcommand{\r}{\mathsf{r}}
\renewcommand{\t}{\mathsf{t}}

\newcommand{\sym}{\mathfrak{S}}
\newcommand{\wt}{\widetilde}

\newcommand{\rest}{\: \rule[-3.5pt]{0.5pt}{11.5pt}\,{}}
\renewcommand{\to}{\longrightarrow}
\newcommand{\ol}[1]{\overline{#1}}

\newcommand{\calA}{\mathcal{A}}
\newcommand{\calB}{\mathcal{B}}

\newcommand{\calD}{\mathcal{D}}
\newcommand{\calE}{\mathcal{E}}
\newcommand{\calF}{\mathcal{F}}
\newcommand{\calG}{\mathcal{G}}
\newcommand{\calH}{\mathcal{H}}

\newcommand{\calI}{\mathcal{I}}

\newcommand{\calL}{\mathcal{L}}
\newcommand{\calM}{\mathcal{M}}

\newcommand{\calO}{\mathcal{O}}

\newcommand{\calP}{\mathcal{P}}
\newcommand{\calQ}{\mathcal{Q}}
\newcommand{\calR}{\mathcal{R}}

\newcommand{\calT}{\mathcal{T}}
\newcommand{\calU}{\mathcal{U}}
\newcommand{\calV}{\mathcal{V}}

\newcommand{\calHom}{\mathcal{H}\!om}
\newcommand{\calExt}{\mathcal{E}\!xt}
\newcommand{\calTor}{\mathcal{T}\!or}

\renewcommand{\AA}{\mathbb{A}}
\newcommand{\CC}{\mathbb{C}}
\newcommand{\C}{\mathbb{C}}

\newcommand{\NN}{\mathbb{N}}
\newcommand{\PP}{\mathbb{P}}

\newcommand{\ZZ}{\mathbb{Z}}
\newcommand{\WW}{\mathbb{W}}

\newcommand{\id}{\mathrm{id}}

\newcommand{\Prym}{\mathsf{Prym}}
\newcommand{\Nm}{\mathsf{Nm}}

\newcommand{\GL}{\mathsf{GL}}
\newcommand{\SL}{\mathsf{SL}}

\newcommand{\PGL}{\mathsf{PGL}}

\newcommand{\morph}[6]{\begin{array}{cccc} #6: & #1  & \stackrel{#5}{\longrightarrow} &  #2  \\ & #3 & \longmapsto & #4  \end{array}}

\newcommand{\quotient}[2]{{\raisebox{.2em}{\thinspace $#1$}\left / \raisebox{-.15em}{ $#2$}\right.}}

\newcommand{\longhookrightarrow}{\lhook\joinrel\longrightarrow}

\title{Fourier--Mukai transforms and normalisation of nodal curves}

\author[E. Franco]{Emilio Franco}
\address{E. Franco,
\newline\indent Universidad Aut\'onoma de Madrid 
\newline\indent and
\newline\indent Instituto de Ciencias Matem\'aticas (CSIC--UAM--UCM--UC3M)
\newline\indent Campus de Cantoblanco 28049, Madrid, Espa\~na.}
\email{emilio.franco@uam.es}

\author[R. Hanson]{Robert Hanson}
\address{R. Hanson, 
\newline\indent Imperial College London, 
\newline\indent Exhibition Rd, South Kensington, London SW7 2AZ}
\email{robert.hanson@imperial.ac.uk}

\author[J. Horn]{Johannes Horn}
\address{J. Horn,
\newline\indent Goethe-Universit\"at Frankfurt, 
\newline\indent Robert-Mayer-Str. 6-8, 60325 Frankfurt am Main}
\email{horn@math.uni-frankfurt.de}

\author[A. Oliveira]{André Oliveira}
\address{A. Oliveira,
\newline\indent Centro de Matemática da Universidade do Porto,
\newline\indent Departamento de Matemática, Faculdade de Ciências, Universidade do Porto,
\newline\indent Rua do Campo Alegre s/n, 4169-007 Porto, Portugal}
\email{andre.oliveira@fc.up.pt}

\thanks{
First author supported by the Spanish Ministry of Science and Innovation, through project PID2022-141387NB-C22 and the \textit{Severo Ochoa Programme for Centres of Excellence in R$\&$D} (CEX2019-000904-S). Second author supported by the Horizon Europe Marie Skłodowska-Curie Action grant \textit{Hyperkähler mirror symmetry and Langlands duality}, grant ID 101204490. Third author supported by the Deutsche Forschungsgemeinschaft (DFG, German Research Foundation) through the Collaborative Research Centre TRR 326 \textit{Geometry and Arithmetic of Uniformized Structures}, project number 444845124. Fourth author partially supported by FCT (Fundação para a Ciência e Tecnologia), under the projects with reference UID/00144/2025, and associated DOI
 https://doi.org/10.54499/UID/00144/2025, CMUP, member of LASI, and 2024.15931.PEX \textit{Higgs bundles: geometry, algebra and physics}.
}

\begin{document}

\renewcommand{\baselinestretch}{1.10}\normalsize

\begin{abstract}
We study Arinkin's Poincaré sheaf $\calP_C$ on the singular locus of $\ol\Jac_C$, the compactified Jacobian of rank one torsion-free sheaves on an integral nodal projective curve $C$. Each stratum of the singular locus $\Sing(\ol\Jac_C)$ is indexed by a partial normalisation $\Sigma \to C$. We prove that the Poincaré sheaf $\calP_{C}$ restricted to each stratum can be expressed through the Poincaré sheaf $\calP_{\Sigma}$, obtaining a relation between Fourier--Mukai transforms associated to $\calP_C$ and $\calP_\Sigma$. Our approach uses an intermediate geometry: the moduli space of parabolic modules of Bhosle and Cook, to intertwine sheaf data over the two curves. In a sequel, our formulae are used to study mirror symmetry in singular loci of Hitchin systems.
\end{abstract}

\maketitle

\renewcommand{\baselinestretch}{0.9}\normalsize 

\input{toc.tex}
{
\small
\hypersetup{linkcolor=black}
\tableofcontents
}

\newpage

\section{Introduction}
\renewcommand{\baselinestretch}{1.12}\normalsize

In seminal work of Mukai \cite{mukai}, a categorification of the Fourier transform yields a derived equivalence $\Phi^{A} : D^b(A) \xrightarrow{\ \cong\ } D^b(A^{\vee})$ between an abelian variety $A$ and its dual $A^{\vee} := \Pic^0(A)$, with $\Phi^{A}$ satisfying properties that emulate the Fourier analysis of $L^2$-functions. Equivalences of this type, known as \textit{Fourier--Mukai transforms}, have since become widespread in algebraic geometry.  

A classical example of an autodual abelian variety is given by the Jacobian of a smooth projective curve, over which the Fourier--Mukai transform is a non-trivial derived autoequivalence. For certain singular curves $C$, the Jacobian may be replaced by the compactified Jacobian $\ol\Jac^{\,0}_C$ of rank one, degree zero torsion-free sheaves on $C$, the geometry of which is classically studied \cite{altman&kleiman, caporaso, dsouza, esteves, esteves&gagne&kleiman2, esteves&kleiman, melo0, simpson}. Extending results of Esteves--Gagn\'e--Kleiman \cite{esteves&gagne&kleiman2}, the autoduality of $\ol{\Jac}^0_C$ for integral planar curves was proven by Arinkin \cite{arinkin}, and later generalised by Melo--Rapagnetta--Viviani \cite{melo0, melo1, melo2} to fine compactified Jacobians of reduced planar curves. These works establish an autoduality isomorphism $\ol\Jac^{\,0}_C \cong \ol{\Pic}^{\,0}(\ol\Jac^{\,0}_C)$ and a Fourier--Mukai transform $\Phi^{\calP_C} : D^b(\ol\Jac^{\,0}_C) \xrightarrow{\ \cong\ } D^b(\ol\Jac^{\,0}_C)$, both controlled by the \textit{Poincar\'e sheaf} $\calP_C$ on $\ol\Jac^{\,0}_C \times\ol\Jac^{\,0}_C$; the universal family for autoduality as a moduli problem. 

Let $C$ be an integral nodal projective curve. This article and the sequels \cite{FHHO, FM2} describe a filtered approach to computing $\Phi^{\calP_C}$ over the singular locus of $\ol\Jac^{\,0}_C$. The starting point for our work is the observation that the singular locus of $\ol\Jac_C$ is stratified by subvarieties of the form 
\[
\ol\Jac_{\Sigma}^{\,0} \cong \ol\Jac_{\Sigma}^{\, - k} \longhookrightarrow \Sing(\ol\Jac^{\,0}_C) \, ,
\]
where $\nu : \Sigma \to C$ is a partial normalisation of $C$ that resolves $k$ nodes. These subvarieties have appeared in several contexts in algebraic geometry, for instance: the study of \textit{parabolic modules} \cite{bhosle:1992, cook:1993} and \textit{presentation schemes} \cite{esteves&gagne&kleiman}; the singularities of Hitchin systems via the normalisation of spectral curves \cite{ngo, ngo2, FGOP, franco&peon, horn, maulik&shen}; and Hitchin's study of `critical loci' \cite{hitchin_critical}, generically modeled on $\ol\Jac_{\Sigma}^{\, - k} \longhookrightarrow \ol\Jac^{\,0}_C$ studied in families. This article, alongside \cite{FHHO, FM2} (see also Sections \ref{se: vary nu}, \ref{se mirror}), develops an approach that compares, in a systematic way, the interactions between these different appearances of the subvarieties $\ol\Jac_{\Sigma}^{\, - k} \longhookrightarrow \ol\Jac^{\,0}_C$. 

\subsection{Outline of results} In this article we work over an integral nodal projective curve $C$ with a fixed  partial normalisation $\nu : \Sigma \to C$ resolving $k$ nodes. We consider two ways of passing between the derived categories of $\ol\Jac^{\,0}_{\Sigma}$ and $\ol\Jac^{\,0}_C$: firstly by pushforward along the map $\check{\nu} : \ol\Jac_{\Sigma}^{\, -k} \longhookrightarrow \ol\Jac^{\,0}_C$, $M \longmapsto \nu_{*}M$; and secondly by convolution over an intermediate geometry: the parabolic modules of Bhosle and Cook \cite{bhosle:1992, cook:1993}, which consist of pairs $(M,V)$ where $M\in \ol\Jac^{\, 0}_\Sigma$ and $V$ is a subsheaf of $M$ restricted to the resolved nodes. The parameter $V$ may be interpreted as gluing data on $M$, used to construct a point of $\ol\Jac_C^{\, 0}$. The resultant fine moduli space $\PMod^{0}_{\nu}$ of such objects then fits into a (non-commutative) diagram
\begin{equation}
\label{eq: convolution}
\begin{tikzcd}
 & \PMod_{\nu}^0 \arrow[dl, "\dot\nu"'] \arrow[dr, "\rho"] & \\
\ol{\Jac}_\Sigma^{\, 0} & \ol{\Jac}_\Sigma^{\, -k} \arrow[r, "\check\nu"'] \arrow[l, "\tau_{k ,y_0}"] & \ol\Jac^{\,0}_C ,
\end{tikzcd}
\end{equation}
such that $\dot\nu(M,V) = M$ is a projective fibre bundle and $\rho$ is a finite map that defines a partial resolution of singularities. We take the translation $\tau_{k, y_0} := (\bullet) \otimes \calO_\Sigma( ky_0 )$, at fixed $y_0 \in \Sigma$, to trivialise $\Jac^{\,-k}_{\Sigma}$ as a torsor and correct for degree changes caused by pushforward along $\nu$. The moduli $\PMod^{0}_{\nu}$ receives a tautological bundle $\calV_\Sigma$, universally parameterising the subsheaves $V$. By treating the arrows $\ol{\Jac}_\Sigma^{\, 0} \longleftarrow \PMod_{\nu}^0 \to \ol\Jac^{\,0}_C$ as a convolution diagram, with convolution kernel $\det(\calV_{\Sigma}) \to \PMod_{\nu}^0$, we define a pair of convolution functors
\begin{equation}
\label{eq: convolution intro}
\Theta^{\calV_{\Sigma}} := \rho_{*}(\det(\calV_{\Sigma}) \otimes \dot\nu^{*} (\bullet)) : D^b(\ol\Jac_{\Sigma}^{\, 0}) \to D^b(\ol\Jac^{\,0}_C) , 
\end{equation}
\begin{equation}
\label{eq: wt convolution intro}
\wt{\Theta}^{\calV_{\Sigma}} := (\id \times \rho)_{*}(\q_2^{*}\det(\calV_{\Sigma}) \otimes (\id \times \dot\nu)^{*} (\bullet)) : D^b(\ol\Jac_{\Sigma}^{\, -k} \times \ol\Jac_{\Sigma}^{\, 0}) \to D^b(\ol\Jac_{\Sigma}^{\, -k} \times \ol\Jac^{\,0}_C) , \end{equation}
with $\wt{\Theta}^{\calV_{\Sigma}}$ a relative version of $\Theta^{\calV_{\Sigma}}$, and $\q_2 : \ol{\Jac}_\Sigma^{\,-k} \times \PMod^{0}_\nu \to \PMod^{0}_\nu$ the natural projection. 

These functors are then used in the following comparison result, which identifies the relation between the respective Poincar\'e sheaves and Fourier--Mukai transforms over $\ol\Jac^{\,0}_{\Sigma}$ and $\ol\Jac^{\,0}_C$. 

\begin{theorem_A}\label{thm:A}
Let $C$ be an integral nodal projective curve, with partial normalisation $\nu: \Sigma \to C$ resolving precisely $k$ nodes. 

\begin{itemize}
    \item (Theorem \ref{tm Poincare and normalisation}). The Poincar\'e sheaves $\calP_{\Sigma}$ on $\ol{\Jac}^{\,0}_{\Sigma} \times \ol{\Jac}^{\,0}_{\Sigma}$ and $\calP_C$ on $\ol\Jac^{\,0}_C \times \ol\Jac^{\,0}_C$ are related by the isomorphism
\begin{equation*}
(\check\nu \times \id)^*\calP_C \cong \wt{\Theta}^{\calV_{\Sigma}}((\tau_{k, y_0} \times \id)^*\calP_\Sigma ).
\end{equation*}

\item (Theorem \ref{tm relation of FM dual}). The associated Fourier--Mukai transforms $\Phi^{\calP_C} : D^b(\ol\Jac^{\,0}_C) \to D^b(\ol\Jac^{\,0}_C)$ and $\Phi^{\calP_{\Sigma}} : D^b( \ol{\Jac}^{\, 0}_\Sigma ) \to D^b( \ol{\Jac}^{\, 0}_\Sigma )$ admit, for every object $\calF^{\bullet} \in D^b(\ol\Jac^{\,-k}_{\Sigma})$, an isomorphism
\begin{equation*}
\Phi^{\calP_C} \left ( \check\nu_*\calF^\bullet \right ) \cong \Theta^{\calV_{\Sigma}}(\Phi^{\calP_\Sigma}(\tau_{k, y_0,*}\calF^\bullet) ) . 
\end{equation*}
\end{itemize}
\end{theorem_A}
Theorem A can be thought of as a natural transformation law for composing integral and convolution functors, emulating classical formulae in harmonic analysis. The convolution kernel $\det(\calV_{\Sigma})$ in our setup universally captures the variation of the subsheaves $V$ in the moduli $\PMod_{\nu}^0$ of parabolic modules, thus taking into account the singular geometry of $\ol\Jac^{\, 0}_C$. In this way Theorem A describes the relationship between the three universal families on the moduli under consideration. 
 
\subsubsection*{Abelian analogue} Theorem A admits an analogue for Fourier--Mukai transforms taken on either side of an embedding $f: A \to B$ of abelian varieties. The transforms for sheaves pushed along $f$ can be computed via the dual map $f^{\vee} : B^{\vee} \to A^{\vee}$ and the natural transformation  
\begin{equation}
\label{eq abelian comparison}
\Phi^B\circ f_* \simeq (f^\vee)^*\circ\Phi^A ;
\end{equation}
see equation (11.3.3) of \cite{polishchuk}. Theorem A can be considered a singular variant of \eqref{eq abelian comparison}, where we take $\check\nu : \ol\Jac^{\, -k}_{\Sigma} \to \ol\Jac^{\, 0}_{C}$ to play the role of $f$. The matching of terms is not exact, stemming from the fact that the dual map $\check\nu^{\vee} = \nu^{*} : \Jac_C^{0} \to \Jac_{\Sigma}^{0}$ does not extend to the compactified Jacobian $\ol\Jac^{\,0}_C$. Instead, we use $\dot\nu:\PMod_{\nu}^{0} \to \ol\Jac_{\Sigma}^{\,0}$, a resolution of the rational map $\nu^*: \ol{\Jac}^{\, 0}_C \dashrightarrow \ol\Jac^{\, 0}_\Sigma$, and pullback along $\dot\nu$ in the convolution construction. This is why the pullback term in \eqref{eq abelian comparison} is replaced by convolution in Theorem A.

\subsubsection*{Spin-valued formulae} By choosing a spin structure on $\PMod^0_{\nu}$, we derive a spin-valued version of Theorem A, expressed in terms of a so-called \textit{spin-valued Wilson operator}
\[
\WW^{1/2}_{\Sigma, \otimes \Exc(\nu)} : D^b(\ol\Jac^{\, 0}_{\Sigma}) \to D^b(\ol\Jac^{\, 0}_{\Sigma}) . 
\]
See \eqref{eq: spin Wilson} for the definition of $\WW^{1/2}_{\Sigma, \otimes \Exc(\nu)}$. These functors satisfy $\WW_{\Sigma, \otimes \Exc(\nu)} \simeq \WW^{1/2}_{\Sigma, \otimes \Exc(\nu)} \circ \WW^{1/2}_{\Sigma, \otimes \Exc(\nu)}$ and so are genuine square roots of the usual (abelianised) Wilson functors, defined by a universal tensoral action (as in \cite[\textsection 4.6]{donagi&pantev}). The following is a spin-valued reformulation of Theorem A. 

\begin{corollary_B} (Corollary \ref{thm: spin and Wilson}). Adopt the same hypothesis as in Theorem A. Additionally, let $\omega_{\PMod}^{1/2}$ be one of the spin structures on $\PMod_{\nu}^{0}$ described in Corollary \ref{spin structures}.    
\begin{itemize}
    \item One has isomorphisms of Poincar\'e sheaves 
\[
(\check\nu \times \id)^*\calP_C 
\cong (\id \times \rho)_* \left ( \q_2^* \left( \omega_{\PMod}^{1/2} \otimes \dot{\nu}^{*} \calU^{1/2}_{\Sigma, \otimes \Exc(\nu)} \right) \otimes (\tau_{k, y_0} \times \dot{\nu})^*\calP_\Sigma \right ) .
\]
\item For every $\calF^\bullet \in D^b(\Jac_\Sigma^{\,-k})$, one has isomorphisms of transformed sheaves 
\[
\Phi^{\calP_C}_{1 \rightarrow 2} \left ( \check\nu_*\calF^\bullet \right) 
\cong \rho_* \left ( \omega_{\PMod}^{1/2} \otimes \dot\nu^{*} \big( \WW^{1/2}_{\Sigma, \otimes \Exc(\nu)} \circ \Phi^{\calP_\Sigma}_{1 \rightarrow 2}(\tau_{k, y_0,*}\calF^\bullet) \big) \right) .  
\]
\end{itemize} 
\end{corollary_B}
The choice of $\omega_{\PMod}^{1/2}$ is auxiliary to Corollary B, in the sense that both displayed formulae are independent of which spin structure is chosen (see Remark \ref{re: spin independence}). The appearance of square roots can be heuristically explained by the fact that $\nu : \Sigma \to C$ is two-to-one over the locus of resolved singularities, so performing a `Wilson loop' over resolved nodes of $C$ lifts to a so-called \textit{`Wilson spin-half loop'} over $\Sigma$ (see \cite[\textsection 6.1]{kapustin&witten} for the physics of Wilson loops).  

\subsubsection*{Prym-valued formulae} We also establish an analogue of Theorem A over the compactified Prym varieties 
\[
\ol\Prym^{\,0}_{\Sigma} \subset \ol\Jac^{\,0}_{\Sigma} , \quad \quad \ol\Prym^{\,0}_{C} \subset \ol\Jac^{\,0}_{C} . 
\]
These are defined as preimages of certain \textit{norm maps}, associated to a pair of finite, ramified and commuting $n$-coverings $\beta_C : C \to X$ and $\beta_\Sigma : \Sigma \to X$ over a smooth projective curve $X$. Our Prymian results are based on the Fourier--Mukai transforms 
\[
\Psi^{\calR_{C}}_{1 \rightarrow 2} : D^b(\ol\Prym_C^{\, 0}) \xrightarrow{\ \cong\ } D^b (\ol\Prym_C^{\, 0}, \Gamma ) , \quad \quad 
\Psi^{\calR_{\Sigma}}_{1 \rightarrow 2} : D^b(\ol\Prym_{\Sigma}^{\, 0}) \xrightarrow{\ \cong\ } D^b(\ol\Prym_\Sigma^{\, 0}, \Gamma ) \, , 
\]
established in \cite[Thm 4.8]{FHR} and \cite[Thm 4.7]{groechenig&shen}. Here $\Gamma$ is the group of $n$-torsion points in $\Jac_X$, acting via tensor product, and $D^b(\ol\Prym^{\,0}_{\beta_\Sigma} , \Gamma)$ denotes the $\Gamma$-equivariant derived category. In Lemma \ref{Lemma:compatib-maps-Pryms}, we construct a Prym variant of the diagram \eqref{eq: convolution}: 
\begin{equation*}
\begin{tikzcd}
 & \PMod_{\nu}^{0} \times_{\ol{\Jac}_C^{\, 0}} \ol\Prym_C^{\, 0}  \arrow[dl, " \mathring{\nu}"'] \arrow[dr, " \mathring{\rho}"] & \\
\ol{\Prym}_\Sigma^{\, 0}  & \ol{\Prym}_\Sigma^{\, -k} \arrow[r, "\breve\nu"'] \arrow[l, "\mathring{\tau}"] & \ol\Prym^{\,0}_C ,
\end{tikzcd}
\end{equation*}
compatible with restriction from the ambient compactified Jacobians. The following is the analogue of Theorem A for compactified Prym varieties. 

\begin{corollary_C} (Corollary \ref{co: SL_n transform}). 
Let $C$ be an integral nodal curve with arithmetic genus $g$ and a partial normalisation $\nu : \Sigma \to C$. Consider compatible $n$-coverings $C \to X$ and $\Sigma \to X$ and their compactified Pryms, as described above. Then, for every $\calF^\bullet \in D^b(\ol\Prym_{\Sigma}^{\,-k})$, one has the isomorphism
\begin{equation*}
\Psi^{\calR_C}_{1 \rightarrow 2} \left ( R\breve\nu_*\calF^\bullet \right ) \cong  \mathring\rho_* \left( \imath_{\nu}^{*}\det(\calV_{\Sigma}) \otimes \mathring\nu^* \Psi^{\calR_\Sigma}_{1 \rightarrow 2} (\mathring\tau_{*} \calF^\bullet) \right). 
\end{equation*} 
\end{corollary_C}

In the remainder of the introduction we discuss future applications of the above results.

\subsection{Varying the normalisation} 
\label{se: vary nu}
By allowing the partial normalisations $\Sigma \to C$ to vary, Theorem A can be used to make global statements on the autoduality of $\ol\Jac^{\,0}_C$. The strategy is to cover $\Sing(\ol\Jac^{\,0}_C)$ by an increasing filtration of subvarieties  
\begin{equation}
\label{eq: filter Sing}
\check\nu_{n}(\ol\Jac^{\, -n}_{\Sigma_{n}}) 
\subset \cdots 
\subset \check\nu_{1}(\ol\Jac^{\, -1}_{\Sigma_{1}}) 
\subset \Sing(\ol\Jac^{\,0}_C) \
\subset \ol\Jac^{\,0}_C ,  
\end{equation} 
obtained from chains of partial normalisation maps $\Sigma_n \xrightarrow{\nu_n} \cdots \to \Sigma_1 \xrightarrow{\nu_1} C$. By applying Theorem A inductively along \eqref{eq: filter Sing}, one obtains new global descriptions of the Poincar\'e sheaf $\calP_C$ and the associated Fourier--Mukai transforms. This will be taken up in the sequel \cite{FM2}.  

\subsection{Higgs bundles and Langlands duality}
\label{se mirror}

Our work is motivated by the study of Langlands duality on moduli of Higgs bundles, most directly in the form of the \textit{Dolbeault geometric Langlands conjecture} of Donagi and Pantev \cite{donagi&pantev}. Via a \textit{classical limit} of the (global, unramified) de Rham geometric Langlands correspondence, the conjecture predicts, for Langlands self-dual group $G = \GL_n$, a derived autoequivalence on the moduli stack $\Higgs(X)$ of $\GL_n$-Higgs bundles on a smooth projective curve $X$, linear with respect to the \textit{Hitchin fibration} $\Higgs(X) \to \calB$ discovered in celebrated work of Hitchin \cite{hitchin-self}. The base $\calB$ parametrises \textit{spectral covers} $C \to X$, built from the eigenvalues of Higgs fields. The fibre at a point $C \in \calB$, known as a \textit{Hitchin fibre}, is isomorphic to $\ol{\Jac}^{\,0}_C$. 

Arinkin's Fourier--Mukai transform $D^b(\ol\Jac^{\,0}_C) \to D^b(\ol\Jac^{\,0}_C)$, understood in families, provides a solution to the $\GL_n$-Dolbeault geometric Langlands conjecture over the cuspidal (i.e. elliptic) locus $\calB^{cusp} \subset \calB$ consisting of integral spectral curves. The conjecture is solved by a relative Fourier--Mukai equivalence 
\[
\Phi^{cusp} : D^b(\Higgs(X) \times_{\calB} \calB^{cusp}) \xrightarrow{\ \cong\ } D^b(\Higgs(X) \times_{\calB} \calB^{cusp}) .
\]
Our work provides a method for computing $\Phi^{cusp}$ within a particular locus of singular spectral curves: the equisingular locus $\calB^{nodal, k} \subset \calB^{cusp}$ of irreducible nodal curves with exactly $k$ nodes. Indeed, our Theorem A can be understood as a method for computing the functor $\Phi^{nodal, k} = \Phi^{cusp} \times_{\calB} \calB^{nodal, k}$ over each strata of the singular locus $\Sing(\Higgs(X) \times_{\calB}\calB^{nodal, k})$, with each strata defined by pushforward along $\nu : \Sigma \to C$. Following the ideas outlined in Section \ref{se: vary nu}, the sequel \cite{FM2} describes how to extract more global conclusions, via an inductive computation of $\Phi^{nodal, k}$ in which all strata participate. 

Langlands duality also manifests as mirror symmetry on the hyperk\"ahler moduli space $\calM_{\Higgs}(X)$ of semistable $\GL_n$-Higgs bundles on the base curve $X$. In this context, the Hitchin fibration $\calM_{\Higgs}(X) \to \calB$ is studied as a special Lagrangian torus fibration, where autoduality of the fibre spaces $\ol{\Jac}^{\, 0}_C$ describes $\calM_{\Higgs}(X)$ as SYZ mirror self-dual, in the sense of Strominger--Yau--Zaslow \cite{SYZ}. A classical heuristic in mirror symmetry is that mirror A-branes and B-branes are interchanged by Fourier--Mukai transforms along dual SYZ fibres. On $\calM_{\Higgs}(X)$, this heuristic has been studied by many authors via the Fourier--Mukai transforms $D^b(\ol{\Jac}^{\, 0}_C) \xrightarrow{\,\cong\,} D^b(\ol{\Jac}^{\, 0}_C)$. Typically, one is constrained to the dense locus $\calB^{smth} \subset \calB$ of smooth spectral curves, where the transforms are classical. Mirror symmetry over the discriminant locus $\Delta = \calB - \calB^{smth}$ of singular and possibly non-reduced curves remains poorly understood in many aspects. 

The results of this article allow us to run the above-described mirror symmetry heuristic on so-called \textit{`critical loci'} of Hitchin \cite{hitchin_critical}, generically modeled on the family of subvarieties $\nu_{*} : \Jac^{\,-k}_{\Sigma} \longhookrightarrow \ol\Jac^{\,0}_C$, with $\nu $ the full normalisation and $C$ varying in $\calB^{nodal, k} \subset \Delta$. When the spectral cover $C \to X$ has the maximal number of nodes, work of the first and fourth named authors with Peon-Nieto and Gothen \cite{FGOP, corrigendum_FGOP} gave a description of the mirror pairing using the autoduality $\Jac_C \cong \Pic^0(\ol{\Jac}_C)$. With the Fourier–Mukai technology developed in this article, we give a twofold extension of their results: to any number of nodes, and to a more complete computation of the mirror pairing; using autoduality on the entirety of the Hitchin fibre $\ol\Jac_C^{\,0}$. This will appear in future work \cite{FHHO}. 

\subsection{Structure of the paper} 

Section \ref{se prelims} recalls Arinkin's Fourier--Mukai transform for compactified Jacobians, specialised to the case of nodal curves. Section \ref{se CH-PM-UD} introduces related moduli of divisors and sheaves used throughout the paper, including a non-standard algebraic space we call the \textit{Chow--Hilbert space:} a `mixture' of symmetric products and Hilbert schemes. We also recall the moduli spaces of parabolic modules, introduced by Bhosle and Cook, adapting their constructions to partial normalisations.  Section \ref{se: comparison} is dedicated to technical lemmas, comparing several universal families that encode information about Arinkin's Poincar\'e sheaf. This allows us, in Section \ref{sc FM and normalisation}, to prove our main result: a comparison between Poincar\'e sheaves and Fourier--Mukai transforms over $C$ and its partial normalisations, as stated above in Theorem A. We then apply this result to derive spin-valued and Prym-valued reformulations, as stated above in Corollaries B and C. 

\subsection{Acknowledgements}
We thank Jonathan Pridham and David Ben-Zvi for discussions on MathOverflow and pointing out the references \cite{gunningham&safronov, pridham}, which helped us understand the role of spin structures in this article and subsequent work \cite{FHHO}. We also thank Yifan Zhao for discussions on presentation schemes. 

\section{Autoduality for compactified Jacobians of nodal curves}
\label{se prelims}

\subsection{Compactified Jacobians} 

Let $C$ be an integral projective curve over $\C$ of arithmetic genus $g_C$, with $|\Sing(C)|\leq g_C$ nodal singularities. Given a sheaf $\calF$ on $C$, we denote its degree by $\deg(\calF) = \chi(\calF) - \chi(\calO_C)$. Consider the Jacobian $\Jac_C^d$, of degree $d$ line bundles on $C$. Let  $\nu: \Sigma \to C$ be a partial normalisation resolving $k\leq|\Sing(C)|$ nodes. $\Sigma$ is again a nodal (possibly smooth) curve of genus $g_\Sigma=g_C-k$, and we have the exact sequence of groups
\[ 
1 \to (\CC^\times)^k \to \Jac^d_C \xrightarrow{\ \nu^*} \Jac^d_\Sigma \to 1.
\] 
If $k=|\Sing(C)|$, then $\Jac^d_\Sigma$ is an abelian variety and, in this sense, $\Jac^d_C$ is said to be a (torsor for a) semi-abelian variety. Note that $\Jac_C^d$ is not proper, and the study of its compactifications is a classical problem in algebraic geometry \cite{dsouza,altman&kleiman, esteves, simpson}. We are primarily concerned with a compactification $\ol\Jac_C^{\, d}$ parametrising rank $1$ and degree $d$ torsion-free sheaves on $C$, rigidified by global scaling. For $C$ smooth, a torsion-free sheaf is automatically locally free, thus one recovers the classical Jacobian as $\ol{\Jac}_C^{\, d} = \Jac_C^{\, d}$.

\begin{notation}\label{not1}
To streamline notation we abbreviate
\[
\Jac_C:=\Jac_C^0,
\qquad
\overline{\Jac}_C:=\overline{\Jac}_C^{\, 0},
\]
and similarly over partial normalisations $\Sigma$.
\end{notation}

$\ol{\Jac}_C^{\, d}$ satisfies the following geometric properties \cite{dsouza,altman&kleiman,melo1}:  
	\begin{itemize}
		\item $\ol\Jac_C^{\, d}$ is a connected, reduced and irreducible projective scheme of dimension $g$,   with locally complete intersection singularities;
		\item $\ol\Jac_C^{\, d}$ has trivial dualising sheaf;
		\item the smooth locus of $\ol\Jac_C^{\, d}$ coincides with the dense open subset $\Jac_C^{d}$. 
    \end{itemize}

Let $\Hilb_{n, C}$ denote the punctual Hilbert scheme of $C$, parameterising zero-dimensional subschemes $D\subset C$ of length $n$. Associated to a smooth point $x_0 \in C$ is the {\it Abel--Jacobi map}
\begin{equation} \label{Abel-Jacobi}
\morph{\Hilb_{n, C}}{\ol\Jac_C}{D}{\calI^\vee_D(-n x_0),}{}{A_{C}}
\end{equation}
with $\calI_D$ the ideal sheaf of the subscheme $D\subset C$ and $\calI^\vee_D$ its dual.

We also consider the {\it curvilinear} Hilbert scheme $\Hilb^{cur}_{n,C}$, parameterising subschemes $D\in \Hilb_{n,C}$ that can be locally embedded into $\AA^1$, i.e.~$D$ is locally isomorphic to the spectrum of $\C[z]/\langle z^{n_p}\rangle$ for some $n_p\in\NN$. $\Hilb^{cur}_{n,C}$ is particularly useful for studying $\ol\Jac_C$ over nodal curves, due to the following result of Arinkin. 

\begin{proposition}[Proposition 4.5 of \cite{arinkin}] 
\label{pr alpha curv surjective}
Let $C$ be an integral nodal projective curve. There exists a sufficiently large integer $n \in \NN$ such that the restriction of the associated Abel--Jacobi map
\[
\alpha_{C}=A_C\rest_{\Hilb^{cur}_{n,C}} : \Hilb_{n, C}^{cur} \longrightarrow 
\ol\Jac_C, 
\]
is surjective.
\end{proposition}

Both $\ol\Jac_C^{\, d}$ and $\Hilb_{n, C}^{cur}$ are fine moduli spaces, giving rise to the following universal objects. 

\begin{itemize}
    \item The sheaf $\calU_{d,C}$ on $C \times \ol\Jac_C^{\, d}$, universal for rank $1$ and degree $d$ torsion-free sheaves on $C$. The universal property is as follows: for every $\calF \in\ol \Jac_C^{\, d}$, there is a universal isomorphism $\calU_{d,C}|_{C\times\{\calF\}}\cong\calF$. We normalise $\calU_{d,C}$ at a fixed point $x_0 \in C$, i.e. $\calU_{d,C}|_{\{x_0\}\times\ol\Jac} \cong \calO_{\ol\Jac^d_C}$. For $d=0$, we simply write $\calU_C = \calU_{0,C}$. 

    \item The universal divisor $\calD_C \subset C\times\Hilb^{cur}_{n,C}$, universal for curvilinear subschemes. 

    \item Along (a restriction of) the projection $h_C:C\times \Hilb^{cur}_{n,C} \to \Hilb^{cur}_{n,C}$, which over $\calD$ is a finite map of degree $n$, we take the pushforward 
    \begin{equation}\label{eq def Acn}
    \calA_{C, n} := h_{C,*} \calO_{\calD_C} , 
    \end{equation}
    which is a locally free sheaf on $\Hilb_{n, C}^{cur}$ of rank $n$.

    \item The Poincar\'e line bundle $P_C \to \Jac_C \times \ol\Jac_C$, universal for topologically trivial line bundles on $\ol\Jac_C$. $P_C$ can be constructed by applying determinant of cohomology to $\calU_C$ \cite{esteves&kleiman}. 
\end{itemize}

\subsection{Fourier--Mukai transform of Arinkin} 
\label{subsec: FM and compactified Jac}
We now review Arinkin's construction of the \textit{Poincar\'e sheaf} $\calP_C$ on $\ol\Jac_C \times \ol\Jac_C$ as a universal extension of the Poincar\'e bundle $P_C$. We give a simplified construction that suffices for nodal curves \cite[§ 4.3.]{arinkin}. The simplification originates from the surjectivity of the Abel--Jacobi map stated in Proposition \ref{pr alpha curv surjective}. 

Consider the flag scheme, $\Flag_{n,C}$, parameterising length $n$ filtrations 
\[
\emptyset = D_0 \subsetneq \cdots \subsetneq D_k \subset \cdots \subsetneq D_n , 
\] 
of finite subschemes $D_i \in \Hilb_{i,C}$ in $C$. Let $\Flag^{cur}_{n, C}$ be the corresponding curvilinear flag scheme such that each $D_i$ is a point of $\Hilb^{cur}_{i, C}$. $\Hilb^{cur}_{n,C}$ and $\Flag^{cur}_{n, C}$ are open subschemes of $\Hilb_{n, C}$ and $\Flag_{n, C}$ respectively, as described by Arinkin \cite[Section 3]{arinkin}. Moreover, by Lemma 3.9 of \cite{melo2}, $\mathrm{codim}(\Hilb_{n,C}\setminus\Hilb^{cur}_{n,C})\geq 2$ in $\Hilb_{n,C}$. The forgetful morphism
\begin{equation*}\label{psiScurv}
    \psi_{C}: \Flag_{n, C}^{cur} \longrightarrow \Hilb_{n, C}^{cur} , \quad D_0 \subsetneq \cdots \subsetneq D_n \mapsto D_n,
\end{equation*}
is a finite flat morphism of degree $n!$ \cite[Proposition 3.5]{arinkin}. Moreover, we define
\begin{equation*}\label{sigmaScurv}
\sigma_{C}: \Flag_{n, C}^{cur} \longrightarrow C^n, \quad D_0 \subset \cdots \subset D_n \mapsto \supp( \ker(\calO_{D_i} \longrightarrow \calO_{D_{i-1}})_{i=1}^{n}) . 
\end{equation*}
The curvilinear flag scheme $\Flag_{n, C}^{cur}$ carries a natural action of the symmetric group $\sym_n$. Every divisor in $\AA^1$ is the zero locus of some polynomial $(z-t_1)\cdots(z-t_n)$. Hence, there is a natural $\sym_n$-action that permutes $t_i$. This induces a $\sym_n$-action on flags of curvilinear subschemes. With respect to this action $\psi_{C}$ is $\sym_n$-invariant and $\sigma_{C}$ is a $\sym_n$-equivariant morphism using the permutation action on $C^n$.

Consider the diagram
\begin{equation*}\label{eq main diagram}
\begin{tikzcd}[column sep = huge]
\Hilb_{n, C}^{cur} \times \ol \Jac_C^{\, d} \arrow[d , "\p_1" ]  & \Flag_{n, C}^{cur} \times \ol \Jac_C^{\, d} \arrow[l , "\psi_{C} \times \id_{\ol \Jac}"'] \arrow[r , "\sigma_{C} \times \id_{\ol \Jac}"] & C^n \times \ol \Jac_C^{\, d}
\\
\Hilb_{n, C}^{cur}
\end{tikzcd} .
\end{equation*} Further, denote by $\calU_{d, C}^{\boxtimes_n}\to C^n\times\ol\Jac^{\, d}_C$ the external tensor product with respect to the projections $C^n\times\ol\Jac_C\to C\times\ol\Jac_C$ onto each factor. We then define the Cohen--Macaulay sheaf $\calG_{d, C}$ on $\Hilb_{n, C}^{cur} \times \ol \Jac_C^{\, d}$ by
\begin{equation}\label{eq description of Gg_C}
\calG_{d, C} := \left ((\psi_{C} \times \id_{\ol \Jac})_{*}  (\sigma_{C} \times \id_{\ol \Jac})^{*} \calU_{d, C}^{\boxtimes_n} \right )^{sign} \otimes \p_1^*\det(\calA_{C,n})^\vee.
\end{equation}
Note that the pushforward sheaf under $\psi_C$ inherits a $\sym_n$-action on the level of germs of sections. The superscript ``sign'' denotes the anti-invariant subsheaf, where $\sym_n$ acts by the sign representation. 

The main result of Arinkin \cite{arinkin} is as following descent statement for $\calG_{0, C}$ along the Abel--Jacobi map. 
 
\begin{theorem}[\cite{arinkin}] 
\label{tm descent construction of Arinkin}
There exists a maximal Cohen--Macaulay sheaf $\calP_C$ over $\ol\Jac_C \times \ol\Jac_C$, called the \emph{Poincar\'e sheaf}, which restricts to the line bundle $P_C$ under the inclusion $\Jac_C \times \ol\Jac_C \cup \ol\Jac_C \times \Jac_C\longhookrightarrow \ol\Jac_C \times \ol\Jac_C$. The Poincaré sheaf $\calP_C$:
\begin{itemize}
    \item is flat over both factors and symmetric under permutation of these factors;
    \item satisfies
\begin{equation} \label{descent construction of Arinkin}
\calG_{0, C}  \cong (\alpha_{C} \times \id_{\ol \Jac})^* \calP_C\otimes \p_2^*N,
\end{equation}
for some line bundle $N$ on $\ol\Jac_C$, where $\p_2:\Hilb_{n,C}^{cur}\times \ol\Jac_C\to\ol\Jac_C$ is the projection;
\item is the universal family for the moduli problem $\ol{\Pic}^{\,0}(\ol\Jac_C)$ of topologically trivial rank one torsion-free sheaves on $\ol\Jac_C$.
\end{itemize}
\end{theorem}

\begin{remark}
The fact that $\calG_{0, C}$ and $(\alpha_{C} \times \id_{\ol \Jac})^* \calP_C$ differ by the pullback of a line bundle is clarified in Proposition 4.12, Remark 4.13 and Lemma 4.14 of \cite{melo2}. The twist has, however, no influence on Arinkin's results.
\end{remark}

Consider the natural projections
\begin{equation*}\label{eq:projections-compact-Jac}
\xymatrix{ & \ol\Jac_C \times \ol\Jac_C  \ar[ld]_{\pi_1} \ar[rd]^{\pi_2} &
\\
\ol\Jac_C & & \ol\Jac_C.
} 
\end{equation*}

The associated Fourier--Mukai transform with kernel $\calP_C$ is defined to be 
\begin{equation} \label{eq FM Arinkin}
\morph{D^b \left ( \ol\Jac_C \right )}{D^b \left ( \ol\Jac_C \right )}{\calF^\bullet}{R \pi_{2,*}(\pi_1^*\calF^\bullet \otimes \calP_C).}{}{\Phi^{\calP_C}_{1 \rightarrow 2}}
\end{equation}

Moreover, the dual sheaf $\calP^\vee_C$ defines the Fourier--Mukai transform 
\begin{equation} \label{eq dual FM Arinkin}
\morph{D^b \left ( \ol\Jac_C \right )}{D^b \left ( \ol\Jac_C \right )}{\calF^\bullet}{R \pi_{1,*}(\pi_2^*\calF^\bullet \otimes \calP^\vee_C).}{}{\Phi^{\calP^\vee_C}_{2 \rightarrow 1}}
\end{equation}

The following is a special case of \cite[Theorem C]{arinkin}.
\begin{theorem} 
\label{thm:derivedequiv}
Let $C$ be an integral projective curve with nodal singularities and arithmetic genus $g$. Then, $\Phi^{\calP_C}_{1 \rightarrow 2}$ is a derived equivalence, with quasi-inverse the $g$-shifted functor $\Phi^{\calP^\vee_C}_{2 \rightarrow 1}[g]$. 
\end{theorem}


\subsection{Varying the degree}
This section states a relation between the sheaves $\calG_{0,C}$ and $\calG_{d, C}$ that will be crucial in Section \ref{sc FM and normalisation}. We pass between the two sheaves via twists taken at the fixed smooth point $x_0 \in C$ over which the universal family $\calU_{d,C}$ is normalised. Since $C$ is irreducible, the choice of $x_0$ defines an isomorphism
\begin{equation}
\label{eq translation iso}
\morph{\ol \Jac_C^{\, d'}}{\ol \Jac_C^{\, d'+d}}{\calF}{\calF\otimes \calO_C(dx_0),}{}{\tau_{d, x_0}}
\end{equation}
with inverse $\tau_{-d, x_0}:\ol \Jac_C^{\, d'+d} \to \ol \Jac_C^{\, d'}$. The choice of normalisations on the universal sheaves $\calU_{0,C}\to C\times\ol\Jac_C$ and $\calU_{d,C}\to C\times\ol\Jac_C^{\, d}$ provides a universal equivalence 
\begin{equation}  
\label{eq relation of Universal sheaves with different ds}
\calU_{d, C} \cong (\id_C \times \tau_{-d, x_0})^* \calU_{0, C} \otimes p_C^*\calO_C(dx_0),
\end{equation}
with $p_C:C \times \ol\Jac_C\to C$ the projection to the first factor. 

Given a point $x \in C$, the line bundle $\calO_C(x)^{\boxtimes_n}=\calO_C(x) \boxtimes \dots \boxtimes \calO_C(x)$ over $C^n$ is $\sym_n$-equivariant under permutations in $C^n$, and since tensorization is invariant under permutations, it follows that $\sym_n$ acts trivially on its fibres over the $\sym_n$-fixed points. By a result of Drézet--Narasimhan \cite[Theorem 2.3]{drezet&narasimhan}, $\calO_C(x)^{\boxtimes_n}$ descends to a line bundle $L$ on the symmetric product $\Sym^{n}_C$; 
\begin{equation}\label{eq line bundle descends}
\calO_C(x)^{\boxtimes_n} \cong \pi_C^*L,
\end{equation}
with $\pi_C:C^n \longrightarrow \Sym^{n}_C$ the quotient map. To describe $L$, note that $\calO_C(x)$ admits a section vanishing at $x$, thus $L$ can be endowed with a section vanishing at any tuple containing $x$, i.e.~vanishing over the subvariety of $\Sym^{n}_C$ given by the image of
\[
\morph{\Sym_C^{n-1}}{\Sym^{n}_C}{\sum_{i=1}^{n-1}x_i}{x+\sum_{i=1}^{n-1}x_i.}{}{f_{C, x}}
\] 
Let 
\begin{equation}\label{eq Chow map}
    \morph{\Hilb_{n, C}^{cur}}{\Sym^{n}_C}{D}{\sum_{p\in C}\length(\calO_{D,p})p}{}{\chow_C}
\end{equation}
be the Chow morphism. On the universal family $\calD_C \subset C\times\Hilb^{cur}_{n,C}$ of length $n$ curvilinear subschemes we take the slice 
\[
\calD_{C,x}=\calD_C\cap\{x\}\times\Hilb_{n,C}^{cur} \, ,
\]
parameterising curvilinear finite subschemes containing a point $x \in C$. Then 
\[
\wt\calD_{C,x} := \chow_C(\calD_C\cap\{x\}\times\Hilb_{n,C}^{cur}),
\]
is a Cartier divisor in $\Sym_C^n$, and is equal to the image of $f_{C,x}$. Therefore, we may conclude that
\begin{equation}\label{eq:L=O(tilde D)}
    L\cong\calO_{\Sym}(\wt\calD_{C,x}),
\end{equation}
with $\Sym=\Sym^{n}_C$. From $\calD_{C, x} = \chow_C^{-1}\left ( \wt\calD_{C, x}\right )$ in $\Hilb_{n, C}^{cur}$ it follows that 
\begin{equation} \label{eq definition of Jj_C}
\calO_\Hilb(\calD_{C,x})\cong \chow_C^*\calO_{\Sym}(\wt\calD_{C,x}). 
\end{equation}
In particular, if $C$ is smooth, then $\chow_C$ is an isomorphism and we may identify $\calO_{\Sym}(\wt\calD_{C,x})=\calO_\Hilb(\calD_{C,x})$.

Let $p_C:C \times \ol\Jac_C \to C$ and $h_C:  C \times \Hilb_{n,C}^{cur} \to \Hilb_{n,C}^{cur}$ be the natural projections. The following universal relation on $C \times \Hilb_{n,C}^{cur}$ is well-known for smooth curves; see for example \cite[Proposition 9]{Schwarzenberger}. 

\begin{lemma}\label{lemm:Extend_Schwarz}
Let $C$ be an integral projective curve, either nodal or smooth. The universal sheaf $\calU_C$ and the dual of the ideal sheaf of the universal divisor $\calD_C \subset C \times \Hilb^{cur}_{n,C}$ are related by the isomorphism  
    \[ \calI_{\calD_{C}}^{\vee} \cong
    (\id\times \alpha_C)^* \left(\calU_C \otimes p_C^* \calO_{C}(nx_0) \right) \otimes h_C^*\calO_{\Hilb_C}(\calD_{C,x_0})
    \,. \] 
\end{lemma}
\begin{proof}
Denote $\calD_{C,x_0}=\calD_C \cap \{x_0\} \times \Hilb_{n,C}^{cur}$ and recall the Abel-Jacobi map $\alpha_C: \Hilb^{cur}_{n,C} \to \ol\Jac_C$ defined in \eqref{Abel-Jacobi}. Consider the rank 1 torsion-free sheaf 
\[
(\id\times \alpha_C)^* \left(\calU_C \otimes p_C^* \calO_{C}(nx_0) \right) , 
\]
defined on $C \times \Hilb_{n,C}^{cur}$. For fixed $D \in \Hilb^{cur}_{n,C}$, its restriction to $C \times \{D\}$ is given by 
\[
\calI^\vee_D(-nx_0) \otimes \calO_C(nx_0)= \calI_D^\vee . 
\]
Thus $\calI_{\calD_C}^\vee$ and $(\id\times \alpha_C)^* \left(\calU_C \otimes p_C^* \calO_{C}(nx_0) \right)$ are equivalent families of rank 1 torsion-free sheaves on $\Hilb_{n,C}^{cur}$. Therefore, they differ by the pullback of a line bundle $\calL$ on $\Hilb_{n,C}^{cur}$, i.e.~
    \[
    \calI^\vee_{\calD_C} \cong (\id\times \alpha_{C})^* \left(\calU_C \otimes p_C^* \calO_{C}(nx_0) \right) \otimes h_C^* \calL
    \] Restricting to $\{x_0\} \times \Hilb_{n,C}^{cur}$, we obtain
    \[
    \calO_{\Hilb_C}(\calD_{C,x_0}) \cong \alpha_C^* (\calU_C |_{\{x_0\} \times\ol\Jac_C}) \otimes \calL= \calL,
    \] where we used that the universal sheaf $\calU_C$ is normalised at $x_0$. This proves the formula.
\end{proof}

We now provide a relation between the sheaves $\calG_{d, C}$ and $\calG_{0,C}$, defined over $\Hilb_{n, C} \times \ol\Jac_C^{\, d}$ and $\Hilb_{n, C} \times \ol\Jac_C$ respectively, that will be used in Proposition \ref{pr fibrewise description of Poincare and normalisation}. Let $\calG^{cur}_{d, C}$ and $\calG^{cur}_{0,C}$ be the restrictions to $\Hilb_{n, C}^{cur} \times \ol\Jac_C^{d}$ and $\Hilb_{n, C}^{cur} \times \ol\Jac_C$. Let $\calO_\Hilb(\calD_{C,x_0})^d$ be the line bundle over $\Hilb_{n, C}^{cur}$ obtained by taking $d$-times the tensor product of the sheaf defined in \eqref{eq definition of Jj_C} over the point $x_0 \in C$. Let $\p_1:\Hilb_{n,C}^{cur}\times\ol\Jac_C\to\Hilb_{n,C}^{cur}$ be the projection and recall the tensorization isomorphism $\tau_{-d,x_0}:\ol\Jac_C^{\, d} \longrightarrow \ol\Jac_C$ from \eqref{eq translation iso}.

\begin{lemma} \label{lm Gg_0 and Gg_d}
In the above notation, there exists an isomorphism of sheaves 
\[
\calG_{d,C}^{cur} \cong \left ( \id_{\Hilb_C} \times \tau_{-d, x_0} \right )^* \left ( \calG_{0,C}^{cur} \otimes \p_1^* \calO_\Hilb(\calD_{C,x_0})^d \right ).
\]
In particular, $\calG_{d,C}^{cur}$ is a maximal Cohen--Macaulay sheaf.
\end{lemma}

\begin{proof}
Consider the diagram
\[
\begin{tikzcd}[column sep = huge]
\Hilb_{n, C}^{cur} \times \ol \Jac_C^{\, d} \arrow[d ,"\cong", "\id_{\Hilb_C} \times \tau_{-d, x_0}"']  & \Flag_{n, C}^{cur} \times \ol \Jac_C^{\, d} \arrow[d , "\id_{\Flag} \times \tau_{-d, x_0}", "\cong"' ] \arrow[l ,  "\psi_{C} \times \id_{\ol \Jac}"'] \arrow[r , "\sigma_{C} \times \id_{\ol \Jac}"] & C^n \times \ol \Jac_C^{\, d} \arrow[d , "\cong"',  "\id_{C^n} \times \tau_{-d, x_0}" ]
\\
\Hilb_{n, C}^{cur} \times \ol\Jac_C  & \Flag_{n, C}^{cur} \times \ol\Jac_C \arrow[l , "\psi_{C} \times \id_{\ol \Jac}"] \arrow[r , "\sigma_{C} \times \id_{\ol \Jac}"'] & C^n \times \ol\Jac_C . 
\end{tikzcd}
\]
By plugging \eqref{eq relation of Universal sheaves with different ds} into \eqref{eq description of Gg_C}, applying functoriality around the right square and base change around the left square,
one has an isomorphism 
\begin{align*}
\calG_{d,C}^{cur} \cong \left ( \id_{\Hilb_C} \times \tau_{-d, x_0} \right )^*\left ((\psi_{C} \times \id_{\ol \Jac})_{*} (\sigma_{C} \times \id_{\ol \Jac})^{*} \left (\calU_{0,C} \otimes p_C^*\calO_C(dx_0) \right)^{\boxtimes_n} \right )^{sign} 
\otimes \p_1^*\det(\calA_{C,n})^\vee.
\end{align*}
Here we made use of the fact that the pullback under the isomorphism $\id_{\Hilb_C} \times \tau_{-d, x_0}$ is invariant under the symmetric group action. Considering the commutative diagram 
\[
\begin{tikzcd}[column sep = huge]
C^n \times \ol\Jac_C \arrow[r, "\q_i \times \id_{\ol \Jac}"] \arrow[d , "\q_{C^n}"'] & C \times \ol\Jac_C \arrow[d, "p_C"]
\\
C^n \arrow[r, "\q_i"'] & C,
\end{tikzcd}
\]
and recalling from \eqref{eq line bundle descends} and \eqref{eq:L=O(tilde D)} that $\calO(dx_0)^{\boxtimes_n}$ is isomorphic to $\pi_C^*\calO_\Sym(\wt\calD^{cur}_{C,x_0})^d$, we obtain
\begin{align*}
\calG_{d,C}^{cur} \cong \left ( \id_{\Hilb_C} \times \tau_{-d, x_0} \right )^*\left ((\psi_{C} \times \id_{\ol \Jac})_{*}  (\sigma_{C} \times \id_{\ol \Jac})^{*} \left ( \calU_{0,C}^{\boxtimes_n} \otimes \q_{C^n}^* \pi_C^*\calO_\Sym(\wt\calD^{cur}_{C,x_0})^d  \right)  \right)^{sign} \\ 
\otimes \p_1^*\det(\calA_{C,n})^\vee.
\end{align*}
Next, we apply functoriality with respect to 
\[
\begin{tikzcd}[column sep = huge]
\Flag_{n, C}^{cur} \times \ol\Jac_C \arrow[d , "\q_{\Flag}"'] \arrow[r , "\sigma_{C} \times \id_{\ol \Jac}"] & C^n \times \ol\Jac_C \arrow[d , "\q_{C^n}" ]
\\
\Flag_{n, C}^{cur} \arrow[r , "\sigma_{C}"'] & C^n, 
\end{tikzcd} \qquad \begin{tikzcd}[column sep = huge]
\Flag_{n, C}^{cur} \arrow[d , "\psi_{C}"'] \arrow[r , "\sigma_C"] &  C^n \arrow[d , "\pi_C"]  
\\
\Hilb_{n, C}^{cur} \arrow[r, "\chow_C"'] &  \Sym^n_C,
\end{tikzcd}
\]
and 
\[
\begin{tikzcd}[column sep = huge]
\Hilb_{n, C}^{cur} \times \ol\Jac_C  \arrow[d , "\p_1"' ] &  \Flag_{n, C}^{cur} \times \ol\Jac_C \arrow[d , "\q_{\Flag}"] \arrow[l , "\psi_{C} \times \id_{\ol \Jac}"'] 
\\
\Hilb_{n, C}^{cur} &  \Flag_{n, C}^{cur} \arrow[l , "\psi_{C}"],
\end{tikzcd}
\]
to observe that
\begin{eqnarray*}
\calG_{d,C}^{cur} \cong \left ( \id_{\Hilb_C} \times \tau_{-d, x_0} \right )^*\left ((\psi_{C} \times \id_{\ol \Jac})_{*} \left ( (\sigma_{C} \times \id_{\ol \Jac})^{*}  \calU_{0,C}^{\boxtimes_n} \otimes  (\psi_{C} \times \id_{\ol \Jac})^*  \p_1^* \calO_\Hilb(\calD_{C,x_0})^d \right )  \right )^{sign}
\\
 \otimes \p_1^*\det(\calA_{C,n})^\vee,
\end{eqnarray*}
where we recall that $\calO_\Hilb(\calD_{C,x_0})^d\cong\chow_C^*\calO_\Sym(\wt\calD^{cur}_{C,x_0})^d$. The statement follows after applying projection formula with respect to $\psi_{C} \times \id_{\ol \Jac}$ followed by \eqref{eq description of Gg_C} for $d=0$.
\end{proof}

\section{Chow--Hilbert spaces and parabolic modules}
\label{se CH-PM-UD}

Over a nodal curve $C$ and a partial normalisation $\nu : \Sigma \to C$, this section establishes preliminary results on fine moduli spaces of sheaves and divisors. We introduce the moduli spaces of \textit{parabolic modules} associated to $\nu$ and discuss the interplay with the curvilinear Hilbert scheme over $C$. Throughout we refer to the following class of divisors related to the normalised nodes.  

\begin{definition} 
\label{de: sing divisors}
\, 
    \begin{itemize}
        \item The \emph{resolved singular locus} $\RSing(\nu) \subset \Sing(C)$ is the divisor consisting of nodes resolved by $\nu$. 

        \item The \emph{non-resolved singularities} are denoted by $\URSing(\nu):=\Sing(C) \setminus \RSing(\nu)$. 

        \item The \emph{exceptional divisor} $\Exc(\nu) \subset \Sigma$ is the reduced divisor $\nu^{-1}(\RSing(\nu))$.
    \end{itemize}
\end{definition}

It will also be convenient for us to fix an ordering $\Sing(C) = \{ b_1, \ldots, b_k, \dots, b_{n}\}$ on the nodes such that 
\[
\RSing(\nu) = b_1 + \dots + b_k, \quad \Exc(\nu) = b^{+}_1 + b^{-}_1 + \dots + b^{+}_k + b^{-}_k, 
\]
where, for each $i = 1, \ldots, k$, the divisor $b^{+}_i + b^{-}_i$ is given by $\nu^{-1}(b_i)$. 

\subsection{Chow--Hilbert spaces}
\label{sc chow-hilbert algebraic space}
In this section we introduce the so-called \textit{Chow--Hilbert space} $\Hyb_\nu$ associated to a partial normalisation $\nu : \Sigma \to C$. Roughly speaking, the Chow--Hilbert space resembles the Hilbert scheme away from $\RSing(\nu)$ and the symmetric product over $\RSing(\nu)$. This space will play an important role in our later analysis of partial normalisations. The construction of the space $\Hyb_\nu$ is as follows. Recall the Chow map $\chow_C : \Hilb_{n,C} \to \Sym^n_C$ and consider the subvariety $R_{\nu} \subset \Hilb_{n,C} \times_{\Sym_{C}^n} \Hilb_{n,C}$ given by pairs of ideals whose localization outside the resolved singular locus $\RSing(\nu)$ coincide: 
\begin{align}
\label{eq:rel_hybrid}
R_{\nu} = \{ (\calI_D,\calI_{D'}) \in \Hilb_{n,C} \times_{\Sym_{C}^n} \Hilb_{n,C} \, | \, \calI_{D,x} = \calI_{D',x} \textnormal{ for all } x \in C \setminus \RSing(\nu) \}. 
\end{align}
Observe that $\calI_{D,x} = \calI_{D',x}$ whenever $x$ lies outside the singular locus $\Sing(C)$, so the defining condition of $R_{\nu}$ needs to be imposed only on the finite set $\URSing(\nu)$. 

The subvariety $R_{\nu}$ naturally provides an equivalence relation on the Hilbert scheme $\Hilb_{n,C}$ given by $\calI_{D} \simeq \calI_{D'}$ if and only if $(\calI_{D}, \calI_{D'})$ is a point in $R_{\nu}$. By identifying $\Hilb_{n,C}$ with its sheaf of points over the \'etale site, we define the {\it Chow--Hilbert space associated to $\nu$} to be the quotient sheaf
\[
\Hyb_{\nu}:= \quotient{\Hilb_{n,C}}{\!\! R_{\nu}} .
\] 
The Chow-Hilbert space $\Hyb_{\nu}$ receives two natural maps: the quotient map  
\begin{align}\label{eq: quotient map to Hyb}
c: \Hilb_{n,C} \to \Hyb_{\nu} , 
\end{align} 
and a descent of the map $\chow_C : \Hilb_{n, C} \to \Sym_C^{n}$, invariant under $R_{\nu}$,  
\begin{equation}
\label{eq: chow descends}
\Hyb_{\nu} \to \Sym_C^{n} . 
\end{equation} 

\begin{lemma} 
\label{etale atlas}
$\Hyb_\nu$ is an algebraic space. 
\end{lemma}

\begin{proof} Consider the pair of Zariski-open subschemes $C_1 := C \setminus \mathsf{URSing(\nu)}$ and $C_2 := C \setminus \RSing(\nu)$. We define an atlas
\begin{equation}
\label{eq: simple etale atlas}
\bigcup_{k = 0}^{n} \Sym^k_{C_1} \times \Hilb_{n - k, C_2} \to \Hyb_{\nu} ,  
\end{equation}
where each chart is defined on $S$-points by  
\[
\Sym^k_{C_1}(S) \times \Hilb_{n-k, C_2}(S) \to \Hyb_{\nu}(S) , 
\]
\[
(D_1, D_2) \longmapsto [ D_1 \cup D_2 ] . 
\]
On geometric points, the atlas represents the possible decompositions of a  divisor $D \subset C$ into disjoint components $D = D_1 \sqcup D_2$ such that $D_1$ contains $D \cap \RSing(\nu)$ and $D_2$ contains $D \cap \URSing(\nu)$. The fibre at $[D] \in \Hyb_{\nu}(\Spec(\CC))$ is finite, of cardinality between $1$ and $2^n$; the lower bound is obtained when $[D]$ is a single point, and the upper bound when $[D]$ is supported on $n$ points in $C \setminus \Sing(C)$). Each chart is therefore \'etale and so \eqref{eq: simple etale atlas} defines an \'etale presentation on $\Hyb_{\nu}$. 
\end{proof}

\begin{remark} The sheaf $\Hyb_{\nu}$ also represents a complex analytic space, where in the analytic topology, one can separate the nodes of $C$ by pair-wise disjoint analytic open neighborhoods $D_{b} \subset C$, modeled on the standard open charts  $\{(x,y)\in \CC^2 \mid xy=0, \vert x \vert, \vert y \vert < \varepsilon \}$  One can then represent $\Hyb_{\nu}$ as the analytic space obtained as a gluing of analytic Hilbert schemes over the neighborhoods $D_{b}$. 
\end{remark}

\begin{remark}
When $\nu : \Sigma \to C$ is the full normalisation, $\Hyb_\nu$ is the symmetric product $\Sym_{C}^n$ and the quotient $\Hilb_C \to \Hyb_{\nu}$ coincides with the chow map in \eqref{eq Chow map}. 
\end{remark}

\subsection{Moduli of parabolic modules}  
\label{sec parabolic modules}

A key ingredient in our work will be the moduli of {\it parabolic modules} \cite{rego:1980, bhosle:1992, cook:1993,cook:1998, gothen&oliveira, FGOP}, used to intertwine sheaf data over $\Sigma$ and $C$. The existing literature concerning these objects is still somewhat limited, having been studied over an integral curve $C$ with only simple singularities (i.e. of type ADE). We specialise to the case where $C$ has simple nodal singularities (i.e. of type $A_1$).

A key reason for introducing parabolic modules is the observation that the map $\hat\nu : \Jac^{\, d}_{C} \to \Jac^{\, d}_{\Sigma}$ acting via the pullback $L \mapsto \nu^{*}L$ does not extend to the compactifications by rank one torsion-free sheaves. The moduli of parabolic modules, introduced by Bhosle and Cook \cite{bhosle:1992,cook:1993,cook:1998}, defines a compactification of $\Jac_C^d$ where $\hat{\nu}$ does extend. We define a particular case of parabolic modules living over the exceptional divisor $\Exc(\nu) = b^{\pm}_1 + \cdots + b_k^{\pm} \subset \Sigma$ consisting of nodes resolved by a fixed partial normalisation $\nu : \Sigma \to C$.  

\begin{definition}
\label{def par mod}
Let $\nu : \Sigma \to C$ be as above. A rank $1$ and degree $d$ \emph{parabolic module} for $\nu$ is a pair $(M,V)$ where $M$ is a degree $d$ rank $1$ torsion-free sheaf on $\Sigma$ and $V$ is a subsheaf of $M\otimes\calO_{\Exc(\nu)}$ such that
\[
V=V_1\oplus\cdots\oplus V_k,
\]
with $V_i$ a $1$-dimensional vector subspace of $M \otimes \calO_{\{ b_i^{+}, b_i^{-} \}} \cong M_{b_i^+}\oplus M_{b_i^-}$, where $M_{b_i^\pm} := M\otimes \calO_{b_i^\pm}$.

\end{definition}

\begin{remark}\mbox{}
\begin{enumerate}
    \item  The more general notion of parabolic modules associated to other kinds of singularities is defined by Cook \cite{cook:1993}, where the subspaces $V_i$ are allowed to be higher dimensional. In this situation, one needs to impose that $V_i$ is an $\calO_{C,b_i}$-submodule of $M\otimes\calO_{\{ b_i^{+}, b_i^{-} \}}$ via pushforward under $\nu$ i.e.~via the inclusion $\calO_{C}\longhookrightarrow\nu_*\calO_{\Sigma}$.
    
    \item In the case where $C$ has only $k=1$ node these objects were first considered by Bhosle in \cite{bhosle:1992}, who named them `generalised parabolic bundles'. The case of various simple nodes is an easy generalization of Bhosle's definition, so our setup fits that of \cite{bhosle:1992}.    
    \end{enumerate}
\end{remark}

Denote by $\PMod_\nu^{\, d}$ the fine moduli space of rank $1$ and degree $d$ parabolic modules for $\nu : \Sigma \to C$. We refer again to \cite{cook:1993,cook:1998,bhosle:1992} for its construction. 

By \cite[Theorem 1]{cook:1998} or \cite[Theorem 3]{bhosle:1992}, there is a finite morphism
\begin{equation} \label{eq definition of rho}
\morph{\PMod_\nu^{\, d}}{\ol \Jac_C^{\, d}}{(M,V)}{\calF}{}{\rho}  , 
\end{equation}
where $\calF$ is defined by the short exact sequence
\begin{equation} \label{eq preimage of tau}
0\longrightarrow\calF\longrightarrow \nu_*M \longrightarrow \nu_*\big(M\otimes\calO_{\Exc(\nu)}/V\big)\longrightarrow 0,
\end{equation}
and the surjection is defined by the composition $\nu_*M\to\nu_*\big(M\otimes\calO_{\Exc(\nu)}\big)\to\nu_*\big(M\otimes\calO_{\Exc(\nu)}/V\big)$. Note that, for each $i$, $\dim(\nu_*(M \otimes \calO_{\{ b_i^{+}, b_i^{-} \}}/V_i)) = 1$, hence 
\begin{equation*} \label{eq sing of C from par mods}
\nu_*\big(M\otimes\calO_{\Exc(\nu)}/V\big) \cong \calO_{\RSing(\nu)}.
\end{equation*}
Since the quotient $\nu_*(M\otimes\calO_{\Exc(\nu)}/V)$ is an $\calO_{C}$-module and $\nu_*M\to\nu_*\big(M\otimes\calO_{\Exc(\nu)}/V\big)$ is a morphism of $\calO_C$-modules, $\calF$ inherits an $\calO_{C}$-module structure too. In addition, $\deg(\nu_*M)=d+k$ and $\dim\big(\nu_*\big(M\otimes\calO_{\Exc(\nu)}/V\big)\big)=k$, thus indeed $\deg(\calF)=d$. 

Let $\rho_0$ denote the restriction of $\rho$ to $\rho^{-1}(\Jac_C^d)$. By \cite[Theorem 1]{cook:1998} or \cite[Theorem 3]{bhosle:1992}, we have an isomorphism 
\begin{equation}\label{eq:JasobianinParMod}
\rho_0:\rho^{-1}(\Jac_C^d)\xrightarrow{\ \cong\ }\Jac_C^d,
\end{equation}
and so $\Jac_C^d$ can be seen as a dense open subspace of $\PMod_{\nu}^{\, d}$. Moreover, $\PMod_{\nu}^{\, d}$ is proper and therefore a compactification of $\Jac_C^d$. It is however different from the compactification $\ol{\Jac}_C^{\, d}$, for it admits a natural projection
\begin{equation}
\label{eq proj PMod onto Jac}
    \morph{\PMod_{\nu}^{\, d}}{\ol\Jac_\Sigma^{\,d}}{(M,V)}{M,}{}{\dot{\nu}} 
\end{equation}
extending the pullback morphism $\hat{\nu}$, via the isomorphism \eqref{eq:JasobianinParMod}. 

We then have the following diagram summarising the previous discussion:
\[  \begin{tikzcd}[ampersand replacement=\&]
\&\rho^{-1}(\Jac^d_{C})\arrow[r,hook]\arrow[ld,"\cong"']\&\PMod^d_{\nu} \arrow[ld,"\rho"'] \arrow[rd,"\dot\nu"] \&\\
\Jac^d_{C}\arrow[r,hook]\&\ol\Jac^{\, d}_{C} \& \&  \ol\Jac^{\,d}_{\Sigma}.
\end{tikzcd}\]

The geometry of $\PMod_{\nu}^{\, d}$ and $\rho$ is particularly simple in our case and appears in \cite{bhosle:1992} in the case of a curve with a single node. We provide a proof of the following result for completeness.
Recall from Definition \ref{de: sing divisors} that we have attached the divisors
\[
\RSing(\nu) = b_1 + \dots + b_k, \qquad \Exc(\nu)=b_1^++b_1^-+\cdots+b_k^++b_k^-, \qquad \nu^{-1}(b_i)= b_i^{+} + b_i^- , 
\]
over $C$ and $\Sigma$, associated to the partial normalisation $\nu : \Sigma \to C$.
\begin{lemma} 
\label{lm tau surjective}
\,
\begin{enumerate} 
    \item The morphism $\dot\nu : \PMod_{\nu}^{\, d} \to \ol\Jac_\Sigma^{\,d}$ from \eqref{eq proj PMod onto Jac} is a $(\PP^1)^k$-bundle. 

    \item The morphism $\rho:\PMod_{\nu}^{\, d}\to\ol\Jac_C^{\, d}$ from \eqref{eq definition of rho} is a partial resolution of singularities, i.e.~it is a proper birational morphism which is an isomorphism on an open  dense subset of $\PMod_{\nu}^{\, d}$, containing $\rho^{-1}(\Jac_C^d)$.

    \item Two points of $\PMod_{\nu}^{\, d}$ have the same image
\[
\rho (M, V) = \rho  (M',V') ,
\]
if and only if
\begin{equation} \label{eq condition for J}
M' = M \otimes J,
\end{equation}
where $J$ is the line bundle
\[
J = \calO_\Sigma(b_1^- - b_1^+)^{a_1} \otimes \dots \otimes \calO_\Sigma(b_k^- - b_k^+)^{a_k},
\]
with $a_i \in \{ -1, 0, 1 \}$, such that for those $l$ with $a_l = 1$, those $m$ with $a_m = -1$ and those $n$ with $a_n = 0$, we have
\begin{equation} \label{eq condition for V_i}
V_l = M \otimes \calO_{b_l^{+}} \quad \text{and} \quad V'_l = M' \otimes \calO_{b_l^{-}}, 
\end{equation}
\begin{equation} \label{eq condition for V_j}
V_m = M \otimes \calO_{b_m^{-}} \quad \text{and} \quad V'_m = M' \otimes \calO_{b_m^{+}}
\end{equation}
and 
\begin{equation} \label{eq condition for V_k}
V_n = V'_n.
\end{equation}

\item If $\calF \in \ol\Jac_C^{\, d}$ is not locally free at $\ell\in\{0,\ldots,k\}$ points of $\RSing(\nu)$, then $\rho^{-1}(\calF)\cong\ZZ_2^\ell$.
\end{enumerate}

\end{lemma} 

\begin{proof}
(1) Follows directly from the definition of a parabolic module. 

For (2), the surjectivity of $\rho$ can be proven as in \cite[Lemma 5.9]{FGOP}. Moreover, as $\rho$ is an isomorphism on the smooth locus $\Jac_C \subset \ol\Jac_C$, any two different points $(M,V)$ and $(M',V')$ with the same image under $\rho$ will be mapped to the singular locus of $\ol\Jac_C$. Hence, $\rho$ is a partial resolution of singularities.
	
For (3), let us first consider that the stalk of $\calF := \rho(M,V)$ at $b \in \Sing(C)$ is the kernel of a map $\calO_{\Sigma, b^+} \oplus \calO_{\Sigma, b^-} \longrightarrow \calO_b$. If $\calF_b$ is not locally-free at $b$, then 
\begin{equation} \label{eq description of Ff_x 1}
\calF_{b} \cong \calO_{\Sigma, b^+} \oplus \calO_{\Sigma}(-b^-)_{b^-},
\end{equation}
if $V$ is locally $\calO_{\Sigma,b^+}$, or
\begin{equation} \label{eq description of Ff_x 2}
\calF_{b} \cong \calO_{\Sigma}(-b^+)_{b^+} \oplus \calO_{\Sigma, b^-},
\end{equation}
when $V$ is given by $\calO_{\Sigma,b^-}$. The previous distinction is redundant since one can build the isomorphism of $\calO_{C,b}$-modules between \eqref{eq description of Ff_x 1} and \eqref{eq description of Ff_x 2}, by means of a meromorphic section of $\calO_\Sigma(b^--b^+)$ with divisor $b^--b^+$. 
	
It then follows that, given $(M,V)$, $(M',V')$ and $J$ satisfying \eqref{eq condition for J}--\eqref{eq condition for V_k}, one obtains an isomorphism between the sheaves $\calF = \rho(M,V) = \ker \left ( \nu_*M \to \nu_*\big(M\otimes\calO_{\Exc(\nu)}/V\big) \right )$ and $\calF':= \rho(M',V') = \ker \left ( \nu_*M' \to \nu_*\big(M'\otimes\calO_{\Exc(\nu)}/V'\big)\right )$ by means of a meromorphic section of $J$ having simple zeros on the $b_i^-$'s and on the $b_j^+$'s, and having simple poles on the $b_i^+$'s and on the $b_j^-$'s. 
	
Assuming that $\calF \cong \calF'$, one has, for every point $b_i \in \RSing(\nu)$, that $\calF$ and $\calF'$ are either locally-free or their stalks are described as \eqref{eq description of Ff_x 1} or as \eqref{eq description of Ff_x 2}. If both have the same local description we set $a_i = 0$. If not, we set $a_i = 1$ if we need an isomorphism given by a meromorphic section vanishing on $b_i^+$ or $a_i = -1$ in the remaining case. Repeating this construction for every point of $\RSing(\nu)$, we obtain $J$ satisfying \eqref{eq condition for J} and the conditions \eqref{eq condition for V_i}, \eqref{eq condition for V_j} and \eqref{eq condition for V_k} hold. 
	
Finally for (4), given a sheaf $\calF \in \ol\Jac_C^{\, d}$ which is not locally free at $\{ b_{j_1}, \dots, b_{j_\ell}\}\subset\RSing(\nu)$, one immediately sees that its preimages under $\rho$ are of the form $(M,V)$ with $V_{j_i}$ of the form $M \otimes \calO_{b_{j_i}^+}$ or $M \otimes \calO_{b_{j_i}^-}$. The statement follows from this observation.
\end{proof}

We now study the universal families associated to the fine moduli spaces that have been previously introduced. Fix a smooth geometric point $x_0 \in C$ and the unique preimage $y_0 = \nu^{-1}(x_0) \in \Sigma$. Let
\[ 
\calU_{\Sigma} \to \Sigma \times \ol\Jac_\Sigma^{\,d},
\]
be the universal sheaf for $\ol\Jac_\Sigma^{\,d}$, normalised at $y_0$. Thus $\calU_{\Sigma}$ is the universal family such that  
\[
\calU_{\Sigma}|_{\{y_0\}\times \ol\Jac} \cong \calO_{\ol\Jac} .
\]
We denote the restriction of the universal bundle to a slice by $\calU_{\Sigma,y} := \calU_{\Sigma}|_{\{y\}\times \ol\Jac}$,
for any point $y\in\Sigma$. Whenever $y$ is non-singular, $\calU_{\Sigma,y}$ is a line bundle. In particular, this is true for $y \in \Exc(\nu)$.

For each of the resolved singular points $\{ b_1, \dots, b_k\} = \RSing(\nu)$, consider the rank two vector bundle 
\[ 
\dot \nu^*\calU_{\Sigma, b_i^+} \oplus \dot \nu^*\calU_{\Sigma, b_i^-} , 
\] 
over $\PMod_{\nu}^d$, where we recall $\dot\nu : \PMod_{\nu}^d \to \ol\Jac_\Sigma^{\,d}$ is the forgetful map \eqref{eq proj PMod onto Jac}. 
Then, consider the relative tautological line bundle $\calV_{\Sigma, i} \to \PMod_{\nu}^d$, that is, the line subbundle
\begin{equation} \label{eq definition of Vv_i}
\calV_{\Sigma, i} \subset \dot \nu^*\calU_{\Sigma, b_i^+} \oplus \dot \nu^*\calU_{\Sigma, b_i^-},
\end{equation}
whose fibre over the point $(M, V_1\oplus \dots\oplus V_k)\in\PMod_{\nu}^d$ is $V_i$, where we recall from the definition of parabolic modules that $V_i$ is a 1-dimensional subspace of $M_{b_i^-}\oplus M_{b_i^+}$. The restriction of this tautological bundle to the preimages $\dot\nu^{-1}(M) \subset \PMod_{\nu}^{d}$ is given by  
\[
\calV_{\Sigma, i}|_{\dot{\nu}^{-1}(M)} \cong \pi_i^*\calO_{\PP^1}(-1),
\] 
with $\pi_i : (\PP^1)^k \to \PP^1$ the projection to the $i$-th factor.

The universal family of parabolic modules is the tuple $\left((\id \times \dot{\nu})^*\calU_\Sigma, \calV_{\Sigma} \right)$, where $\calV_{\Sigma}$ is the vector bundle on $\PMod_{\nu}^d$ given by
\begin{equation}\label{eq:def V}
    \calV_{\Sigma}=\calV_{\Sigma, 1} \oplus \dots \oplus \calV_{\Sigma, k},
\end{equation}
which comes naturally equipped with the inclusion
    \[ \calV_{\Sigma}=\calV_{\Sigma, 1} \oplus \dots \oplus \calV_{\Sigma, k} \subset \dot{\nu}^*\calU_\Sigma|_{\nu^{-1}(\RSing(\nu)) \times \PMod}\cong\bigoplus\limits_{i=1}^k \dot \nu^*\calU_{\Sigma,b_i^+}\oplus \dot \nu^*\calU_{\Sigma,b_i^-}\,.
    \]
    Due to the compatible normalisations of $\calU_C$ and $\calU_\Sigma$, we obtain a universal version of the exact sequence \eqref{eq preimage of tau} on $C \times \PMod_{\nu}^d$, given by 
    \begin{equation}\label{eq:univ-ses}
    0 \to (\id \times \rho)^{*}\calU_C 
    \to (\nu \times \id)_*(\id \times \dot\nu)^* \calU_\Sigma 
    \to \calT
    \to 0, 
    \end{equation} 
    where 
    \begin{equation}\label{eq:calT}
        \calT:=\big((\nu\times\id)_*\big(\id \times\dot{\nu})^*\calU_\Sigma|_{\nu^{-1}(\Sing(C)) \times \PMod}\big)/\calV_{\Sigma} 
    \end{equation}
    is supported on $\Sing(C)\times\PMod_{\nu}^d$. The pushforward of $\calT$ along the projection $C \times \PMod_{\nu}^d \to \PMod_{\nu}^d$ recovers the universal quotient bundle
    \begin{equation}
         \label{equ:univ_quotient}\calQ:=\calQ_1\oplus\cdots\oplus\calQ_k:=\bigoplus_{i=1}^k ( \dot \nu^*\calU_{\Sigma, b_i^+} \oplus \dot \nu^*\calU_{\Sigma, b_i^-})/\,\calV_{\Sigma, i}
\end{equation}  
on $\PMod_{\nu}^d$.

We conclude this section by describing the canonical bundle $\omega_{\PMod}$ of $\PMod_{\nu}^d$. 
\begin{lemma} 
\label{lm description of omega_PMod}
Let $C$ be an integral nodal projective curve with partial normalisation $\nu : \Sigma \to C$. Let $\RSing(\nu) = \{ b_1, \dots, b_k\}$ be the divisor of singularities resolved by $\nu$. Then, there exists an isomorphism of line bundles \begin{equation}\label{eq: canonical PMod}
\omega_{\PMod} \cong \bigotimes\limits_{i=1}^k \calV_{\Sigma, i}^2 \otimes \dot\nu^*\calU_{d, b_i^+}^\vee \otimes \dot \nu^*\calU_{d, b_i^-}^\vee
\end{equation}
\end{lemma}
\begin{proof} The vector bundle $U_i := \calU_{d, b_i^+} \oplus \calU_{d, b_i^-}$ gives rise to the projective bundle 
\[
p: \PP(U_i)=\PP\left(\calU_{d, b_i^+} \oplus \calU_{d, b_i^-}\right) \to \ol\Jac_\Sigma^{\,d}. 
\] 
By a standard formula (see \cite[Ex III, 8.4]{hartshorne}) the relative dualising sheaf of the bundle map $p$ is  
\[ 
\omega_{\PP(U_i)/\ol\Jac_\Sigma^{\,d}}\cong\calO_p(-2) \otimes p^*\det(U_i)^\vee=\calO_p(-2)\otimes p^*\calU_{d, b_i^+}^\vee \otimes  p^*\calU_{d, b_i^-}^\vee,
\]
where $\calO_p(-2)$ is the square of the relative tautological line bundle. $\PMod_{\nu}^d$ is the $(\PP^1)^k$-bundle given by the fibre product 
\[
\PP(U_1) \times_{\ol\Jac_\Sigma^{\,d}}\cdots \times_{\ol\Jac_\Sigma^{\,d}} \PP(U_k) \to \ol\Jac^d_\Sigma
\]
It follows that the relative dualising sheaf $\omega_{\PMod/\ol\Jac_\Sigma^{\,d}}$ is isomorphic to \eqref{eq: canonical PMod}. Moreover $\ol\Jac_\Sigma^{\,d}$ has trivial canonical bundle, and so we have $
\omega_{\PMod} \cong \omega_{\PMod/\ol\Jac_\Sigma^{\,d}}$ concluding the proof. 
\end{proof} 

We now take square roots of the expression in Lemma \ref{lm description of omega_PMod}, which requires the following. 

\begin{lemma} For each $y \in \Sigma\setminus\Sing(\Sigma)$, the restriction $\calU_{\Sigma, y} = \calU_{\Sigma} |_{\{y\} \times \ol\Jac_\Sigma}$ admits a square root. Moreover, the choice of square root is parameterised by $H^1(\Sigma, \ZZ_2)$.
\end{lemma}
\begin{proof}
The compatibility $A_\Sigma^*\calP_\Sigma \cong \calU_\Sigma$ between the Poincaré sheaf with the Abel-Jacobi map shows that $\calU_{\Sigma,y} \cong \calP_\Sigma |_{\{A_\Sigma(y)\}\times \ol\Jac_{\Sigma}}$ is a geometric point of $\Pic^{0}(\ol\Jac_{\Sigma})$. The autoduality $\Pic^{0}(\ol\Jac_{\Sigma}) \cong \Jac_\Sigma$ preserves the group structure, so it is equivalent to describe the square roots of the corresponding geometric point in $\Jac_\Sigma$. These are parametrised by the 2-torsion points of $ \Jac^0_{\Sigma}$, which in turn correspond to $H^1(\Sigma, \ZZ_2)$. 
\end{proof}

\begin{corollary}
\label{spin structures}
Adopt the same hypothesis as Lemma \ref{lm description of omega_PMod}. At every point $b_i^{\pm} \in \Exc(\nu) = \nu^{-1}(\RSing(\nu))$, choose a square root $\calU^{1/2}_{\Sigma, b_i^{\pm}}$ of the line bundle $\calU_{\Sigma, b_i^{\pm}}$. Each such family of choices $\big\{ \calU^{1/2}_{\Sigma, b_i^{\pm}}\}_{i = 1,\ldots,k}$ determines a spin structure on $\PMod_{\nu}^d$ given by 
\[
\omega_{\PMod}^{1/2} :=  \bigotimes\limits_{i=1}^k \calV_{\Sigma, i} \otimes \dot\nu^*\calU^{-1/2}_{b_i^-} \otimes \dot \nu^*\calU^{-1/2}_{b_i^+} .
\]
\end{corollary}

\section{Comparison results for curvilinear Hilbert schemes and parabolic modules}
\label{se: comparison}

\subsection{Parabolic modules via curvilinear divisors}
\label{se para and curv}
In this section we present some technical results concerning the interplay between the curvilinear Hilbert scheme of $C$ and the moduli space of parabolic modules over $\nu : \Sigma \to C$. The results described are preliminaries for Section \ref{sc relation Poincares}. As before, we take $C$ to be an irreducible nodal projective curve with simple nodes and denote by $\nu : \Sigma \to C$ its partial normalisation at the nodes $\{b_1, \dots, b_k\}=\RSing(\nu)\subset\Sing(C)\subset C$. Recall that we have chosen a smooth point $x_0\in C$ with preimage $y_0 = \nu^{-1}(x_0) \in \Sigma$. We specialise to the degree zero case $d=0$.

\begin{notation}\label{not2}
As in Notation \ref{not1}, when $d=0$, we will write $\PMod_\nu:=\PMod_\nu^0$.
\end{notation}

Recall the Chow-Hilbert space introduced in Section \ref{sc chow-hilbert algebraic space}. The natural morphism 
\[
\chow(\nu): \Sym_{\Sigma\setminus{\nu^{-1}(\URSing(\nu))}}^{n} \to \Sym_{C\setminus{\URSing(\nu)}}^n
\]
and the identity isomorphism $\Hilb^n_{\Sigma\setminus{\Exc(\nu)}} \to \Hilb_{C\setminus{\RSing}}^n$, glue to a morphism 
\[ \nu^{(n)}: \Hilb^{cur}_{n,\Sigma} \to \Hyb_\nu\,.
\] Taking the cartesian square above $\nu^{(n)}$ and the quotient map $c$ from \eqref{eq: quotient map to Hyb} we obtain  
\begin{equation} \label{eq cartesian diagram Hilb Sym}
\begin{tikzcd}[column sep = huge]
\Hilb_{n,C}^{cur} \times_{\Hyb_{\nu}} \Hilb_{n,\Sigma}^{cur} \arrow[d, "\beta"'] \arrow[r, "q"] & \Hilb_{n,\Sigma}^{cur} \arrow[d, "\nu^{(n)}"]
\\
\Hilb_{n,C}^{cur}  \arrow[r, "c"'] & \Hyb_\nu\,.
\end{tikzcd}
\end{equation}
Note the pullback is a scheme as the Chow--Hilbert space $\Hyb_{\nu}$ is an algebraic space (Lemma \ref{etale atlas}). Recall also the finite morphism $\rho$ of \eqref{eq definition of rho} and consider also the Cartesian diagram
\begin{equation}\label{eq cartesian diagram Hilb PMod}
\begin{tikzcd}[column sep = huge]
\Hilb_{n,C}^{cur} \times_{\ol\Jac_C} \PMod_{\nu} \arrow[d, "\wt \rho \,"'] \arrow[r, "\wt{\alpha}_C"] & \PMod_{\nu} \arrow[d, "\rho"]
\\
\Hilb_{n,C}^{cur} \arrow[r, "\alpha_C"'] & \ol\Jac_C.
\end{tikzcd}  
\end{equation}
Proposition \ref{pr alpha curv surjective} allowed us to choose $n$ so that $\alpha_C:\Hilb_{n,C}^{cur}\to\ol\Jac_C$ is surjective. However, by \cite{dsouza,altman&kleiman}, for large enough $n$, we know that $A_C:\Hilb_{n,C}\to\ol\Jac_C$, is actually a projective bundle (see page 425 and Note 4.7 of \cite{dsouza} and  \cite[Theorem 8.6]{altman&kleiman}). We assume such choice of $n$ from now on:

\begin{assumption}\label{assumption}
Take the number $n\in\NN$ to be sufficiently large such that $A_C:\Hilb_{n,C}\to\ol\Jac_C$ is a projective bundle over $\ol\Jac_C$, with fibre over $\calF$ equal to $\PP(H^0(C,\calF(nx_0)))$.
\end{assumption}

Under this assumption, $\alpha_C:\Hilb_{n,C}^{cur}\to\ol\Jac_C$ is a smooth surjective morphism, such that the complement of its fibres in the (projective) fibres of $A_C:\Hilb_{n,C}\to\ol\Jac_C$ is of codimension at least two, by \cite[Lemma 3.9]{melo2}. 
In particular, the base change $\wt{\alpha}_C$ in \eqref{eq cartesian diagram Hilb PMod} is also smooth and surjective.

Let $(D,(M,V))\in \Hilb_{n,C}^{cur} \times_{\ol \Jac_C} \PMod_{\nu}$ and let 
\[\calF:=\alpha_C(D) = \rho(M,V).\] 
By the definition of $\alpha_C$, the subscheme $D\subset C$ determines a section $s_D : \calO_C \to \calI_D^\vee=\calF(nx_0)$. On the other hand, $(M,V)$ corresponds under $\rho$ to an embedding $\jmath: \calF \longhookrightarrow \nu_*M$. Hence, the composition $\jmath\circ s_D$ is a section of the sheaf $(\nu_*M)(nx_0) \cong \nu_*(M(ny_0))$. By adjointness, this gives a section of the rank $1$ torsion-free sheaf $M(ny_0)$ over $\Sigma$. Its vanishing locus determines, in turn, a degree $n$ effective divisor $E\in\alpha_\Sigma^{-1}(M) \subset \Hilb_{n,\Sigma}^{cur}$, where 
\[
\alpha_\Sigma:\Hilb_{n,\Sigma}^{cur}\to\ol\Jac_\Sigma,\ \ \ \alpha_\Sigma(\calI_E)=\calI_E^\vee(-ny_0) 
\] 
is also smooth and projective, whenever Assumption \ref{assumption} holds. As $\jmath$ vanishes nowhere (because the torsion in \eqref{eq preimage of tau} is of rank $1$), the vanishing locus of $\jmath \circ s_D$ is still $D$. Then, again by adjointness, we have $\nu(E) \subset D$, so $c (D) = \nu^{(n)} (E)$ as both have the same length. 

The upshot of the preceding discussion is that, taking into account the diagrams \eqref{eq cartesian diagram Hilb Sym} and \eqref{eq cartesian diagram Hilb PMod}, the morphism 
\begin{equation} \label{eq def j}
j : \Hilb_{n,C}^{cur} \times_{\ol\Jac_C} \PMod_{\nu} \to \Hilb_{n,C}^{cur} \times_{\Hyb_\nu} \Hilb_{n,\Sigma}^{cur},\ \ \ \ j(D,(M,V))=(D,E)
\end{equation}
is well-defined. Here $E$ is constructed as above. To describe this map in more detail, we require the following description of points on $\Hilb^{cur}_{n,C}$.

\begin{lemma} \label{lm description of Hilb^cur}
Let $C$ be an irreducible nodal projective curve with $k$ nodes at $\{b_1, \dots, b_k\}=\RSing(\nu)\subset C$, with $\nu : \Sigma \to C$ the partial normalisation at $\RSing(\nu)$. Decompose the subscheme $D \in \Hilb_{n,C}^{cur}$ as 
\begin{equation} \label{eq decomposition of D}
D = D_0 + D_1 + \dots + D_k,
\end{equation}
with $D_0$ being supported on a subset of $C\setminus\RSing(\nu)$ and $D_i$ supported on the nodal point $b_i$. 

For $i\geq 1$, if $D_i$ is non-empty, then one of the following two options holds:
\begin{enumerate}
\item $D_i$ 
is supported on one of the branches of $C$ at $b_i$, in which case $\calI_{D_i}$ is not principal. 
\item $D_i$ is $D_{b_i,a_i}$ for $a_i \in \CC^*$, where
\[
D_{b_i,a_i} := \Spec \left ( \CC[y_i^+, y_i^-]/\langle y_i^+y_i^-, y_i^+ - a_i y_i^- \rangle \right ),
\]
with $y_i^+$ and $y_i^-$ being local coordinates of $C$ at the node $b_i$. In this case, $D_i$ has length $2$ and $\calI_{D_i}$ is principal.  
\end{enumerate}
\end{lemma}

\begin{proof}
Knowing that $C$ is locally isomorphic to $\Spec \left ( \CC[y_i^+, y_i^-]/\langle  y_i^+ y_i^-\rangle \right )$, the proof becomes a simple exercise. If $D_i$ is not contained in $\Spec \left ( \CC[y_i^\pm] \right )$ it must have length at least $2$. As, by hypothesis, it is contained in some $\AA^1$, then $D_i \subset D_{b_i, a_i}$ for some $a_i \in \CC^*$. Observe that the right-hand-side is a length $2$ subscheme, hence the latter is identified with $D_i$ itself. It follows from this identification that $\calI_{D_i}$ is principal. 
\end{proof}

We now describe points $(\calI_D,E) \in \Hilb_{n,C}^{cur} \times_{\Hyb_\nu} \Hilb_{n,\Sigma}^{cur}$ that lie in the image of $j$. From the description prior to \eqref{eq def j}, such points satisfy the condition $\nu(E) \subset D$. The following is immediate from combining this condition with the description of curvilinear divisors in Lemma \ref{lm description of Hilb^cur}. 

\begin{corollary} \label{co description of Hilb^cur times PMod} Let $C$ be an irreducible nodal curve and $\nu : \Sigma \to C$ a partial normalisation at $\RSing(\nu) = \{b_1, \dots, b_k\}$, with $\Exc(\nu) =\{ b_1^+, b_1^-, \dots, b_k^+, b_k^-\}$. Take $(\calI_D, \calI_E) \in \Hilb_{n,C}^{cur} \times_{\Hyb_\nu} \Hilb_{n,\Sigma}^{cur}$ so that $\nu(E) \subset D$. Considering the decomposition \eqref{eq decomposition of D} of $D$, then $E$ decomposes as
\begin{equation} \label{eq decomposition of E}
E = E_0 + E_1^+ + E_1^- + \dots + E_k^+ + E_k^-,
\end{equation}
with $E_0$ being supported on a subset of $\Sigma\setminus\Exc(\nu)$, $E_i^+$ supported at $b_i^+$ and $E_i^-$ at $b_i^-$.

Then $D_0 = \nu (E_0)$ and, for $i\geq 1$,
\begin{enumerate}
\item if $D_i$ is $0$ then $E_i^+ = 0$ and $E_i^- = 0$; 
\item if $D_i$ 
is supported on one of the branches of $C$ at $b_i$, then either $D_i = \nu(E_i^+)$ and $E_i^- = 0$, or $D_i = \nu(E_i^-)$ and $E_i^+ = 0$;
\item if $D_i$ is $D_{b_i,a_i}$ for $a_i \in \CC^*$, then $E_i^+ = b_i^+$ and $E_i^- = b_i^-$, hence $\nu(E_i) = b_i \subset D_i$.
\end{enumerate}
\end{corollary}


We are now in a position to complete our description of the map $j$.

\begin{lemma} \label{lm closed embedding cartesian diagram}
In the conditions stated above, the morphism
\[
j : \Hilb_{n,C}^{cur} \times_{\ol \Jac_C} \PMod_{\nu} \longhookrightarrow \Hilb_{n,C}^{cur} \times_{\Hyb_\nu} \Hilb_{n,\Sigma}^{cur} \, , 
\]
as defined in \eqref{eq def j}, is a closed immersion. Furthermore, the following diagram commutes:
\[
    \begin{tikzcd}
        \Hilb_{n,C}^{cur} \times_{\ol\Jac_C} \PMod_{\nu}^0 \ar[d,"\wt{\alpha}_C"'] \ar[r,"q \circ j\ "] &  \Hilb_{n,\Sigma}^{cur} \ar[d,"\alpha_\Sigma"] \\
         \PMod_{\nu}^0 \ar[r,"\dot \nu"'] & \ol\Jac_\Sigma,
    \end{tikzcd}
\]
where $q:\Hilb_{n,C}^{cur} \times_{\Hyb_\nu} \Hilb_{n,\Sigma}^{cur}\to\Hilb_{n,\Sigma}^{cur}$ is the projection. 
\end{lemma}

\begin{proof}
Let $Z\subset\Hilb_{n,C}^{cur} \times_{\Hyb_\nu} \Hilb_{n,\Sigma}^{cur}$ denote the closed subset 
\[
Z:=\{(D, E)\in\Hilb_{n,C}^{cur} \times_{\Hyb_\nu} \Hilb_{n,\Sigma}^{cur}\,|\,\nu(E)\subset D\}.
\] 
The strategy of the proof will be to construct, on $Z$, an explicit inverse map to $j$. Consider a point $(D,E)\in Z$ and decompose $D$ and $E$ as per \eqref{eq decomposition of D} and \eqref{eq decomposition of E}. The geometry we require is an understanding of $\nu^*\calO_{D}$. By Corollary \ref{co description of Hilb^cur times PMod}, one has $\calO_{E_0} = \nu^* \calO_{D_0}$, alongside also the following case by case analysis of the decomposition: 
\begin{enumerate}
\item either $D_i$ and $E_i$ are empty;
\item or $D_i$ is supported on one of the branches of $C$ at $b_i$ (so either $\nu^*\calO_{D_i} = \calO_{E_i^+ + b_i^-}$ and $E_i = E_i^+$, or $\nu^*\calO_{D_i} = \calO_{b_i^+ + E_i^-}$ and $E_i= E_i^-$);
\item or $D_i$ is $D_{b_i,a_i}$ for $a_i \in \CC^*$ (so $\nu^*\calO_{D_i} \cong \calO_{b_i^+ + \, b_i^-} = \calO_{E_i}$).
\end{enumerate}
In all cases, there exists a divisor $F$ on $\Sigma$, supported on a subset of $\Exc(\nu)$, such that $\nu^*\calO_{D} = \calO_{E +F}$ and $F \cap E = \varnothing$.

The following homological analysis determines the pullback $\nu^*\calI_D^\vee$. We have the exact sequence
\[
0 \to \calO_C \to \calI_D^\vee \to \calO_D \to 0,
\] 
and pullback along $\nu: \Sigma \to C$ yields a 5-term exact sequence on $\Sigma$:
\[ 
0 \to \calTor_1(\calI_D^\vee,\calO_\Sigma) \to \calTor_1(\calO_D,\calO_\Sigma) \xrightarrow{\,(\ast)\,} \calO_\Sigma \to \nu^*\calI_D^\vee \to \nu^* \calO_{D}=\calO_{E+F} \to 0.
\]
Since $\calTor_1(\calO_D,\calO_\Sigma)$ is a torsion sheaf on $\Sigma$, the map $(\ast)$ is $0$ and we obtain the short exact sequence
\begin{align}\label{equ:pullback} 0 \to \calO_\Sigma \to \nu^*\calI_D^\vee \to \nu^* \calO_{D}=\calO_{E+F} \to 0.
\end{align} 
The coherent sheaf $\nu^*\calI_D^\vee$ is torsion-free at $\nu^{-1}(\URSing(\nu))$ and decomposes into a locally free sheaf and a torsion sheaf at the smooth points $\Exc(\nu)$. Hence, globally it decomposes into a torsion-free sheaf - the torsion-free pullback - and a torsion sheaf
\[ \nu^*\calI_D^\vee=\nu^{\mathrm{tf}} \calI_D^\vee \oplus \calT_D. \]
Let $\RSing(\nu)=S_{1,3} \cup S_2$ where $S_{1,3}$ contains all nodes of the resolved singular locus where $D$ is principal, i.e.~\ described by case (1) or (3), and $S_2$ are the nodes where it is not principal, i.e.~\ described by case (2). By a local computation one can see that $\calT_D=\calO_{\nu^{-1}(S_2)}=\calO_{F + \sigma F}$, where $\sigma: \Exc(\nu) \to \Exc(\nu)$ is the involution that pairwise interchanges the preimages $b_i^+,b_i^-$. Hence, equation \eqref{equ:pullback} reduces to
\[
0 \to \calO_\Sigma \to \nu^{\mathrm{tf}}\calI_D^\vee \to \calO_{E-\sigma F} \to 0.
\]
Therefore, we conclude 
\begin{align}\label{equ:tf_pullback}
    \nu^{\mathrm{tf}}\calI_D^\vee=\calI_{E-\sigma F}^\vee.
\end{align} 
Using the previous calculation we can now write down an exact sequence of the form \eqref{eq preimage of tau}, therefore defining the parabolic module corresponding to $\calI_D^\vee(-nx_0)$.
Consider the intermediate partial normalisation $\nu_2: \Sigma' \to C$ at the nodes in $S_2$. Then $\nu$ factors as 
\[  \nu: \Sigma \xrightarrow{\ \nu_1\ } \Sigma' \xrightarrow{\ \nu_2\ } C.
\]
By Cook \cite[Lemma 1]{cook:1998}, there is a push-pull isomorphism 
\[ \calI_D^\vee \cong \nu_{2*} \nu_2^{\mathrm{tf}} \calI_D^\vee,
\] where $\nu_2^{\mathrm{tf}}$ is the torsion-free pullback. Now, $\nu_2^{\mathrm{tf}} \calI_D^\vee$ is a locally free sheaf on $\Sigma'$ away from $\nu_2^{-1} ( \Sing(C) \setminus \RSing(\nu) )$, where $\nu_1$ and $\nu_2$ are locally isomorphisms. Hence, there exists an exact sequence
\[ 0 \to \nu_2^{\mathrm{tf}} \calI_D^\vee \to \nu_{1*} \nu_1^*\nu_2^{\mathrm{tf}} \calI_D^\vee \to \calO_{\RSing(\nu_1)} \to 0.\]
Note that $\RSing(\nu_1)=\nu_2^{-1}(S_{1,3})$ and $ \nu_1^*\nu_2^{\mathrm{tf}} \calI_D^\vee= \nu^{\mathrm{tf}}\calI_D^\vee$. Pushing forward this exact sequence to $C$, we obtain
\begin{align} \label{equ:ses_f1}
0 \to \calI_D^\vee \xrightarrow{\ f_1\ } \nu_{*} \nu^{\mathrm{tf}} \calI_D^\vee \to \calO_{S_{1,3}} \to 0.
\end{align}
From equation \eqref{equ:tf_pullback} we have the exact sequence
\[ 0 \to \nu^{\mathrm{tf}}\calI_{D}^\vee(-ny_0)=\calI_{E-\sigma F}^\vee(-ny_0) \xrightarrow{\ f_2\ } \calI_{E}^\vee(E-ny_0) \to \calO_{\sigma F} \to 0.
\] Composing its pushforward with $f_1$ of \eqref{equ:ses_f1} results in the desired sequence
\begin{align}\label{eq ses}
    0 \to \calI_{D}^\vee(-n x_0) \xrightarrow{\nu_*f_2 \circ f_1 } \nu_*\calI_{E}^\vee(-n y_0) \to \calO_{\RSing(\nu)} \to 0. 
\end{align} 
Then, since the map $\rho:\PMod_{\nu}^0\to\ol\Jac_C$ is surjective, $\calF := \alpha_C (\calI_D) = \calI_D^\vee(-nx_0)$ and $M:= \alpha_\Sigma(\calI_E) = \calI_{E}^\vee(-n y_0)$ determine a point $(M,V)$ in $\PMod_{\nu}$. Since we are specifying both $\calF$ and $M$, Lemma \ref{lm tau surjective} says that  the point $(M,V)$ is uniquely determined by \eqref{eq ses}. This gives rise to a map
\[
g : Z \to \Hilb_{n,C}^{cur} \times_{\ol\Jac_C} \PMod_{\nu},\ \ \ g(D,E)=(D,(\calI_{E}^\vee(-ny_0),V)).
\]
From the definition of $j$ in \eqref{eq def j}, we conclude that $g$ is inverse to $j$, hence $\Hilb_{n,C}^{cur} \times_{\ol\Jac_C} \PMod_{\nu}$ is naturally isomorphic to $Z$ closed in $\Hilb_{n,C}^{cur} \times_{\Hyb_\nu} \Hilb_{n,\Sigma}^{cur}$. 

Finally, the identification $M=\calI_{E}^\vee(-ny_0)$ is equivalent to the commutativity of the diagram.
\end{proof}

The next technical result will be applied in Section \ref{sc relation Poincares}. We say that the big diagonal $\Delta$ of $\Hyb_\nu$ is the preimage of the big diagonal in $\Sym_C^n$ under the natural morphism $\Hyb_\nu \to \Sym_C^n$ from \eqref{eq: chow descends}.

\begin{lemma} \label{lm complement of image j}
 The complement of the image of $j:\Hilb_{n,C}^{cur} \times_{\ol \Jac_C} \PMod_{\nu} \longhookrightarrow \Hilb_{n,C}^{cur} \times_{\Hyb_\nu} \Hilb_{n,\Sigma}^{cur}$ maps into the big diagonal $\Delta$ under the projection $q:\Hilb_{n,C}^{cur} \times_{\Hyb_\nu} \Hilb_{n,\Sigma}^{cur}\to\Hyb_\nu$.
\end{lemma}

\begin{proof}
Recalling the decomposition \eqref{eq decomposition of D}, if $(D, E)$ is a point of $\Hilb_{n,C}^{cur} \times_{\Hyb_\nu} \Hilb_{n,\Sigma}^{cur}$ such that every $D_i$ has length at most $1$, after Corollary \ref{co description of Hilb^cur times PMod} one immediately finds that $D_i$ are contained in the branches of $C$. In that case $\nu(E)$ embeds into $D$ so the pair lies on the image of $j$. Therefore, the points that lie outside the big diagonal of $\Hyb_\nu$ are contained in the image of $j$, and the last statement is a straightforward consequence of this fact. 
\end{proof}

\subsection{Equivalence of universal divisors}
\label{univ formulae} 

The main result of this section describes a relation between twists of the line bundles $\det \calA_{C,n}$ and $\det \calA_{\Sigma,n}$ constructed in \eqref{eq def Acn}. We also recall the tautological line bundles $\calV_{\Sigma, i}$ over $\PMod_{\nu}^d$, defined in \eqref{eq definition of Vv_i}, alongside the definition $\calV_{\Sigma} = \calV_{\Sigma, 1} \oplus \cdots \oplus \calV_{\Sigma, k}$ of \eqref{eq:def V}. 

\begin{proposition} \label{pr Aa_Sigma, Aa_C and Vv} 
Let $C$ be an irreducible nodal projective curve and let $\nu : \Sigma \to C$ be a partial normalisation at $k$ nodes. Then there exists an isomorphism
\begin{align} 
\label{eq complete relation of line bundles}
j^* q^* (\det(\calA_{\Sigma,n})\otimes \calO_{\Hilb_\Sigma}(\calD_{\Sigma, y_0})^{\,-k}) \otimes \wt\rho^* \left ( \det(\calA_{C,n})\right )^\vee & \cong 
\wt{\alpha}_C^{*} \det(\calV_{\Sigma}) , 
\end{align}
as line bundles on $\Hilb_{n,C}^{cur} \times_{\overline{\Jac_\Sigma}^0_C} \PMod_{\nu}$. 
\end{proposition}



The remainder of this section is dedicated to the proof of Proposition \ref{pr Aa_Sigma, Aa_C and Vv}. 
The idea is to relate $\calA_{C,n}$ and $\calA_{\Sigma,n}$ via a universal exact sequence on $\Hilb_{n,C}^{cur} \times_{\ol \Jac} \PMod_{\nu}$ that is a pullback of \eqref{eq:univ-ses}.
We shall make intensive use of the commutative diagram shown in Figure \eqref{equ:big_diagram} below.  
\begin{figure}[h]
\begin{equation}\label{equ:big_diagram}
\begin{tikzcd}
 \Sigma
\arrow[dd, equal] 
& \Sigma \times \ol\Jac_\Sigma
\arrow[l,"p_\Sigma"]
& \Sigma \times \Hilb^{cur}_{n,\Sigma} 
\ar[l,"\id \times \alpha_\Sigma"']
\arrow[ll,bend right=15,"\ol p_\Sigma"'{pos=0.4}]
\arrow[r, "h_{\Sigma}"] 
& \Hilb^{cur}_{n,\Sigma} \ar[r,"\alpha_\Sigma"] \ar[rdd,phantom,"\qquad \boxed{2}"]
& \ol\Jac_\Sigma
\\
&
& \Sigma \times \Hilb^{cur}_{n,C} \times_{\Hyb_\nu} \Hilb^{cur}_{n,\Sigma} 
\arrow[u, "\id \times q"]
\arrow[r,phantom,"\boxed{1}"]
& \Hilb^{cur}_{n,C} \times_{\Hyb_\nu} \Hilb^{cur}_{n,\Sigma} 
\arrow[u, "q"'] 
\\
\Sigma 
\arrow[d,"\nu"']
& \Sigma\times\PMod_{\nu} 
\arrow[d,"\nu\times\id"']
\arrow[uu,"\id\times\dot\nu"]
\arrow[dr, phantom, "\boxed{3}"]
& \Sigma \times \Hilb^{cur}_{n,C} \times_{\ol\Jac_C} \PMod_{\nu} 
\arrow[dr, phantom, "\boxed{4}"]
\arrow[ll,bend right=15,"q_{\Sigma}"'{pos=0.4}]
\arrow[u, "\id \times j"]
\arrow[r, "\wt h_{\Sigma}"]
\arrow[d, "\nu \times \id"']
\arrow[l, "\id\times\wt\alpha_C"]
& \Hilb_{n, C}^{cur} \times_{\ol\Jac_C} \PMod_{\nu} \ar[r,"\wt\alpha_C"]
\arrow[u, "j"']
& \PMod_{\nu} \ar[uu,"\dot\nu"']\ar[dd,"\rho"]
\\
C
\arrow[d,equal]
& C \times \PMod_{\nu} 
\arrow[d, "\id \times \rho"']
\arrow[dr, phantom, "\boxed{5}"]
& C \times \Hilb^{cur}_{n,C} \times_{\ol\Jac_C} \PMod_{\nu}
\arrow[dr, phantom, "\boxed{6}"]
\arrow[ll,bend left=15,"q_{C}"{pos=0.4}]
\arrow[l, "\id \times \wt\alpha_C"]
\arrow[d, "\id \times \wt\rho"']
\arrow[r, "\wt h_C"]
& \Hilb^{cur}_{n,C} \times_{\ol\Jac_C} \PMod_{\nu}
\arrow[u, equal]
\arrow[d, "\wt\rho"]
\\
C 
& C \times \ol\Jac_C \ar[l,"p_C"']
& C \times \Hilb^{cur}_{n,C}  \ar[ll,bend left=15,"\ol p_C"{pos=0.4}]
\arrow[l, "\id \times \alpha_C"]
\arrow[r, "h_C"']
& \Hilb^{cur}_{n,C} \ar[r,"\alpha_C"']
& \ol\Jac_C
\end{tikzcd}
\end{equation}
\end{figure}

Let us first express $\det(\calA_{C,n})$ and $\det(\calA_{\Sigma,n})$ in terms of the ideal sheaves.
\begin{lemma}\label{lemm:calA=R1calI}
We have isomorphisms
\[
\det(\calA_{C,n}) \cong \det\Big(R^1h_{C, *}\calI_{\calD_C}\Big), \quad \det(\calA_{\Sigma,n}) \cong \det\Big(R^1h_{\Sigma, *}\calI_{\calD_\Sigma}\Big).
\]
\end{lemma}
\begin{proof} Recall that, by definition $\calA_C = h_{C, *}\calO_{\calD_C}$. We use the short exact sequence 
\[ 
0 \to \calI_{\calD_C} \to \calO_{C \times \Hilb_C} \to \calO_{\calD_C} \to 0,
\]
over $C \times \Hilb^{cur}_{n,C}$. The long exact sequence induced by pushforward along $h_C: C \times \Hilb^{cur}_{n,C} \to \Hilb^{cur}_{n,C}$ is given by 
\begin{align*}
0 \to \calA_C 
\to R^1h_{C,*}\calI_{\calD_C} \to \calO_{\Hilb_C}^{\oplus g} \to 0,
\end{align*}
since $h_{C, *}\calI_{\calD_C} = 0$, $h_{C,*}\calO_{C \times \Hilb_C} = \calO_{\Hilb_C}$ and so $R^1h_{C,*}\calO_{C\times \Hilb_C}=\calO_{\Hilb_C}^{\oplus g}$.
Taking determinants yields the formula of the lemma. The same argument applies to the smooth curve $\Sigma$.
\end{proof}

Recall that $C$ is an irreducible nodal projective curve equipped with a partial normalisation $\nu:\Sigma\to C$ at $k$ (simple) nodes and with a choice of a smooth point $x_0\in C$ and $y_0=\nu^{-1}(x_0)$. Recall that we denote the resolved singular locus by $\RSing(\nu)=\{b_1, \dots, b_k\}\subset C$ and exceptional divisor by $\Exc(\nu)=\{b_1^+, b_1^-, \dots, b_k^+, b_k^-\}\subset \Sigma$, where in our notation $\nu^{-1}(b_i)=\{b_i^+,b_i^-\}$. Recall also the universal exact sequence \eqref{eq:univ-ses} involving the universal parabolic module $\left((\id \times \dot{\nu})^*\calU_\Sigma, \calV_{\Sigma} \right)$ and the torsion sheaf $\calT$ defined in \eqref{eq:calT}.

\begin{lemma}
\label{lem:unvrsl_ext_seq}
    There is an exact sequence, 
    \begin{align*} 
    0 \to (\id \times \wt\rho)^*\calI_{\calD_{C}}^{\vee}
\to (\nu \times \id)_{*} (\id \times &(q\circ j))^*\calI_{\calD_{\Sigma}}^{\vee} \to (\id\times\wt\alpha_C)^*\calT \otimes (\id \times \wt\rho)^*h_C^*\calO_{\Hilb_C}(\calD_{C,x_0}) \to 0 \, ,
    \end{align*}
     on $C\times\Hilb_{n,C}^{cur} \times_{\ol\Jac_C} \PMod_{\nu}$.
\end{lemma}
\begin{proof}
    By Assumption \ref{assumption}, $\alpha_C:\Hilb_{n,C}^{cur}\to\ol\Jac_C$ is surjective and smooth. It follows by \eqref{eq cartesian diagram Hilb PMod} that the same holds for $\wt{\alpha}_C$. In particular, $\id\times\wt{\alpha}_C: C \times \Hilb^{cur}_{n,C} \times_{\ol\Jac_C} \PMod_{\nu} \to C \times \PMod_{\nu}$ is flat. So pulling back along $\id \times \wt\alpha_C$ is an exact functor. We pullback the universal exact sequence \eqref{eq:univ-ses} along $\id \times \wt\alpha_C$. Using commutativity around diagram $\boxed{5}$ and flat base change around diagram $\boxed{3}$ of \eqref{equ:big_diagram}, we obtain 
    \begin{align} \label{equ:towards_universal}
    0 \to (\id \times \wt{\rho})^*(\id \times \alpha_C)^*\calU_C 
    \to (\nu \times \id)_*(\id \times \wt\alpha_C)^*(\id \times \dot\nu)^* \calU_\Sigma 
    \to (\id \times \wt\alpha_C)^*\calT \to 0 . 
    \end{align} 
    We saw in Lemma \ref{lm closed embedding cartesian diagram} that the diagram $\boxed{2}$ of \eqref{equ:big_diagram} commutes. 
    Hence, we can replace the middle term above by the sheaf  
    \[ (\nu \times \id)_*(\id \times q \circ j)^*(\id \times \alpha_\Sigma)^* \calU_\Sigma\,.
    \] 
    Now, using Lemma \ref{lemm:Extend_Schwarz}, we have an isomorphism 
    \[ (\id \times \alpha_\Sigma)^* \calU_\Sigma
    \cong \calI_{\calD_{\Sigma}}^{\vee} \otimes \ol p_\Sigma^* \calO_{\Sigma}(-ny_0)  \otimes h_\Sigma^*\calO_{\Hilb_\Sigma}(\calD_{\Sigma,y_0})^\vee,
    \] 
    as well as a corresponding isomorphism for $(\id \times \wt\alpha_C)^*\calU_C$. Inserting this into \eqref{equ:towards_universal} we obtain the exact sequence
    \begin{align*}
    0 &\to (\id \times \wt{\rho})^*\left(\calI_{\calD_{C}}^{\vee} \otimes \ol p_C^* \calO_{C}(-nx_0)  \otimes h_C^*\calO_{\Hilb_C}(\calD_{C,x_0})^\vee\right) \\
    &\to (\nu \times \id)_*(\id \times q \circ j)^* \left(\calI_{\calD_{\Sigma}}^{\vee} \otimes \ol p_\Sigma^* \calO_{\Sigma}(-ny_0)  \otimes h_\Sigma^*\calO_{\Hilb_\Sigma}(\calD_{\Sigma,y_0})^\vee\right)
    \to (\id \times \wt\alpha_C)^*\calT \to 0.
    \end{align*} 
    We saw in Corollary \ref{co description of Hilb^cur times PMod} that $E$ contains $y_0$ if and only if $D$ contains $x_0$. 
    Hence,
    \begin{equation}\label{eq:pullback slice under qj}
        (q\circ j)^*\calO_{\Hilb_\Sigma}(\calD_{\Sigma,y_0})\cong \wt{\rho}^*\calO_{\Hilb_C}(\calD_{C,x_0}).
    \end{equation} In particular, 
    \[ (\id \times (q\circ j))^*h_\Sigma^*\calO_{\Hilb_\Sigma}(\calD_{\Sigma,y_0})=(\nu \times \id)^*(\id \times \wt{\rho})^*h_C^*\calO_{\Hilb_C}(\calD_{C,x_0})\,.
    \] 
    So, using projection formula and tensoring with the line bundle $(\id \times \wt{\rho})^*h_C^*\calO_{\Hilb_C}(\calD_{C,x_0})$, we obtain 
    \begin{align*} 
    0 \to (\id \times \wt\rho)^*\left(\calI_{\calD_{C}}^{\vee} \otimes \ol p_C^* \calO_{C}(-nx_0) \right)
&\to (\nu \times \id)_{*}(\id \times (q\circ j))^* \left( \calI_{\calD_{\Sigma}}^{\vee} \otimes \ol p_\Sigma^* \calO_{\Sigma}(-ny_0) \right) \\
&\to (\id \times \wt\alpha_C)^*\calT \otimes (\id \times \wt\rho)^*h_C^*\calO_{\Hilb_C}(\calD_{C,x_0})
\to 0.
    \end{align*}
    Finally, using projection formula for $\nu^* \calO_{C}(-nx_0) = \calO_\Sigma(-ny_0)$ and tensoring by $q_C^*\calO_{C}(nx_0)=(\id\times\wt\rho)^*\ol p_C^*\calO_C(nx_0)$, yields the exact sequence stated in the lemma. Note that $(\id \times \wt\alpha_C)^*\calT$ is supported on $\RSing(\nu) \times\Hilb_{n,C}^{cur} \times_{\ol\Jac_C} \PMod_{\nu}$ and hence the tensorization with a line bundle that is a pullback from $C$ yields an isomorphism in the last term. 
    \end{proof}

To simplify some notation, let us introduce the abbreviations 
\[
\underline{C}:= C \times \Hilb^{cur}_{n,C} \times_{\ol\Jac_C} \PMod_{\nu} , \quad 
\underline{\Sigma}: = \Sigma \times \Hilb^{cur}_{n,C} \times_{\Hyb_\nu} \Hilb^{cur}_{n,\Sigma} , 
\]
for the triple products in the main diagram \eqref{equ:big_diagram}. Similarly, we set
\[ \ol\calI_C:=(\id \times \wt\rho)^*\calI_{\calD_{C}}, \quad \ol\calI_\Sigma:=(\id \times (q\circ j))^* \calI_{\calD_{\Sigma}}.
\] 

\begin{lemma}
\label{key lemma}
    There exists an exact sequence given by 
    \[ 0 \to  \wt\alpha_C^{*}  \calQ^{\vee} \otimes \widetilde{\rho}^*\calO_{\Hilb_C}(\calD_{C,x_0})^{\vee}
    \to  R^1\widetilde{h}_{\Sigma*}(\ol\calI_\Sigma \otimes q_{\Sigma}^*\omega_{\Sigma/C})
    \to R^1\wt{h}_{C*}\ol\calI_C 
    \to 0, 
    \]
    on $\Hilb^{cur}_{n,C} \times_{\ol\Jac_C} \PMod_{\nu}$.
\end{lemma}
\begin{proof} 
    With the above notation, Lemma \ref{lem:unvrsl_ext_seq} says that there exists an exact sequence on $\underline{C}$ given by 
    \[0 \to \ol\calI_C^{\vee}
\to (\nu \times \id)_*\ol\calI_\Sigma^{\vee} \to (\id\times\wt\alpha_C)^*\calT \otimes (\id \times \wt\rho)^*h_C^*\calO_{\Hilb_C}(\calD_{C,x_0})
\to 0.
    \] 
    Now we compute the dual of this sequence. First, since torsion-free sheaves on $C$ are reflexive (see Lemma 1.1 of \cite{hartshorne}), we have 
    \[  \calHom(\ol\calI_C^{\vee},\calO_{\underline{C}}) \cong \ol\calI_C^{\vee\vee}\cong\ol\calI_C.
    \] 
    To dualise the second term, we apply Grothendieck-Verdier duality. If $\omega_{\Sigma/C}:= \omega_\Sigma \otimes \nu^*\omega_C^{\vee}$ and $\omega_{\underline{\Sigma}/\underline{C}}:= \omega_{\underline{\Sigma}} \otimes (\nu\times\id)^*\omega_{\underline{C}}^{\vee}$ are the relative dualising sheaves, then $\omega_{\underline{\Sigma}/\underline{C}}=q_\Sigma^*\omega_{\Sigma/C}$, and then
    \begin{align*}
    \calHom((\nu \times \id)_{*}\ol\calI_\Sigma^{\vee}, \calO_{\underline{C}})
    & \cong (\nu \times \id)_*(\calHom( \ol\calI_\Sigma^{\vee}, \calO_{\underline{\Sigma}}) \otimes \omega_{\underline{\Sigma}/\underline{C}}) \\
    & \cong (\nu \times \id)_* (\ol\calI_\Sigma \otimes q_\Sigma^*\omega_{\Sigma/C}). 
    \end{align*}
    We therefore obtain the dual exact sequence
    \begin{align*}
    0 \shortrightarrow  (\nu \times \id)_*(\ol\calI_\Sigma \otimes q_\Sigma^*\omega_{\Sigma/C}) 
    \shortrightarrow \ol\calI_C 
    \shortrightarrow \calExt^1((\id\times\wt\alpha_C)^*\calT\otimes (\id \times \wt\rho)^*h_C^*\calO_{\Hilb_C}(\calD_{C,x_0}),\calO_{\underline{C}}) \shortrightarrow 0\,.
    \end{align*}
    By Hom-tensor adjointness  and functoriality around diagram $\boxed{6}$ from \eqref{equ:big_diagram}, the last term can be rewritten as 
    \begin{align*} 
    &\calExt^1((\id\times\wt\alpha_C)^*\calT, (\id \times \wt\rho)^*h_C^*\calO_{\Hilb_C}(\calD_{C,x_0})^\vee)\\
    =\ &\calExt^1((\id\times\wt\alpha_C)^*\calT ,\calO_{\underline{C}}) \otimes (\id \times \wt\rho)^*h_C^*\calO_{\Hilb_C}(\calD_{C,x_0})^\vee \\
    =\ &\calExt^1((\id\times\wt\alpha_C)^*\calT ,\calO_{\underline{C}}) \otimes \widetilde h_C^*\wt{\rho}^*\calO_{\Hilb_C}(\calD_{C,x_0})^\vee .
    \end{align*}
    Taking the pushforward along $\wt{h}_C: \underline{C} \to \Hilb^{cur}_{n,C} \times_{\ol\Jac_C} \PMod$,  we obtain the exact sequence
    \begin{align}
    \begin{split}
    \label{eq: final sequence}
    0 \to \wt h_{C,*}\calExt^1((\id\times\wt\alpha_C)^*\calT ,\calO_{\underline{C}}) &\otimes \widetilde{\rho}^*\calO_{\Hilb_C}(\calD_{C,x_0})^\vee\\
    &\to R^1\wt h_{C,*}(\nu \times \id)_*(\ol\calI_\Sigma\otimes q_\Sigma^*\omega_{\Sigma/C})
    \to R^1\wt h_{C,*}\ol\calI_C\to 0\,.
    \end{split} 
    \end{align}
    It remains to compare the first and second terms of this sequence with the statement of the Lemma. To compute the first term, we consider $\wt h_{C,*}^1((\id\times\wt\alpha_C)^*\calT ,\calO_{\underline{C}})$. We may apply Grothendieck-Verdier duality for the projection $\wt h_{C}$ with relative dualising sheaf $q_{C}^*\omega_C$. Also recall the definition of $\calQ$ from equation \eqref{equ:univ_quotient}. This allows us to write 
    \begin{align*}
    \wt h_{C,*}\calExt^1((\id\times\wt\alpha_C)^*\calT,\calO_{\underline{C}}) 
    & \cong \wt h_{C,*}R^1\calHom((\id\times\wt\alpha_C)^*\calT\otimes q_C^*\omega_C,q_C^*\omega_C) , \\ 
    & \cong \calHom( \wt h_{C,*} ((\id\times\wt\alpha_C)^*\calT\otimes q_C^*\omega_C),\calO_{\Hilb \times \PMod}) , \\ 
    & \cong \calHom( \wt h_{C,*}(\id\times\wt\alpha_C)^*\calT,\calO_{\Hilb \times \PMod})=\widetilde{\alpha}_C^*\calQ^\vee.
     \end{align*}
     Here, as in the proof of the previous lemma, $(\id\times\wt\alpha_C)^*\calT\otimes q_C^*\omega_C \cong(\id\times\wt\alpha_C)^*\calT$, since $(\id\times\wt\alpha_C)^*\calT$ is torsion, supported on $\RSing(\nu)\times \Hilb^{cur}_{n,C} \times_{\ol\Jac_C} \PMod_{\nu}$. This computation shows that the first term of \eqref{eq: final sequence} coincides with the statement of the Lemma. Finally, for the second term, commutativity of diagram \boxed{4} from \eqref{equ:big_diagram} yields the isomorphism  
     \[
     R^1\widetilde{h}_{C*}(\nu \times \id)_*(\ol\calI_\Sigma \otimes q_{\Sigma}^*\omega_{\Sigma/C})
     \cong R^1\widetilde{h}_{\Sigma*}(\ol\calI_\Sigma \otimes q_{\Sigma}^*\omega_{\Sigma/C}), 
     \] 
     and so \eqref{eq: final sequence} agrees with the statement of the lemma. This concludes the proof. 
\end{proof}

Next, we compute the determinant of the extension term from the previous lemma. 

\begin{lemma} 
\label{lemm:compute extension term}
There exists an isomorphism of line bundles, 
\[ 
\det\left(R^1\wt h_{\Sigma,*}(\ol\calI_{\Sigma} \otimes q_\Sigma^*\omega_{\Sigma/C})\right) \cong j^*q^*\left(\det(\calA_{\Sigma,n})\otimes \bigotimes_{i=1}^k \calO_{\Hilb_\Sigma}(\calD_{\Sigma, b_i^+})^\vee\otimes \calO_{\Hilb_\Sigma}(\calD_{\Sigma, b_i^-})^\vee\right),\]
on $\Hilb^{cur}_{n,C}\times_{\ol\Jac_C}\PMod_{\nu}$.
\end{lemma}
\begin{proof}
    Consider again the exceptional divisor $B=\{b_1^+,b_1^-,\ldots,b_k^+,b_k^-\}=\Exc(\nu) \subset\Sigma$. The canonical bundle of a nodal curve $C$ and the partial normalisation $\Sigma$ are related by the exact sequence   
    \[
    0 \to \omega_{C} \to \nu_*\omega_{\Sigma}(B) \xrightarrow{\,res\,} \calO_{\RSing(\nu)} \to 0\,,
    \] where $res$ is the sum of the residues; see \cite[Page 82]{harris-morrison}. 
    Pulling it back to $\Sigma$, and noticing that $\nu^*\omega_C$ is a line bundle (so torsion-free), yields 
    \[
    0 \to\nu^*\omega_{C} \to\nu^*\nu_*(\omega_{\Sigma}(B)) \xrightarrow{res} \calO_B \to 0.
    \]
    But $\nu^*\nu_*(\omega_{\Sigma}(B))\cong\omega_\Sigma(B)\oplus\calO_B$, hence 
    $\nu ^*\omega_C \cong \omega_{\Sigma}(B)$. The upshot is that the relative canonical bundle is 
    \[ 
    \omega_{\Sigma/C} = \omega_\Sigma \otimes \nu^*\omega_C^{\vee} \cong \calO_\Sigma(-B).
    \] 
    Define a divisor $\calB \subset \Sigma \times \Hilb_{n,\Sigma}^{cur} $ by
    \[
    \calB:=\ol p_\Sigma^{-1}(B)= B \times \Hilb_{n,\Sigma}^{cur}  .
    \]
    Then, we have the isomorphisms 
    \[ 
    \ol\calI_{\Sigma} \otimes q_\Sigma^*\omega_{\Sigma/C} 
    \cong (\id \times (q\circ j))^*(\calI_{\calD_{\Sigma}} \otimes \ol p_\Sigma^*\calO_\Sigma(-B))
    \cong (\id \times (q\circ j))^*(\calO_{\Sigma\times\Hilb_\Sigma}(-\calD_\Sigma-\calB)) . 
    \] 
    On the other hand, we have an exact sequence on $\Sigma \times \Hilb_{n,\Sigma}^{cur}$ given by 
    \[ 
    0 \to \calO_{\Sigma\times\Hilb_\Sigma}(-\calD_\Sigma-\calB) \to \calO_{\Sigma\times\Hilb_\Sigma}(-\calD_\Sigma) \to \calO_{\calB}\otimes \calO_{\Sigma\times\Hilb_\Sigma}(-\calD_\Sigma) \to 0.
    \] 
    Pushing it forward along $h_\Sigma$ yields  
    \[ 
    0 \shortrightarrow R^0 h_{\Sigma,*}(\calO_{\calB} \otimes \calO_{\Sigma\times\Hilb_\Sigma}(-\calD_\Sigma)) \to R^1h_{\Sigma,*}\calO_{\Sigma\times\Hilb_\Sigma}(-\calD_\Sigma-\calB)\to R^1h_{\Sigma,*}\calO_{\Sigma\times\Hilb_\Sigma}(-\calD_\Sigma) \shortrightarrow 0.
    \]
    Now we use that $R^0 h_{\Sigma,*}(\calO_{\calB} \otimes \calO_{\Sigma\times\Hilb_\Sigma}(-\calD_\Sigma))= \bigoplus_{i=1}^k \calO_{\Hilb_\Sigma}(\calD_{\Sigma, b_i^+})^\vee \otimes \calO_{\Hilb_\Sigma}(\calD_{\Sigma, b_i^-})^\vee 
    $. So, taking determinants and using Lemma \ref{lemm:calA=R1calI}, we obtain the isomorphisms
    \begin{align*} \det(R^1h_{\Sigma,*}(\calI_{\calD_{\Sigma}} \otimes \ol p_\Sigma^*\omega_{\Sigma/C}))&=\det(R^1h_{\Sigma,*}\calO_{\Sigma\times\Hilb_\Sigma}(-\calD_\Sigma-\calB))\\
     & \cong \det(R^1h_{\Sigma,*}\calO_{\Sigma\times\Hilb_\Sigma}(-\calD_\Sigma)) \otimes \bigotimes_{i=1}^k \calO_{\Hilb_\Sigma}(\calD_{\Sigma, b_i^+})^\vee \otimes \calO_{\Hilb_\Sigma}(\calD_{\Sigma, b_i^-})^\vee \\
    & \cong \det(\calA_{\Sigma,n}) \otimes \bigotimes_{i=1}^k \calO_{\Hilb_\Sigma}(\calD_{\Sigma, b_i^+})^\vee \otimes \calO_{\Hilb_\Sigma}(\calD_{\Sigma, b_i^-})^\vee
    \end{align*}
    Now we obtain the result by the base change around diagram $\boxed{1}$ from \eqref{equ:big_diagram}.
    \end{proof}

We are now ready to conclude the proof of Proposition \ref{pr Aa_Sigma, Aa_C and Vv}. 

\begin{proof}[Proof of Proposition \ref{pr Aa_Sigma, Aa_C and Vv}]
Using Lemmas \ref{lemm:calA=R1calI} and \ref{lemm:compute extension term}, and taking the determinant of the exact sequence from Lemma \ref{key lemma}, we conclude that
\begin{align*}
 &j^*q^*\left(\det(\calA_{\Sigma,n}) \otimes \bigotimes_{i=1}^k \calO_{\Hilb_\Sigma}(\calD_{\Sigma, b_i^+})^\vee\otimes \calO_{\Hilb_\Sigma}(\calD_{\Sigma, b_i^-})^\vee \right)\otimes \wt\rho^*\det(\calA_{C,n})^\vee \\
 \cong\  &\wt\alpha_C^*\det(\calQ)^\vee\otimes \wt\rho^*\calO_{\Hilb_C}(\calD_{C,x_0})^{\,-k}.
\end{align*}
Note that $\wt\alpha_C^*\calQ=(\wt h_{C,*}(\id\times\wt\alpha_C)^*\calT)^\vee$, and thus $\calQ$ has rank $k$.
From \eqref{eq:pullback slice under qj}, we have $\wt{\rho}^*\calO_{\Hilb_C}(\calD_{C,x_0})^k\cong j^*q^*\calO_{\Hilb_\Sigma}(\calD_{\Sigma,y_0})^k$. Moreover, by restricting the formula of Lemma \ref{lemm:Extend_Schwarz} (applied to the smooth curve $\Sigma$ with the marked point $y_0$) to each slice $\{b_i^\pm\}\times \Sym_\Sigma^n$, we get 
\[\calO_\Sym(\calD_{\Sigma,b_i^\pm})^\vee\cong\alpha_\Sigma^*\calU_{0,b_i^\pm}^\vee\otimes\calO_\Sym(\calD_{\Sigma,y_0})^\vee.
\] Hence, we obtain isomorphisms 
\begin{align*}
j^*q^*\Big( \det(\calA_{\Sigma,n}) \otimes \calO_{\Hilb_\Sigma}(\calD_{\Sigma,y_0})^{\,-k} \Big) \otimes \wt\rho^*\det(\calA_{C,n})^\vee & \cong \wt\alpha_C^*\det(\calQ)^\vee\otimes j^*q^*\alpha_\Sigma^*\bigg(\bigotimes_{i=1}^k(\calU_{0,b_i^+} \otimes \calU_{0,b_i^-})\bigg)\\
& \cong \wt\alpha_C^*\bigg(\det(\calQ)^\vee\otimes\bigotimes_{i=1}^k  (\dot\nu^*\calU_{0,b_i^+}\otimes\dot\nu^*\calU_{0,b_i^-})\bigg). 
\end{align*} 
The second isomorphism uses the commutative square $\boxed{2}$ in \eqref{equ:big_diagram}, which commutes by Lemma \ref{lm closed embedding cartesian diagram}. Finally, by the definition of $\calQ$ in \eqref{equ:univ_quotient}, we have, for each $i$, an exact sequence
\[ 0 \to \calV_{\Sigma, i} \to \dot{\nu}^*\calU_{0,b_i^+} \oplus \dot{\nu}^*\calU_{0,b_i^-} \to \calQ_i \to 0. 
\] Hence, we conclude $\det(\calQ)^\vee\otimes \bigotimes_{i=1}^k (\dot{\nu}^*\calU_{0,b_i^+} \otimes \dot{\nu}^*\calU_{0,b_i^-})=\det(\calV_{\Sigma})$. This completes the proof. 
\end{proof}

\section{Autoduality and the partial normalisation map}
\label{sc FM and normalisation}

This section computes the relation between two Fourier--Mukai transforms: the one associated to an integral nodal curve $C$ and the one associated to a partial normalisation $\Sigma \to C$. We first describe the restriction of the Poincar\'e sheaf to the locus of the compactified Jacobian described by pushforward under the normalisation map, which is addressed in Section \ref{sc relation Poincares}, and subsequently provide the relation of the associated Fourier--Mukai transforms in Section \ref{sc relation FM}.

\subsection{Isomorphism of Poincar\'e sheaves}
\label{sc relation Poincares}

We denote by $\calP_C$ and $\calP_{\Sigma}$ the respective Poincar\'e sheaves on $\ol\Jac_{C} \times \ol\Jac_{C}$ and $\ol\Jac_\Sigma \times \ol\Jac_\Sigma$. The main result of this section is a comparison between $\calP_{\Sigma}$ and the pullback of $\calP_C$ along the closed embedding 
\[
\check\nu \times \id : \ol\Jac_\Sigma^{\,-k} \times \ol\Jac_C \longhookrightarrow \ol\Jac_C \times \ol\Jac_{C}, 
\]
where
\begin{equation} 
\label{eq:pushforwardmorphism}
\check{\nu} : \ol\Jac_\Sigma^{\,-k} \longhookrightarrow \ol\Jac_C, \quad \calL \longmapsto \nu_{*}\calL \,  
\end{equation}
is the closed embedding defined by pushforward under $\nu$. 

We begin with a slice-by-slice comparison on restrictions of the form 
\[
\calP_{\Sigma,\calF} := \calP_\Sigma|_{\{\calF\} \times \ol\Jac_{\Sigma}}, \quad \calP_{C,\calF'} := \calP_C|_{\{\calF'\} \times \ol\Jac_C} ,
\]
taken at a compatible pair of geometric points $\calF \in \ol\Jac_{\Sigma}$ and $\calF' \in \ol\Jac_C$. Fix a smooth geometric point $x_0\in C$, with preimage $y_0 := \nu^{-1}(x_0)\subset\Sigma$. Our comparison uses the maps 
\[
\tau_{k,y_0} : \ol\Jac_\Sigma^{\,-k} \xrightarrow{\cong} \ol\Jac_\Sigma , \quad 
\rho:\PMod_{\nu}\longrightarrow\ol\Jac_C, \quad 
\dot\nu:\PMod_{\nu}\longrightarrow \ol\Jac_\Sigma, 
\]
as defined in \eqref{eq translation iso}, \eqref{eq definition of rho} and \eqref{eq proj PMod onto Jac} respectively. 

\begin{proposition} 
\label{pr fibrewise description of Poincare and normalisation}
For every geometric point $\calL \in \ol\Jac^{\,-k}_\Sigma$, there exists an isomorphism 
\begin{equation} 
\label{eq description of Pp on a slice}
\calP_{C,\nu_*\calL} 
 \cong \rho_* \Big(\det(\calV_{\Sigma})\otimes \dot \nu^* \calP_{\Sigma, \calL(ky_0)} \Big), \\
\end{equation} 
and similarly for the dual sheaf:  
\begin{align} 
\begin{split}
\label{eq description of Pp dual on a slice}
\calP_{C,\nu_*\calL}^{\vee} 
 \cong \rho_* \Big( \det(\calV_{\Sigma})^{\vee}\otimes\omega_{\PMod}\otimes  \dot \nu^* \calP_{\Sigma, \calL(ky_0)}^{\vee} \Big) . 
\end{split} 
\end{align}
\end{proposition}

\begin{proof}
Since $C$ is nodal and integral, by means of Proposition \ref{pr alpha curv surjective}, one can choose $n \gg 0$ such that the restriction of the Abel--Jacobi map to the curvilinear Hilbert scheme $\alpha_C$ is surjective. One can assume, without losing generality, that $\alpha_\Sigma$ is surjective too. In that case, and thanks to equivariance of the Poincar\'e sheaf under permutation \cite[Lemma 6.1.(c)]{arinkin} it is enough to describe  the restriction $\calG_{0, C,\nu_*\calL}:=\calG_{0, C}|_{\Hilb\times\{\nu_*\calL\}}$ of $\calG_{0, C}$ to the slice associated to $\nu_*\calL\in \ol\Jac_C$. 

As an application of proper base change and projection formula, one has the isomorphism $\left (\nu_*\calL \right )^{\boxtimes_n} \cong \nu_{*}^n\left ( \calL^{\boxtimes_n} \right )$, where $\nu^n : \Sigma^n \longrightarrow C^n$ is induced by $\nu$. Specializing formula \eqref{eq description of Gg_C} to the slice henceforth results in
\[
\calG_{0, C,\nu_*\calL} \cong \left (\psi_{C,*}\sigma_C^*\nu_{*}^n \calL^{\boxtimes_n} \right )^{sign}  \otimes \det(\calA_{C,n})^\vee.
\]

Consider the following Cartesian diagram
\begin{equation} \label{n-fold normalisation, Flag and surface}
\begin{tikzcd}[column sep = huge]
\Flag_{n, C}^{cur} \times_{C^n} \Sigma^n \arrow[r, "\wt\sigma"] \arrow[d, "\wt\nu"'] & \Sigma^n \arrow[d, "\nu^{n}"] \\
\Flag_{n, C}^{cur} \arrow[r, "\sigma_C"'] & C^n.
\end{tikzcd}
\end{equation}
Observe that the data of a length $1$ extension of ideal sheaves $\calI_D \subset \calI_{D'}$ over $C$, an ideal sheaf $\calI_{E}$ such that $\chow_{C}(\calI_D) = \chow_{\Sigma}(\calI_E)$ and a point $y \in \Sigma$ whose image $\nu(y)$ coincides with $D'\setminus D$, determines naturally a length $1$ extension of ideal sheaves $\calI_E \subset \calI_{E'}$ over $\Sigma$. To prove this claim note that either the partial normalisation $\nu : \Sigma \to C$ is a local isomorphism around $y$ or $\Sigma$ is smooth at $y$. In the first case, the extension $\calI_D \subset \calI_{D'}$ determines naturally the extension $\calI_E \subset \calI_{E'}$ and in the second case, there is only one possible extension of ideal sheaves supported at a $y$ as it is a smooth point. This implies that a point $((\calI_{D_1}, \dots, \calI_{D_n}), (y_1, \dots, y_n ))$ in $\Flag_{n, C}^{cur} \times_{C^n} \Sigma^n$ determines naturally a point $(\calI_{E_1}, \dots, \calI_{E_n})$, so one obtains a morphism
\[
s: \Flag_{n, C}^{cur} \times_{C^n} \Sigma^n \to \Flag_{n, \Sigma}^{cur}
\]
fitting in a commutative diagram
\begin{equation} \label{wtsigma = s composed with sigma}
\begin{tikzcd}[column sep = huge]
\Flag_{n, C}^{cur} \times_{C^n} \Sigma^n \arrow[r, "s"] \arrow[rd, "\wt\sigma"'] & \Flag_{n, \Sigma}^{cur} \arrow[d, "\sigma_\Sigma"] 
\\
 & \Sigma^n
\end{tikzcd}.
\end{equation}
Making use of $\sigma_C$, $\sigma_\Sigma$ and $s$ defined above, one can easily construct the morphism $\wt\psi$ fitting in the commutative square  
\begin{equation} \label{Flag, normalisation and Hilb}
\begin{tikzcd}[column sep = huge]
\Flag_{n, C}^{cur} \times_{C^n} \Sigma^n \arrow[r, "\wt\psi"] \arrow[d, "\wt\nu"'] & \Hilb_{n, C}^{cur} \times_{\Hyb_\nu} \Hilb_{n, \Sigma}^{cur} \arrow[d, "\beta"] \\
\Flag_{n, C}^{cur} \arrow[r, "\psi_C"'] & \Hilb_{n, C}^{cur}
\end{tikzcd}.
\end{equation}
Recalling that $\nu^{n}$ is proper (because $\nu$ is so), we use proper base change around (\ref{n-fold normalisation, Flag and surface}), and functoriality around \eqref{wtsigma = s composed with sigma} and \eqref{Flag, normalisation and Hilb}, to get 
\[
\calG_{0, C,\nu_*\calL} \cong \left (\beta 
_{*}\wt\psi_{*} s^* \sigma^{*}_{\Sigma} \calL^{\boxtimes_n} \right )^{sign} \otimes \det(\calA_{C,n})^\vee,
\]
where the permutation action lifts to $\Flag_{n, C}^{cur} \times_{C^n} \Sigma^n$ by pullback under $\wt \nu$. By invariance of $\beta$ with respect to this action we obtain, 
\[
\calG_{0, C,\nu_*\calL} \cong \beta 
_{*} \left (\wt\psi_{*} s^* \sigma^{*}_{\Sigma} \calL^{\boxtimes_n} \right )^{sign} \otimes \det(\calA_{C,n})^\vee.
\]
We shall see next that the previous construction factors through the closed embedding $j:\Hilb^{cur}_{n,C} \times_{\ol\Jac_C} \PMod_{\nu} \longhookrightarrow \Hilb^{cur}_{n,C} \times_{\Hyb_\nu} \Hilb_{n,\Sigma}^{cur}$ of \eqref{eq def j}.

We claim that the support of $\left (\wt\psi_{*} s^* \sigma^{*}_{\Sigma} \calL^{\boxtimes_n} \right )^{sign}$ lies in the image of $j$. This will be proved in Lemma \ref{lm support} below. Consequently, one has that
\[
\calG_{0, C,\nu_*\calL} \cong \wt \rho
_{*} j^* \left (\wt\psi_{*}  s^* \sigma^{*}_{\Sigma}\calL^{\boxtimes_n} \right )^{sign} \otimes \det(\calA_{C,n})^\vee,
\]
where we observe, as in \eqref{eq cartesian diagram Hilb PMod}, that $\wt \rho = \beta \circ j$.

The projection $\psi_C:\Flag_{n, C}^{cur} 
\to \Hilb_{n, C}^{cur}$, gives rise to $\Flag_{n, C}^{cur} \times_{C^n} \Sigma^{n} 
\to \Hilb_{n, C}^{cur} \times_{\Hyb_\nu} \Flag_{n, \Sigma}^{cur}$. By definition, the elements of $\Flag_{n, C}^{cur}$ and $\Hilb_{n, C}^{cur}$ are locally contained in smooth curves. Hence, starting from $D \in \Hilb_{n, C}^{cur}$, one determines uniquely a filtration of $D$ out of a filtration on $\Flag_{n,\Sigma}$. This naturally provides an inverse to the previous morphism, so
\[
\Flag_{n, C}^{cur} \times_{C^n} \Sigma^{n} 
\cong \Hilb_{n, C}^{cur} \times_{\Hyb_\nu} \Flag_{n, \Sigma}^{cur}.
\]
The statement above, followed by the isomorphism
\[
\Hilb_{n, C}^{cur} \times_{\Hyb_\nu} \Flag_{n, C}^{cur}
\cong ( \Hilb_{n, C}^{cur} \times_{\Hyb_\nu} \Hilb_{n, \Sigma}^{cur} ) \times_{\Hilb_{n, \Sigma}^{cur}} \Flag_{n,\Sigma}^{cur}, 
\]
provides us with the Cartesian diagram, 
\begin{equation*}
\begin{tikzcd}[column sep = huge]
\Flag_{n, C}^{cur} \times_{C^n} \Sigma^{n} \arrow[r, "\wt{\psi}"] \arrow[d, "s"'] & \Hilb^{cur}_{n,C} \times_{\Hyb_\nu}\Hilb_{n, \Sigma}^{cur} \arrow[d, "q"] \\
\Flag_{n,\Sigma}^{cur} \arrow[r, "\psi_{\Sigma}"'] & \Hilb_{n, \Sigma}^{cur},
\end{tikzcd}
\end{equation*}
where $q$ denotes the obvious projection. Since $\pi_\Sigma$ is proper, then proper base change with respect to it gives 
\[
\calG_{0, C,\nu_*\calL} \cong \wt \rho
_{*} j^* \left ( q^*  \psi_{\Sigma,*} \sigma_\Sigma^* \calL^{\boxtimes_n} \right )^{sign} \otimes \det(\calA_{C,n})^\vee,
\]
and by equivariance with respect to the action of the symmetric group,
\begin{equation} \label{Q restriction stage 2}\calG_{0, C,\nu_*\calL} \cong \wt \rho
_{*} j^*q^* \left (\psi_{\Sigma,*} \sigma_\Sigma^* \calL^{\boxtimes_n} \right )^{sign} \otimes \det(\calA_{C,n})^\vee.
\end{equation}

We now turn our attention to the sheaf $\calG_{-k, \Sigma}$. After \eqref{eq description of Gg_C} (applied to $\Sigma$), we have that the restriction $\calG_{-k, \Sigma}|_{\Hilb_{n,\Sigma}\times\{\calL\}}$ to slice associated to $\calL \in \Jac_\Sigma^{\,-k}$ is
\[
\calG_{-k, \Sigma, \calL} \cong ( \psi_{\Sigma, *} \sigma_{\Sigma}^{*}( \calL^{\, \boxtimes_n} ) )^{sign} \otimes \det(\calA_{\Sigma,n})^\vee. 
\]
Recalling Lemma \ref{lm Gg_0 and Gg_d} (again applied to $\Sigma$ and to the point $y_0$), 
\[
\calG_{-k,\Sigma, \calL} \cong  \calG_{0,\Sigma, \calL(ky_0)} \otimes \calO_{\Hilb_\Sigma}(\calD_{\Sigma,y_0})^{\,-k}.
\]
Using the two previous isomorphisms in (\ref{Q restriction stage 2}) yields 
\[
\calG_{0, C,\nu_*\calL} \cong \wt \rho_{*}j^*q^* \left( \calG_{0, \Sigma, \calL(ky_0)} \otimes \det(\calA_{\Sigma,n}) \otimes \calO_{\Hilb_C}(\calD_{\Sigma,y_0})^{\,-k} \right ) \otimes \det(\calA_{C,n})^\vee.
\]
Statement \eqref{descent construction of Arinkin} and functoriality with respect to the commutative diagram
\[
\begin{tikzcd}[column sep = huge]
\Hilb_{n, C}^{cur} \times_{\ol \Jac_C} \PMod_\nu \arrow[r, "\tilde \alpha_C"] \arrow[d, "q \circ j"'] & \PMod_\nu \arrow[d, "\dot{\nu}"] \\
\Hilb_{n, \Sigma}^{cur} \arrow[r, "\alpha_{\Sigma}"'] & \ol\Jac_{\Sigma},
\end{tikzcd}
\]
gives rise to the isomorphism 
\[
\calG_{0, C,\nu_*\calL} \cong \wt \rho_{*} \left( \wt{\alpha}_C^* \dot \nu^{*} \calP_{\Sigma, \calL(ky_0)} \otimes j^*q^*\det(\calA_{\Sigma,n}) \otimes j^*q^*\calO_{\Hilb_C}(\calD_{\Sigma,y_0})^{\,-k} \right ) \otimes \det(\calA_{C,n})^\vee.
\]
An application of the projection formula together with Proposition \ref{pr Aa_Sigma, Aa_C and Vv}, allows us to write
\begin{align*}
\calG_{0, C,\nu_*\calL} 
& \cong \wt{\rho}_{*} \wt\alpha_C^{*} \Big( \dot{\nu}^{*} \calP_{\Sigma, \calL(ky_0)} \otimes \det(\calV_{\Sigma}) \Big) . 
\end{align*}
Proper base change with respect to the Cartesian diagram \eqref{eq cartesian diagram Hilb PMod} provides us with the conclusive isomorphism 
\begin{align}\label{equ:calG_slice}
\calG_{0, C,\nu_*\calL} 
& \cong \alpha_C^* \rho_* \left( \dot \nu^* \calP_{\Sigma, \calL(ky_0)} \otimes  \det(\calV_{\Sigma}) \right ) .  
\end{align}
Note that $\rho_* \left( \dot \nu^* \calP_{\Sigma, \calL(ky_0)} \otimes  \det(\calV_{\Sigma}) \right )$ is maximal Cohen--Macaulay as $\rho: \PMod_{\nu} \to \ol\Jac_C$ is finite and $\dot\nu^* \calP_{\Sigma, \calL(ky_0)} \otimes  \det(\calV_{\Sigma})$ is maximal Cohen--Macaulay. Therefore, both sides of \eqref{equ:calG_slice} extend along $\Hilb_{C,n}^{cur} \subset \Hilb_{C,n}$ to isomorphic  maximal Cohen--Macaulay sheaves using \cite[Lemma 2.2]{arinkin} and \cite[Lemma 3.9]{melo2}. Now, the faithfulness of the pullback along the projective bundle $A_C:\Hilb_{C,n} \to \ol\Jac_C$ (recall Assumption \ref{assumption}) implies the first statement of the proposition, namely \eqref{eq description of Pp on a slice}. 

We address the second statement making use of Grothendieck-Verdier duality in \eqref{eq description of Pp on a slice}, recalling that $\rho$ is proper and that the dualising sheaf of $\ol\Jac_C$ is trivial:
\begin{align*}
    \calP_{C,\nu_*\calL}^{\vee}&\cong\calHom(\rho_* \left( \dot \nu^* \calP_{\Sigma, \calL(ky_0)} \otimes  \det(\calV_{\Sigma}) \right ),\calO_{\ol\Jac_C})\\
    &\cong\rho_*(\calHom(\dot \nu^* \calP_{\Sigma, \calL(ky_0)} \otimes  \det(\calV_{\Sigma}),\omega_{\PMod}))\\
    &\cong\rho_*(\dot \nu^* \calP^{\vee}_{\Sigma, \calL(ky_0)} \otimes  \det(\calV_{\Sigma})^{\vee}\otimes\omega_{\PMod}).
\end{align*}

Now, the second statement \eqref{eq description of Pp dual on a slice} follows naturally from \eqref{eq description of Pp on a slice} and Lemma 2.1 of \cite{arinkin}. This concludes the proof of Proposition \ref{pr fibrewise description of Poincare and normalisation}. 
\end{proof}

The above proof relied on the following Lemma, whose proof we now address.   

\begin{lemma} \label{lm support}
The support of the sheaf $\left (\wt\psi_{*} \wt\sigma^{*} \calL^{\boxtimes_n} \right )^{sign}$ is contained in the image of the map 
\[
j:\Hilb^{cur}_{n,C} \times_{\ol\Jac_C} \PMod_{\nu}\to \Hilb_{n, C}^{cur} \times_{\Hyb_\nu} \Hilb^{cur}_{n,\Sigma} , 
\] 
as defined in \eqref{eq def j}. 
\end{lemma} 
\begin{proof}
Recall that the action of the symmetric group $\sym_n$ on $\Flag^{cur}_{n,C}$ is induced by its action on $C^n$. It follows that, for $\gamma\in\sym_n$, the fixed point set $\left ( \Flag_{n,C}^{cur} \right )^\gamma$ coincides with the preimage of $(C^n)^\gamma$ inside $\Flag_{n,C}^{cur}$. Furthermore, since we consider the action of $\sym_n$ on $\Flag_{n, C}^{cur} \times_{C^n} \Sigma^n$ to be the one obtained by pullback under $\wt \nu:\Flag_{n, C}^{cur} \times_{C^n} \Sigma^n\to\Flag_{n, C}^{cur}$, and if $\left ( \Flag_{n, C}^{cur} \times_{C^n} \Sigma^n \right )^\gamma$ denotes the locus fixed by $\gamma$, then 
\begin{equation}\label{eq fixed by gamma}
    \left ( \Flag_{n, C}^{cur} \times_{C^n} \Sigma^n \right )^\gamma\cong \left ( \Flag_{n, C}^{cur} \right )^\gamma \times_{(C^n)^\gamma} (\nu^n)^{-1} (C^n)^\gamma.
\end{equation}

By Lemma \ref{lm closed embedding cartesian diagram}, the image of the immersion $j$ is closed in $\Hilb_{n, C}^{cur} \times_{\Hyb_\nu} \Sym_\Sigma^n$, and by Lemma \ref{lm complement of image j} lies over the big diagonal $\Delta\subset\Hyb_\nu$, so it is contained in $\chow^{-1}\left ( \Delta \right ) \times_{\Delta}(\nu^{(n)})^{-1}\left ( \Delta \right ) \subset \Hilb_{n, C}^{cur} \times_{\Hyb_\nu} \Sym_\Sigma^n$. Observe that $\Delta$ is the image under $\pi_C : C^n \to \Hyb_\nu$ of the union of $(C^n)^\gamma$, with $\gamma$ varying in $\sym_n\setminus\{0\}$. It follows, using \eqref{eq fixed by gamma}, that the complement of the image of $j$ is contained in
\[
\bigcup_{\gamma\in\sym_n\setminus\{0\}}\wt \psi \left ( \left ( \Flag_{n, C}^{cur} \times_{C^n} \Sigma^n \right )^\gamma \right ).
\]
Consider $\gamma \in \sym_n$ odd , since $\gamma$ is the identity on $\left ( \Flag_{n, C}^{cur} \times_{C^n} \Sigma^n \right )^\gamma$, it follows that the restriction of $\left (\wt\psi_{*}\wt\sigma^{*} \calL^{\boxtimes_n} \right )^{sign}$ to $\wt \psi \left ( \left ( \Flag_{n, C}^{cur} \times_{C^n} \Sigma^n \right )^\gamma \right )$ vanishes as the sections must satisfy $-s=\gamma^*s = s$. For any non-trivial even $\gamma \in \sym_n$, there always exists an odd permutation $\gamma'$ acting trivially on $\left ( \Flag_{n, C}^{cur} \times_{C^n} \Sigma^n \right )^\gamma$, the locus fixed by $\gamma$ and $\left (\wt\psi_{*}\wt\sigma^{*} \calL^{\boxtimes_n} \right )^{sign}$ vanishes there. Hence, $\left (\wt\psi_{*} \wt\sigma^{*} \calL^{\boxtimes_n} \right )^{sign}$ vanishes on the complement of $\im(j)$ 
\end{proof}

With the fibrewise comparison between $\calP_{C}$ and $\calP_{\Sigma}$ at hand, we now pass to the global comparison result, which describes how the Poincar\'e sheaves interact with the partial normalisation map. 

\begin{theorem} 
\label{tm Poincare and normalisation} 
Let $C$ be an integral nodal curve. Fix a partial normalisation $\nu : \Sigma \to C$ that resolves precisely $k$ nodes. Pick $y_0 \in \Sigma$ such that $\nu(y_0)$ is a smooth point of $C$. Then, there exists an isomorphism 
\begin{align} 
\begin{split}
\label{eq relation of Poincares}
(\check\nu \times \id)^*\calP_C 
\cong (\id \times \rho)_* \left ( \q_2^* \det(\calV_{\Sigma}) \otimes (\tau_{k, y_0} \times \dot{\nu})^*\calP_\Sigma \right),
\end{split}
\end{align}
over $\ol\Jac_\Sigma^{\,-k}\times \ol\Jac_C$, where $\q_2 : \ol\Jac_\Sigma^{\,-k} \times \PMod_{\nu} \to \PMod_{\nu}$ is the natural projection.

For the dual sheaves, we similarly have
\begin{align} 
\begin{split}
\label{eq relation of dual Poincares}
(\id \times \check\nu)^*\calP_C^{\vee} 
& \cong (\rho \times \id)_* \left ( \q_1^*\det(\calV_{\Sigma})^{\vee}\otimes\q_1^*\omega_{\PMod}\otimes  (\dot{\nu} \times \tau_{k, y_0})^*\calP_\Sigma^{\vee} \right ) \, . 
\end{split}
\end{align}
where $\q_1 : \PMod_{\nu} \times \ol\Jac_\Sigma^{\,-k} \to \PMod_{\nu}$ is the natural projection. 
\end{theorem}

\begin{remark}
In Theorem \ref{tm Poincare and normalisation}, the term $\det(\calV_{\Sigma})$ universally reflects the possible variations of the subspaces $V$ that define points $(M,V) \in \PMod_{\nu}^{d}$ on the moduli of parabolic modules (see Definition \ref{def par mod}). 
\end{remark}

\begin{remark}\label{rmk:indep y_0}
    The Poincaré sheaf $\calP_C$ is the unique maximal Cohen-Macaulay extension of the Poincaré line bundle on $\ol\Jac_C \times \Jac_C \cup \Jac_C \times \ol\Jac_C$ defined via determinant of cohomology \cite[Theorem A]{arinkin}. Moreover, this Poincaré line bundle in degree $0$ is independent of the choice of normalisation of the universal sheaf $\calU_C$ \cite[Remark 4.2]{melo2}. Hence the left-hand-side of formula \eqref{eq relation of Poincares} is independent of the choice of the smooth point $y_0$. 
    
    That the right-hand-side of \eqref{eq relation of Poincares} is independent of this choice can be seen a posteriori by Theorem \ref{tm Poincare and normalisation}, but also a priori as follows. Changing the fixed point to $y_1 \in \Sigma\setminus\nu^{-1}(\Sing(C))$ acts on the Poincaré sheaf by 
    \[
    (\tau_{k, y_1} \times \id)^*\calP_\Sigma \cong  \pi_2^*\calP_{\Sigma,\calO(y_1-y_0)^k} \otimes (\tau_{k, y_0} \times \id)^*\calP_\Sigma
    \]
    as proven by Arinkin in \cite[Lemma 6.5]{arinkin}. Here $\pi_2$ is the projection onto the second factor of $\ol\Jac_\Sigma \times \ol\Jac_\Sigma$. Moreover, by compatibility with the Abel-Jacobi map one has $\calP_{\Sigma,\calO(y_1-y_0)^k} \cong\calP_{\Sigma,\calO(y_1-y_0)}^k\cong \calU_{\Sigma,y_1}^{k}$, and thus $(\tau_{k,y_1} \times \dot\nu)^*\calP_{\Sigma}\cong \pi_2^*\dot\nu^*\calU_{\Sigma,y_1}^k\otimes(\tau_{k,y_0} \times \dot\nu)^*\calP_\Sigma$. On the other hand, $\det(\calV_{\Sigma})$ depends on the choice of $y_0$ and changes to $\det(\calV_{\Sigma})\otimes \dot\nu^*\calU_{\Sigma,y_1}^{\,-k}$ upon choosing $y_1$. A cancellation of terms allows us to conclude the right-hand-side is independent of choosing $y_0$. The same conclusion holds for \eqref{eq relation of dual Poincares}.
\end{remark}

We need to recall the {\it see-saw} principle before addressing the proof of Theorem \ref{tm Poincare and normalisation}. Here, we reproduce the statement as in \cite[Lemma 5.5]{melo2}, adapting the hypothesis to our case.

\begin{lemma}[See-saw principle] \label{lm see-saw}
	Let $Z$ and $T$ be two reduced locally Noetherian schemes with $Z$ proper and connected. Let $\calE$ and $\calF$ be two sheaves on $T \times Z$, flat over $T$, such that 
	\begin{enumerate}
		\item[(i)] $\calF |_{\{ t \} \times Z} \cong \calE |_{\{ t \} \times Z}$, for all $t \in T$;
		\item[(ii)] $\calF|_{\{ t \} \times Z}$ is simple for every $t \in T$;  
		\item[(iii)] there exists $z_0 \in Z$ and an isomorphism of line bundles $\calF |_{T \times \{ z_0 \}} \cong \calE |_{T \times \{ z_0 \}}$.
	\end{enumerate}
	Then $\calE$ and $\calF$ are isomorphic. 
\end{lemma}

\begin{proof}[Proof of Theorem \ref{tm Poincare and normalisation}] 

It suffices to show that the right and left-hand-sides of \eqref{eq relation of Poincares} and \eqref{eq relation of dual Poincares} satisfy the hypothesis of the see-saw principle. 

First we check the flatness hypothesis. Let $\calH := \q_2^* \, \det(\calV_{\Sigma}) \otimes  (\tau_{k, y_0} \times \dot{\nu})^*\calP_\Sigma$. This is a sheaf over $\ol\Jac_\Sigma^{\,-k} \times \PMod_{\nu}$ that is flat over $\ol\Jac_\Sigma^{\,-k}$. Thus pushforward $(\id \times \rho)_* \calH$ along the affine map $\id \times \rho$ is flat over $\ol\Jac_\Sigma^{\,-k}$ by \cite[Lemma 29.25.4]{stacks}. Hence the right-hand side of \eqref{eq relation of Poincares} is a flat sheaf over the first factor. Similarly, starting with the sheaf $\q_1^*\omega_\PMod \otimes\calH^\vee$, one can show that the right-hand side of \eqref{eq relation of dual Poincares} is a flat sheaf over the second factor. For the left-hand sides of \eqref{eq relation of Poincares} and \eqref{eq relation of dual Poincares}, flatness follows by noting that $\calP_C$ and $\calP^\vee_C$ are flat over the second factor, and so $(\check\nu \times \id)^*\calP_C$ and  $(\id \times \check\nu)^*\calP_C^\vee$ are flat over $\ol\Jac_\Sigma^{\, -k}$. 

Now we deal with (i)-(iii). Proposition \ref{pr fibrewise description of Poincare and normalisation} implies that condition (i) holds for \eqref{eq relation of Poincares} and \eqref{eq relation of dual Poincares}. Condition (ii) is automatic since torsion-free rank one sheaves on irreducible varieties are simple. Since $\rho$ is locally an isomorphism at $\calO_C \in \ol\Jac_C$ and $\q_2^*\det(\calV_{\Sigma})$ is pullback from $\PMod_{\nu}$ the right-hand-side restricted to $\ol\Jac^{\,-k}_\Sigma \times \{ \calO_C \}$ is trivial as the Poincar\'e bundle of $\Sigma$ restricts to the trivial bundle on $\ol\Jac_\Sigma \times \{ \calO_\Sigma \}$. Therefore, hypothesis (iii) holds for \eqref{eq relation of Poincares}. Similarly, $\calP_C^\vee$ and $\calP_\Sigma^\vee$ are trivial when restricted to $\{ \calO_C \} \times \ol\Jac_C$ and $\{ \calO_\Sigma \} \times \ol\Jac_\Sigma$, implying (iii) in the case of \eqref{eq relation of dual Poincares}.
\end{proof}

\subsection{Equivalence of Fourier--Mukai transforms}
\label{sc relation FM}

Using the description of the Poincar\'e sheaves obtained in Section \ref{sc relation Poincares}, we provide here a relation between the associated Fourier--Mukai transforms. 


\begin{theorem} 
\label{tm relation of FM dual}
Let $C$ be an integral nodal curve. Fix a partial normalisation $\nu : \Sigma \to C$ that resolves precisely $k$ nodes. Pick $y_0 \in \Sigma$ such that $\nu(y_0)$ is a smooth point of $C$. Then, for every object $\calF^\bullet \in D^b(\ol\Jac_\Sigma^{\,-k})$, one has an isomorphism
\begin{align} 
\begin{split}
\label{eq relation of FM}
\Phi^{\calP_C}_{1 \rightarrow 2} \left ( \check\nu_*\calF^\bullet \right ) 
& \cong \rho_* \left (\det(\calV_{\Sigma}) \otimes  \dot\nu^*\Phi^{\calP_\Sigma}_{1 \rightarrow 2}(\tau_{k, y_0,*}\calF^\bullet) \right ) ,
\end{split} 
\end{align}
and, similarly, for the inverse transform one has 
\begin{align} 
\begin{split}
\label{eq relation of dual FM}
\Phi^{\calP^{\vee}_C}_{2 \rightarrow 1} \left ( \check\nu_*\calF^\bullet \right ) 
& \cong \rho_* \left (\det(\calV_{\Sigma})^{\vee}\otimes\omega_{\PMod} \otimes  \dot\nu^*\Phi^{\calP_\Sigma^{\vee}}_{2 \rightarrow 1}(\tau_{k, y_0,*}\calF^\bullet) \right ) . 
\end{split}
\end{align}
\end{theorem} 

\begin{remark}
\label{re: convolution rule}
Theorem \ref{tm relation of FM dual} can be interpreted as a rule for composing Fourier--Mukai transforms with the convolution functor
\begin{equation}
\label{eq: def convolution V}
\Theta^{\calV_{\Sigma}} := \rho_{*} \big(\det(\calV_{\Sigma}) \otimes \dot\nu^{*}(\bullet)  \big) : D^b(\ol\Jac_{\Sigma}) \to D^b(\ol\Jac_C) , 
\end{equation} 
where $\det(\calV_{\Sigma})$ is a tautological choice of convolution kernel, universal for the variation of the vector spaces $V$ for points $(M, V)$ in the moduli $\PMod_{\nu}$ of parabolic modules. Indeed, Theorem \ref{tm relation of FM dual} provides a natural equivalence  
\[
\Phi^{\calP_C}_{1 \rightarrow 2} \circ \check\nu_* \simeq \Theta^{\calV_{\Sigma}} \circ \Phi^{\calP_\Sigma}_{1 \rightarrow 2} \circ \tau_{k, y_0} , 
\]
emulating formulae for composing integration and convolution in Fourier analysis. 
\end{remark}

\begin{remark} Recall that $\rho$ is a finite morphism and hence the pushforward $\rho_*$ is exact (see \cite[Corollaire 5.2.2]{EGAII}). Thus, from Theorem \ref{tm relation of FM dual} we can deduce that the pushforward $\check\nu_*\calF$ will be WIT for $\Phi^{\calP_C}$ for every coherent sheaf $\calF$ on $\ol{\Jac}_\Sigma^{\, 0}$ that is WIT for $\Phi^{\calP_\Sigma}$.
\end{remark}

\begin{remark}
\label{remark on y_0}
Similarly to Remark \ref{rmk:indep y_0}, we can see that the right-hand-side of \eqref{eq relation of FM} and \eqref{eq relation of dual FM} do not depend on the choice of $y_0$. We will see that for \eqref{eq relation of FM}, the case of \eqref{eq relation of dual FM} being the same. If we choose $y_1\in\Sigma\setminus\nu^{-1}(\Sing(C))$ instead, then $\det(\calV_{\Sigma})$ changes to $\det(\calV_{\Sigma})\otimes\dot\nu^*\calU_{\Sigma,y_1}^{\,-k}$. Now, we look at $\Phi^{\calP_\Sigma}_{1 \rightarrow 2}(\tau_{k, y_1,*}\calF^\bullet)$. Let $\wt\tau:\Jac_\Sigma^0\to\Jac_\Sigma^0$ given by $\wt\tau(M)=M\otimes\calO_\Sigma(ky_1-ky_0)$, so that $\tau_{k,y_1}=\wt\tau\circ\tau_{k,y_0}$. Then, using base change, projection formula and \cite[Lemma 6.5]{arinkin}
    \begin{equation*}
\begin{split}
\Phi^{\calP_\Sigma}_{1 \rightarrow 2}(\tau_{k,y_1,*}\calF^\bullet)
&= R\pi_{2,*}\bigl(\pi_1^*\wt\tau_*\tau_{k,y_0,*}\calF^\bullet\otimes\calP_\Sigma\bigr)\\
&\cong R\pi_{2,*}(\id\times\wt\tau)_*\bigl(\pi_1^*\tau_{k,y_0,*}\calF^\bullet\otimes(\id\times\wt\tau)^*\calP_\Sigma\bigr)\\
&\cong R\pi_{2,*}\big(\pi_1^*\tau_{k,y_0,*}\calF^\bullet\otimes(\id\times\tau_{-k,y_0})^*(\pi_2^*\calP_{\Sigma,\calO(y_1-y_0)^k}
      \otimes(\id\times\tau_{k,y_0})^*\calP_\Sigma)\big)\\
&\cong R\pi_{2,*}\big(\pi_1^*\tau_{k,y_0,*}\calF^\bullet\otimes\calP_\Sigma\big)\otimes\calP_{\Sigma,\calO(y_1-y_0)^k}\\
&\cong \Phi^{\calP_\Sigma}_{1 \rightarrow 2}(\tau_{k,y_0,*}\calF^\bullet)\otimes \calU_{\Sigma,y_1}^k.
\end{split}
\end{equation*}
This proves the independence of the choice of $y_0$. 
\end{remark}

\begin{proof}[Proof of Theorem \ref{tm relation of FM dual}]
Recall the construction of $\Phi^{\calP_C}_{1 \rightarrow 2}$ described in \eqref{eq FM Arinkin}. Following the Cartesian diagram
\[
\begin{tikzcd}[column sep = huge]
\ol\Jac_\Sigma^{\,-k} \times \ol\Jac_C \arrow[r, "\t_1"] \arrow[d, "(\check\nu \times \id)"'] & \ol\Jac_\Sigma^{\,-k} \arrow[d, "\check\nu"] 
\\
\ol\Jac_C \times \ol\Jac_C  \arrow[r, "\pi_1"'] & \ol\Jac_C,
\end{tikzcd}
\]
one has 
\begin{equation} \label{eq Phi nu_*Ff 1}
\Phi^{\calP_C}_{1 \rightarrow 2}(\check\nu_*\calF^\bullet) \cong R\t_{2,*}\left (\t_1^*\calF^\bullet \otimes (\check\nu \times \id)^*\calP_C \right ),
\end{equation}
where we have applied flat base change, projection formula and functoriality with respect to 
\[
\begin{tikzcd}[column sep = huge]
\ol\Jac_\Sigma^{\,-k} \times \ol\Jac_C \arrow[d, "(\check\nu \times \id)"'] \arrow[rd, "\t_2"] & 
\\
\ol\Jac_C \times \ol\Jac_C  \arrow[r, "\pi_2"'] & \ol\Jac_C.
\end{tikzcd} 
\]
We now substitute \eqref{eq relation of Poincares} into \eqref{eq Phi nu_*Ff 1} and apply the projection formula with respect to $\t_2$ and $(\rho \times \id)$, and functoriality with respect to the commutative diagrams 
\[
\begin{tikzcd}[column sep = huge]
\ol\Jac_\Sigma^{\,-k} \times \PMod_\nu \arrow[d, "(\id \times \rho)"'] \arrow[rd, "\q_1"] & 
\\
\ol\Jac_\Sigma^{\,-k} \times \ol\Jac_C  \arrow[r, "\t_1"'] & \ol\Jac_\Sigma^{\,-k},
\end{tikzcd}
\qquad
\begin{tikzcd}[column sep = huge]
\ol\Jac_\Sigma^{\,-k} \times \PMod_\nu \arrow[d, "(\id \times \rho)"'] \arrow[r, "\q_2"] & \PMod_\nu \arrow[d, "\rho"]
\\
\ol\Jac_\Sigma^{\,-k} \times \ol\Jac_C  \arrow[r, "\t_2"'] & \ol\Jac_C \, ,
\end{tikzcd}
\]
which allows us to write 
\begin{align}
\begin{split}
\label{eq Phi nu_*Ff 2}
\Phi^{\calP_C}_{1 \rightarrow 2}(\check\nu_*\calF^\bullet) 
& \cong \rho_* \left( \det(\calV_{\Sigma}) \otimes  R\q_{2,*} \left( \q_1^*\calF^\bullet \otimes (\tau_{y_0} \times \dot{\nu})^*\calP_\Sigma \right) \right) \, .
\end{split} 
\end{align}
Due to functoriality with respect to the diagram on the left, and flat base change for the diagram on the right,
\[
\begin{tikzcd}[column sep = huge]
\ol\Jac_\Sigma^{\,-k} \times \PMod_\nu \arrow[d, "(\id \times \dot \nu)"'] \arrow[rd, "\q_1"] & 
\\
\ol\Jac_\Sigma^{\,-k} \times \ol\Jac_\Sigma  \arrow[r, "\r_1"'] & \ol\Jac_\Sigma^{\,-k},
\end{tikzcd}
\qquad
\begin{tikzcd}[column sep = huge]
\ol\Jac_\Sigma^{\,-k} \times \PMod_\nu \arrow[d, "(\id \times \dot\nu)"'] \arrow[r, "\q_2"] & \PMod_\nu \arrow[d, "\dot \nu"]
\\
\ol\Jac_\Sigma^{\,-k} \times \ol\Jac_\Sigma  \arrow[r, "\r_2"'] & \ol\Jac_\Sigma \, ,
\end{tikzcd}
\]
one observes that \eqref{eq Phi nu_*Ff 2} becomes
\begin{align*}
\begin{split}
\Phi^{\calP_C}_{1 \rightarrow 2}(\check\nu_*\calF^\bullet) 
& \cong \rho_* \left( \det(\calV_{\Sigma}) \otimes  \dot\nu^* R\r_{2,*}\left (\r_1^*\calF^\bullet \otimes  (\tau_{k, y_0} \times \id)^*\calP_\Sigma \right) \right) \, . 
\end{split} 
\end{align*}
This proves the statement \eqref{eq relation of FM} after applying projection formula with respect to $\tau_{k, y_0} \times \id$.

Since the definition of $\Phi^{\calP_C^\vee}_{2 \rightarrow 1}$ given in \eqref{eq dual FM Arinkin} is analogous to that of \eqref{eq FM Arinkin} and the relation between $\calP^\vee_C$ and $\calP_\Sigma^\vee$ of \eqref{eq relation of dual Poincares} is analogous to \eqref{eq relation of Poincares}, the proof \eqref{eq relation of dual FM} follows verbatim from the proof of \eqref{eq relation of FM}. 
\end{proof}

\begin{example}\label{babyexample}
We provide here an example where the formula of Theorem \ref{tm relation of FM dual} can be reconfirmed by direct calculation. Suppose that $C$ is a nodal elliptic curve, hence $k=1$, and let the node be $b\in C$. Then $\Jac_C^0\cong C\setminus\{b\}$ and $\ol\Jac_C\cong C$ \cite[Section 3]{kass}. In addition, $\Sigma=\PP^1$, $\Jac_{\PP^1}^0=\{\calO_{\PP^1}\}$ and $\PMod_{\nu}=\PP^1$. Under these identifications, the maps $\nu$ and $\rho$ coincide and $\dot\nu$ is the constant map. Furthermore, $\calV_{\Sigma}=\det(\calV_{\Sigma})=\calO_{\PP^1}(-1)$. Now, take $\calF^\bullet=\calO_{\Jac_{\PP^1}^{-1}}$ the trivial bundle over the point $\Jac_{\PP^1}^{-1}=\{\calO_{\PP^1}(-1)\}$. Then, 
\[\rho_*\big(\det(\calV_{\Sigma})\otimes\dot\nu^*\Phi^{\calP_{\PP^1}}_{1\rightarrow 2}(\tau_{1,y_0,*}\calO_{\Jac_{\PP^1}^{-1}})\big)=\rho_*\calO_{\PP^1}(-1).
\] On the other hand, $\Phi^{\calP_C}_{1\rightarrow 2}\big(\check\nu_*\calO_{\Jac_{\PP^1}^{-1}}\big)=\Phi^{\calP_C}_{1\rightarrow 2}(\calO_{\tilde b})$ where $\tilde b$ is the point in $\ol\Jac_C\cong C$ corresponding to the node $b\in C$. By the universal property of $\calP_C$, $\Phi^{\calP_C}_{1\rightarrow 2}(\calO_{\tilde b})$ is the unique rank $1$ and degree zero torsion-free sheaf on $\ol\Jac_C$ which is not a line bundle. With the isomorphism $\ol\Jac_C\cong C$ this corresponds to the unique such object but now on $C$, this being $\nu_*\calO(-1)$. Hence, 
\[\Phi^{\calP_C}_{1\rightarrow 2}(\calO_{\tilde b})\cong\rho_*\calO_{\PP^1}(-1).\] We then see that the claimed isomorphism holds.
\end{example}

\subsection{Spin-valued Wilson operators}
\label{sec: Spin FM} 

We explain how a spin structure on $\PMod_{\nu}$ and a spin-valued Wilson operator leads to a more symmetric presentation of our Fourier--Mukai transform formulae from Theorem \ref{tm relation of FM dual}. This form of our results will be applied to mirror symmetry for branes in the Hitchin system in a sequel article \cite{FHHO}. The appearance of a spin structure compensation for the fact that $\nu : \Sigma \to C$ is two-to-one over the locus $\Exc(\nu) \subset \Sigma$ of resolved singularities, so a Wilson loop over $C$ (as per \cite[\textsection 6.1]{kapustin&witten}) lifts to a so-called \textit{`Wilson spin-half loop'} over $\Sigma$. 

Precisely, the Wilson operators in question are defined as follows. Given a point $y \in \Exc(\nu) \subset \Sigma$, we choose a square root $\calU_{\Sigma, y}^{1/2}$ of the line bundle $\calU_{\Sigma, y} := \calU_{\Sigma} |_{\{ y \} \times \ol\Jac_{\Sigma}}$. The restriction $\calU_{\Sigma, y}$ is locally free, even though $\calU_{\Sigma}$ is not, because $y \in \Sigma$ is by assumption a smooth point. Then, to any $y \in \Exc(\nu)$ and any subset $I = \{ y_1, \ldots, y_{l}\} \subset \Exc(\nu)$, let us introduce the notation 
\begin{equation}
\label{eq: spin Wilson}
\WW_{\Sigma, y}^{1/2} := (\bullet) \otimes \calU_{\Sigma, y}^{1/2} : D^b(\ol\Jac_{\Sigma}) \to D^b(\ol\Jac_{\Sigma}), 
\quad \WW^{1/2}_{\Sigma, \otimes I} := \WW^{1/2}_{\Sigma, y_1} \circ \cdots \circ \WW^{1/2}_{\Sigma, y_l} ,
\end{equation} 
and refer to these as the \textit{half-twisted Wilson operators}. Note that the definition of $\WW^{1/2}_{\Sigma, \otimes I}$ is independent of the labeling of the elements of $I$ as tensor products commute. These functors naturally satisfy the relations $\WW_{\Sigma, y} \simeq \WW_{\Sigma, y}^{1/2} \circ \WW_{\Sigma, y}^{1/2}$ and $\WW_{\Sigma, I} \simeq \WW_{\Sigma, I}^{1/2} \circ \WW_{\Sigma, I}^{1/2}$. The divisor $I$ of interest for us is the entire resolved locus $I = \Exc(\nu)$. For convenience, let us fix orderings $\RSing(\nu) = \{b_1, \ldots, b_k\}$ and $\Exc(\nu) = \{b^{\pm}_1, \ldots, b^{\pm}_k\}$. The construction of $\WW^{1/2}_{\Sigma, \otimes \Exc(\nu)}$ depends on the family of choices $\{ \calU_{\Sigma, b^{\pm}_i}^{1/2}\}_{b^{\pm}_i \in \Exc(\nu)}$, and we recall from Corollary \ref{spin structures} that such choices define a spin structure on $\PMod_{\nu}$ given by
\begin{equation}
\label{eq: choice of spin}
\omega_{\PMod}^{1/2} :=  \bigotimes\limits_{i=1}^k \calV_{\Sigma, i} \otimes \dot\nu^*\calU^{-1/2}_{\Sigma, b_i^-} \otimes \dot \nu^*\calU^{-1/2}_{\Sigma, b_i^+} .
\end{equation} 
Note that, by construction, there exists a natural equivalence $\det(\calV_{\Sigma}) \otimes (\bullet) \simeq \omega_{\PMod}^{1/2} \otimes \dot\nu^{*} \WW^{1/2}_{\Sigma, \otimes \Exc(\nu)}$, which we use to give the following immediate restatement of Theorems \ref{tm Poincare and normalisation} and \ref{tm relation of FM dual}. 

\begin{corollary}
\label{thm: spin and Wilson}
Let $C$ be an integral nodal curve. Fix a partial normalisation $\nu : \Sigma \to C$ that resolves precisely $k$ nodes. Pick $y_0 \in \Sigma$ such that $\nu(y_0)$ is a smooth point of $C$. Moreover, fix square roots $\calU_{\Sigma, b^{\pm}_i}^{1/2}$ and let $\omega_{\PMod}^{1/2}$ be the corresponding square root on $\PMod_{\nu}$ as per \eqref{eq: choice of spin}. 

\begin{itemize}
    \item One has isomorphisms of Poincar\'e sheaves 
\[
(\check\nu \times \id)^*\calP_C 
\cong (\id \times \rho)_* \left ( \q_2^* \left( \omega_{\PMod}^{1/2} \otimes \dot{\nu}^{*} \calU^{1/2}_{\Sigma, \otimes \Exc(\nu)} \right) \otimes (\tau_{k, y_0} \times \dot{\nu})^*\calP_\Sigma \right ) ,
\]
\[
(\id \times \check\nu)^*\calP_C^{\vee} 
\cong (\rho \times \id)_* \left ( \q_1^* \left( \omega_{\PMod}^{1/2} \otimes \dot{\nu}^{*} \calU^{-1/2}_{\Sigma, \otimes \Exc(\nu)} \right) \otimes (\dot{\nu} \times \tau_{k, y_0})^*\calP_\Sigma^{\vee} \right ) ,
\]
where $\q_i: \ol\Jac_\Sigma^{\,-k} \times \PMod_\nu \to \PMod_\nu$ are the natural projections.
\item For every $\calF^\bullet \in D^b(\ol\Jac_\Sigma^{\,-k})$, one has isomorphisms of transformed sheaves 
\[
\Phi^{\calP_C}_{1 \rightarrow 2} \left ( \check\nu_*\calF^\bullet \right) 
\cong \rho_* \left ( \omega_{\PMod}^{1/2} \otimes \dot\nu^{*} \big( \WW^{1/2}_{\Sigma, \otimes \Exc(\nu)} \circ \Phi^{\calP_\Sigma}_{1 \rightarrow 2}(\tau_{k, y_0,*}\calF^\bullet) \big) \right) , 
\]
\[
\Phi^{\calP_C^{\vee}}_{2 \rightarrow 1} \left ( \check\nu_*\calF^\bullet \right) 
\cong \rho_* \left ( \omega_{\PMod}^{1/2} \otimes \dot\nu^{*} \big( \WW^{-1/2}_{\Sigma, \otimes \Exc(\nu)} \circ \Phi^{\calP_\Sigma^{\vee}}_{2 \rightarrow 1}(\tau_{k, y_0,*}\calF^\bullet) \big) \right) . 
\]
\end{itemize} 
\end{corollary}

\begin{remark}
\label{re: spin independence}
A priori, the formulae in Corollary \ref{thm: spin and Wilson} appear to depend on the choice of square roots $\calU^{1/2}_{\Sigma, b_i^{\pm}}$ and the resultant choice of spin structure $\omega_{\PMod}^{1/2}$. However, the functor 
\[
\omega_{\PMod}^{1/2} \otimes \dot\nu^{*} \WW^{1/2}_{\Sigma, \otimes \Exc(\nu)} ( \bullet)  \simeq  \det(\calV_{\Sigma}) \otimes (\bullet) , 
\]
and therefore the formulae of Corollary \ref{thm: spin and Wilson}, are independent of these choices. 
\end{remark}

\subsection{Restriction to compactified Prym varieties}\label{sec:Prym}

Motivated by the study of mirror symmetry between moduli spaces of $\SL(m,\CC)$ and $\PGL(m,\CC)$-Higgs bundles \cite{hausel&thaddeus, donagi&pantev}, this section shows that our main constructions and results are well-behaved upon restriction to compactified Prym varieties. In \cite{FHR, groechenig&shen}, an equivalence is established between the derived category of a compactified Prym variety and a $\Gamma$-equivariant incarnation of the latter, where $\Gamma$ is a group of $m$-torsion points of a base smooth curve. This can be understood as a duality statement between cuspidal $\SL(m,\CC)$ and $\PGL(m,\CC)$-Hitchin fibres, generalizing the autoduality of cuspidal $\GL(m,\CC)$-Hitchin fibres obtained by Arinkin \cite{arinkin}. 

We begin by describing the compactified Prym varieties and their properties. Consider a partial normalisation map $\nu : \Sigma \to C$ between irreducible nodal curves together with two ramified $m:1$-coverings $\beta_C : C \to X$ and $\beta_\Sigma : \Sigma \to X$  over a smooth projective curve $X$. We assume $\beta_C$ and $\beta_{\Sigma}$ fit into the commutative diagram
\begin{equation} \label{eq nu and beta commute}
\begin{tikzcd}
\Sigma & & C
\\
& X. &
\arrow[from=1-1, to=1-3, "\nu"]
\arrow[from=1-1, to=2-2, "\beta_\Sigma"']
\arrow[from=1-3, to=2-2, "\beta_C"]
\end{tikzcd} 
\end{equation}

Let us define the {\it norm map} associated to $\beta_C$ to be the degree preserving morphism
\[
\Nm_{\beta_C} (\bullet) := \det(\beta_{C,*}\calO_C)^{-1} \otimes \det \circ \beta_{C,*}  (\bullet) : \ol\Jac_C^{\,d} \longrightarrow \Jac_X^d,
\]
and $\Nm_{\beta_\Sigma} : \ol\Jac_\Sigma^{\,d} \longrightarrow \Jac_X^d$ is defined identically. For $d=0$, the associated \textit{compactified Prym varieties} are defined to be the preimage under the norm map of the trivial line bundle,
\[
\ol{\Prym}_C := \Nm_{\beta_C}^{-1}(\calO_X) \quad \textnormal{and} \quad \ol\Prym_\Sigma := \Nm_{\beta_\Sigma}^{-1}\left (\calO_X \right). 
\]
For the non-zero degree $d = -k$ we give a twisted version of the Prym construction, compatible with the normalisation $\nu : \Sigma \to C$. It is convenient to consider the geometric point $\calO(-\beta_C(\RSing(\nu)) \in \Jac_X^{-k}$ defined by the resolved singularities $\RSing(\nu) \subset C$ and take the preimage
\[
 \ol{\Prym}_\Sigma^{\,-k} := \Nm_{\beta_\Sigma}^{-1}\left (\calO(-\beta_C(\RSing(\nu))) \right).
\] 
The compactified Prym varieties $\ol{\Prym}_C$, $\ol\Prym_\Sigma$ and $\ol{\Prym}_{\Sigma}^{\,-k}$ are Gorenstein varieties with trivial dualising sheaf (see \cite[Corollary 3.2]{FHR} for instance). They come equipped with closed embeddings 
\[
\jmath_C : \ol{\Prym}_C \longhookrightarrow \ol\Jac_C, \quad \jmath_\Sigma : \ol\Prym_\Sigma \longhookrightarrow \ol\Jac_\Sigma \quad \textnormal{and} \quad \jmath_\Sigma' : \ol\Prym_\Sigma^{\,-k} \longhookrightarrow \ol\Jac_\Sigma^{\,-k}.
\]
The choice of preimage in the construction of $\ol{\Prym}_\Sigma^{\,-k}$ is designed for the following purpose. The translation isomorphism 
\[
\morph{\ol \Jac_\Sigma^{\,-k}}{\ol \Jac_\Sigma}{\calF}{\calF\otimes \calO_\Sigma(k y_0),}{}{\tau_{k, y_0}}  
\] 
associated to a smooth point $y_0 \in \Sigma$, can be shown, after an additional twist, to restrict to $\ol{\Prym}_\Sigma$, simultaneously trivialising the torsors $\ol \Jac_\Sigma^{\,-k}$ and $\ol{\Prym}_\Sigma^{\,-k}$. We prove this in parallel with the following compatibility results between normalisation maps and Prym constructions.  Let 
\[\Gamma := \Jac_X[m] \, , \] 
denote the subgroup of $\Jac_X$ consisting of $n$-torsion line bundles.

\begin{lemma}
\label{Lemma:compatib-maps-Pryms}
Let $C$ be an integral nodal curve. Let $\nu : \Sigma \to C$ be a partial normalisation resolving the divisor $\RSing(\nu) \subset \Sing(C)$ of cardinality $k$. Then, the following hold:
\begin{enumerate}
\item The tensoral action of $\Gamma$ on $\ol\Jac_C$, $\ol\Jac_\Sigma$ and $\ol{\Jac}_\Sigma^{\,-k}$ restricts to $\ol{\Prym}_C$, $\ol\Prym_\Sigma$ and $\ol{\Prym}_\Sigma^{\,-k}$.

\item The morphism $\check \nu$ restricts to the compactified Prym varieties, giving rise to a closed embedding
\[
\breve{\nu} : \ol\Prym^{\,-k}_\Sigma \longhookrightarrow \ol{\Prym}_C,
\]
which then lies in the commutative square
\begin{equation} \label{eq relation jmath nu}
\begin{tikzcd}
\ol\Prym^{\,-k}_\Sigma & \ol{\Prym}_C
\\
\ol\Jac_\Sigma^{\,-k} & \ol\Jac_C.
\arrow[from=1-1, to=1-2, "\breve\nu"]
\arrow[from=1-1, to=2-1, "\jmath_\Sigma'"']
\arrow[from=1-2, to=2-2, "\jmath_C"]
\arrow[from=2-1, to=2-2, "\check\nu"']
\end{tikzcd}
\end{equation} 

\item Let $\calL$ be an $m$-th root of $\calO_X (\beta_C(\RSing(\nu) - k x_0))$ and $\tau_{\calL}$ be the translation 
\begin{equation} \label{eq traslation including pull-back}
\tau_{\calL}: \ol\Jac_\Sigma \xrightarrow{\ \cong \ } \ol\Jac_\Sigma, \quad \calF \to \calF \otimes \beta_\Sigma^*\calL\,.
\end{equation}
Then the following diagram commutes:
\begin{equation} \label{eq relation tau and jmath}
\begin{tikzcd}
\ol{\Prym}_\Sigma^{\,-k} & \Nm_{\beta_\Sigma}^{-1}(\calO_X(k\beta_\Sigma(y_0)-\beta_C(\RSing(\nu)))) & \ol\Prym_\Sigma
\\
\ol{\Jac}_\Sigma^{\,-k} & \ol{\Jac}_\Sigma & \ol{\Jac}_\Sigma\,.
\arrow[from=1-1, to=1-2, "\mathring{\tau}_{k, y_0}", "\cong"']
\arrow[from=1-2, to=1-3, "\mathring{\tau}_{\calL}", "\cong"']
\arrow[from=1-2, to=2-2, hook]
\arrow[from=1-3, to=2-3, "\jmath_\Sigma", hook]
\arrow[from=1-1, to=2-1, "\jmath_\Sigma'"', hook]
\arrow[from=2-1, to=2-2, "\tau_{k, y_0}"', "\cong"]
\arrow[from=2-2, to=2-3, "\tau_{\calL}"', "\cong"]
\end{tikzcd} 
\end{equation}
Here $\mathring{\tau}_{k, y_0}$, $\mathring{\tau}_{\calL}$ denote the restricted translation isomorphisms.

\item In the following Cartesian square 
\begin{equation} \label{eq def of ring rho}
\begin{tikzcd}
\PMod_{\nu} \times_{\ol{\Jac}_C} \ol\Prym_C \arrow[d, "\imath_{\nu}"']\arrow[r, "\mathring\rho"] 
& \ol\Prym_C\arrow[d, "\jmath_{C}"] 
\\ 
\PMod_{\nu} \arrow[r, "\rho"'] 
& \ol{\Jac}_{C} \, , 
\end{tikzcd}
\end{equation}
the composition of the inclusion $\imath_{\nu}$ with $\dot{\nu}$ factors through $\ol \Prym_\Sigma$, giving rise to the commutative diagram
\begin{equation} \label{eq def of ring nu}
\begin{tikzcd}
\PMod_{\nu} \times_{\ol{\Jac}_{C}} \ol{\Prym}_{C} \arrow[d, "\imath_{\nu}"']\arrow[r, "\mathring\nu"] 
& \ol{\Prym}_{\Sigma} \arrow[d, "\jmath_{\Sigma}"] 
\\ 
\PMod_{\nu} \arrow[r, "\dot\nu"'] 
& \ol{\Jac}_{\Sigma} \,. 
\end{tikzcd}
\end{equation} 
\end{enumerate}
\end{lemma}

\begin{proof}
As a direct consequence of the projection formula, we have, for any $\calL \in \Jac_X^0$;
\begin{equation} \label{eq Nm and otimes}
\Nm_{\beta_C} \left ( \beta_C^*\calL \otimes \bullet \right) \cong \calL^m \otimes \Nm_{\beta_C}(\bullet) \quad \textnormal{and} \quad \Nm_{\beta_\Sigma} \left ( \beta_\Sigma^*\calL \otimes \bullet \right ) \cong \calL^m \otimes \Nm_{\beta_\Sigma}(\bullet).
\end{equation}
This immediately implies (1). For (2), we make use of the determinant formula
\[
\det(R\beta_{C, *}\calO_C) \cong \det(R\beta_{\Sigma, *}\calO_\Sigma) \otimes \calO_X(-\beta_C(\RSing(\nu))), 
\]
derived from pushforward of the short exact sequence
\[
0 \to \calO_C \to \nu_{*}\calO_{\Sigma} \to \calO_{\RSing(\nu)} \to 0 .  
\]  
Given a point $\calF \in \ol\Prym^{-k}_\Sigma$, we then obtain the following isomorphisms: 
\begin{align*}
\Nm_{\beta_C}(\nu_{*}\calF) 
& = \det(\beta_{C, *}\calO_C)^{-1} \otimes \det(\beta_{C, *}\nu_* \calF )  \\
& \cong \det(\beta_{\Sigma, *}\calO_{\Sigma})^{-1} \otimes \calO_X(\beta_{C}(\RSing(\nu))) \otimes \det(\beta_{\Sigma, *}\calF)   \\
& \cong \Nm_{\beta_{\Sigma}}(\calF) \otimes \calO_X(\beta_{C}(\RSing(\nu))) \\  
& \cong \calO_X , 
\end{align*}
and thus $\check\nu(\calF) = \nu_{*}\calF$ lands in $\ol\Prym_{C}$. This proves (2). For (3), by definition of the $\ol \Prym^{-k}_\Sigma$ we have  
\begin{equation} \label{eq tau Prym is another fibre}
\tau_{k,y_0} \big ( \ol{\Prym}_\Sigma^{\,-k}  \big ) \cong \Nm_{\beta_\Sigma}^{-1} \big ( \calO_X(-\beta_C\RSing(\nu) +k \beta_{\Sigma}(y_0) ) \big ),
\end{equation}
and it follows from \eqref{eq Nm and otimes} that composing with $\tau_{\calL} = (\bullet) \otimes \beta_\Sigma^*\calL$ lands in $\ol\Prym^0_\Sigma$.

For (4), we push forward under $\beta_C$ the short exact sequence \eqref{eq preimage of tau} associated to a point in $\PMod_{\nu} \times_{\ol{\Jac}} \ol\Prym_C$, followed by taking determinants. We observe that the composition of the inclusion $\imath_{\nu}$ with the morphism $\dot{\nu} : \PMod_{\nu} \to \ol{\Jac}_{\Sigma}$ factors through $\ol \Prym_\Sigma$ (see \cite{gothen&oliveira} for more details).
\end{proof} 

\begin{lemma} \label{lm equivariance of PMod}
There is a $\Gamma$-action on $\PMod_{\nu}$, such that the morphisms 
\[  \ol\Jac_C \xleftarrow{\ \rho \ } \PMod_{\nu} \xrightarrow{\ \dot\nu\ } \ol\Jac_\Sigma
\] are $\Gamma$-equivariant. In particular, the universal bundle $\calV_{\Sigma} \to \PMod_{\nu}$ inherits a $\Gamma$-action. \end{lemma} 
\begin{proof}
Given a parabolic module $(M,V)$, with $M\in\ol\Jac_\Sigma$, a $n$-torsion line bundle $L\in\Gamma$ acts by
\[ L\cdot(M,V)=(\beta_\Sigma^*L\otimes M, W),
\] where  
\[ W\ :=\ V\otimes \beta^*_\Sigma L\ = \ \bigoplus_{i=1}^k V_i\otimes(\beta^*_\Sigma L)_{b_i^\pm} \ \subset\  \bigoplus_{i=1}^k  (M_{b_i^-}\oplus M_{b_i^+})\otimes(\beta_\Sigma^*L)_{b_i^\pm}.
\]
Note that there is a natural isomorphism of stalks $(\beta_\Sigma^*L)_{b_i^+} \cong (\beta_\Sigma^*L)_{b_i^-} \cong  L_{\beta_C(b)}$. Hence, $W$ is well-defined. By definition of the action $\dot{\nu}$ is $\Gamma$-equivariant. Furthermore, the exact sequence \eqref{eq preimage of tau} is compatible with the $\Gamma$-action on $\ol\Jac_C$ and the action on $\PMod_{\nu}$ just defined. In particular, $\rho$ is $\Gamma$-equivariant.
\end{proof}


Let us now restrict the Poincaré sheaves to our compactified Prym varieties. Consider the pullback the Poincar\'e sheaf along the closed embeddings of the Prym varieties,
\begin{equation} \label{eq restricted Poincare}
\calR_C := (\jmath_C \times \jmath_C)^*\calP_C, \quad \calR_\Sigma := \left ( \jmath_\Sigma \times \jmath_\Sigma \right )^* \calP_\Sigma,
\end{equation}
supported on $\ol\Prym_C \times \ol\Prym_C$ and $\ol\Prym_\Sigma \times \ol \Prym_\Sigma$ respectively. 
These sheaves can be related via the following Prymian version of Theorem \ref{tm Poincare and normalisation}. 

\begin{proposition}
With the notation introduced above, one has
\begin{equation} \label{eq relation between Rr s}
(\breve \nu \times \id)^* \calR_C \cong (\id \times \mathring\rho)_*( \mathring q_2^* \imath_\nu^* \det(\calV_\Sigma) \otimes ((\mathring\tau_{\calL} \circ \mathring \tau_{k,y_0}) \times \mathring \nu)^* \calR_\Sigma) \, ,
\end{equation}
where $\mathring q_2$ is the natural projection from $\ol{\Prym}^{\, -k}_\Sigma \times \left ( \PMod_{\nu} \times_{\ol{\Jac}_{C}} \ol{\Prym}_{C} \right )$ to $\PMod_{\nu} \times_{\ol{\Jac}_{C}} \ol{\Prym}_{C}$.
\end{proposition}

\begin{proof}
Applying $(\jmath'_\Sigma \times \jmath_C)^*$ to \eqref{eq relation of Poincares}, one gets the identification
\begin{equation}\label{eq previous relation of Rr}
(\breve \nu \times \id)^* (\jmath_C \times \jmath_C)^*\calP_C \cong (\id \times \mathring\rho)_* \left ( \mathring q_2^* \imath_\nu^* \det(\calV_\Sigma) \otimes ((\mathring\tau_{\calL} \circ \mathring \tau_{k,y_0}) \times \mathring \nu)^* \left ( \jmath_\Sigma \times \jmath_\Sigma \right )^*(\tau_{\calL^{-1}} \times \id)^*\calP_\Sigma \right) \, ,
\end{equation}
after an iteration of base change under \eqref{eq def of ring rho} and functoriality under \eqref{eq relation jmath nu}, \eqref{eq def of ring nu} and
\[
\begin{tikzcd}
\ol{\Prym}^{\, -k}_\Sigma \times \left ( \PMod_{\nu} \times_{\ol{\Jac}_{C}} \ol{\Prym}_{C} \right ) \arrow[d, "\jmath'_\Sigma \times \imath_{\nu}"']\arrow[r, "\mathring q_2"] 
& \PMod_{\nu} \times_{\ol{\Jac}_{C}} \ol{\Prym}_{C} \arrow[d, "\imath_{\nu}"] 
\\ 
\ol\Jac^{\, -k}_\Sigma \times \PMod_{\nu} \arrow[r, "q_2"'] 
& \PMod_{\nu} \, .
\end{tikzcd}
\]
Thanks to \cite[Lemma 6.5]{arinkin} and \cite[Proposition 4.5]{FHR}, one has for any line bundle $N \in \Jac_X$ and $\tau_N$ defined as in \eqref{eq traslation including pull-back},
\[
(\tau_{N} \times \id)^*\calP_\Sigma \cong \p_{2}^*\calP_{\Sigma}|_{\{ \beta_{\Sigma}^*N\} \times \ol \Jac } \otimes \calP_\Sigma \cong \p_{2}^*\Nm^*\calP_{X}|_{\{N\} \times \Jac_X } \otimes \calP_\Sigma,
\]
where we recall that $\p_{2}$ is the projection to the second factor $\ol\Jac_\Sigma \times \ol\Jac_\Sigma \to \ol\Jac_\Sigma$. If further, $\mathring\p_{2}$ denotes the projection to the second factor $\ol\Prym_\Sigma \times \ol\Prym_\Sigma \to \ol\Prym_\Sigma$,  
\begin{equation}\label{eq pull-back of Poincare_Sigma under jmath}
\left ( \jmath_\Sigma \times \jmath_\Sigma \right )^*(\tau_{N} \times \id)^*\calP_\Sigma \cong \mathring \p_2^* \jmath_\Sigma^*\Nm^*\calP_{X}|_{\{N\} \times \Jac_X } \otimes\calR_\Sigma \cong \calR_\Sigma,
\end{equation}
since $\jmath_\Sigma$ amounts to the inclusion of the central fibre of the Norm map $\Nm$. We obtain \eqref{eq relation between Rr s} after substituting into \eqref{eq previous relation of Rr} the identifications \eqref{eq restricted Poincare} and \eqref{eq pull-back of Poincare_Sigma under jmath} when $N= \calL^{-1}$.
\end{proof}

We consider the $\Gamma$-action on the products $\ol\Prym_C \times \ol\Prym_C$ and $\ol\Prym^{\,-k}_\Sigma \times \ol\Prym_\Sigma$ given by the standard $\Gamma$-action on the second factor and the trivial one on the first. Observe that the morphisms $\jmath_C$, $\jmath_\Sigma$, $\rho$ and $\dot\nu$ are $\Gamma$-equivariant by Lemma \ref{lm equivariance of PMod}. Then, so are $\mathring \rho$, $\imath_\nu$ and $\mathring \nu$ by construction, hence \eqref{eq def of ring rho} and \eqref{eq def of ring nu} are, respectively, Cartesian and commutative diagrams in the $\Gamma$-equivariant category. It follows from Lemma \ref{lm equivariance of PMod} that $\imath_{\nu}^{*}\det(\calV_{\Sigma})$ is $\Gamma$-equivariant. 

Crucially, the sheaves in \eqref{eq restricted Poincare} are $\Gamma$-equivariant \cite[Proposition 4.5]{FHR}. Furthermore, they can be equipped with compatible $\Gamma$-equivariant structures. 

\begin{proposition} \label{pr compatible Gamma-structures}
A $\Gamma$-equivariant structure on $\calR_C$ and $\imath_\nu^* \det(\calV_\Sigma)$ naturally induce a $\Gamma$-equivariant structure on $\calR_\Sigma$.
\end{proposition}

\begin{proof}
Given a certain $L \in \Gamma$, denote by $\tau_{L,C}$ and $\tau_{L, \Sigma}$ the shifts provided by tensoring under $\beta_C^*L$ and $\beta_\Sigma^*L$. An isomorphism between $\calR_C$ and $(\id \times \tau_{L,C})^*\calR_C$ induces an isomorphism between $(\id \times \mathring \rho)^*(\breve \nu \times \id)^*\calR_C$ and $(\id \times \tau_{L, \Sigma})^*(\id \times \mathring \rho)^*(\breve \nu \times \id)^*\calR_C$, where we made use of the $\Gamma$-equivariance of $\mathring \rho$. Under the identification \eqref{eq relation between Rr s} and the counit transformation, this further provides an isomorphism between $( \mathring q_2^* \imath_\nu^* \det(\calV_\Sigma) \otimes ((\mathring\tau_{\calL} \circ \mathring \tau_{k,y_0}) \times \mathring \nu)^* \calR_\Sigma)$ and $(\id \times \tau_{L, \Sigma})^*( \mathring q_2^* \imath_\nu^* \det(\calV_\Sigma) \otimes ((\mathring\tau_{\calL} \circ \mathring \tau_{k,y_0}) \times \mathring \nu)^* \calR_\Sigma)$. Recalling that $\mathring \nu$ is $\Gamma$-equivariant, the previous isomorphism together with the one provided by the $\Gamma$-equivariant structure on $\imath_\nu^* \det(\calV_\Sigma)$, induces an isomorphism between $\calR_\Sigma$ and $(\id \times \tau_{L, \Sigma})^*\calR_\Sigma$.
\end{proof}

Let us equip $\calR_C$ and $\calR_\Sigma$ with compatible $\Gamma$-equivariant structures, as per Proposition \ref{pr compatible Gamma-structures}. One can then consider the following integral functors in the category of $\Gamma$-equivariant sheaves, with $\Gamma$ acting trivially on the first factor,
\[
\Psi^{\calR_{C}}_{1 \rightarrow 2} : D^b\left (\ol\Prym_C \right ) \to D^b \left (\ol\Prym_C, \Gamma \right ) , \quad \quad \Psi^{\calR_{\Sigma}}_{1 \rightarrow 2} : D^b\left (\ol\Prym_{\Sigma} \right) \to D^b\left (\ol\Prym_\Sigma, \Gamma \right ) \, . 
\]
Both functors are shown to be equivalences of categories in \cite[Thm 4.8]{FHR}, \cite[Thm 4.7]{groechenig&shen}. The following Prymian analogue of Theorem \ref{tm relation of FM dual} describes the relation between them. 

\begin{corollary} 
\label{co: SL_n transform} 
Let $C$ be an integral nodal curve with arithmetic genus $g$ and a partial normalisation $\nu : \Sigma \to C$. Consider compatible $m$-coverings $C \to X$ and $\Sigma \to X$ and their compactified Pryms, as described above. Pick $y_0 \in \Sigma$ such that $\nu(y_0)$ lies in the smooth locus of $C$. Then, for every $\calF^\bullet \in D^b(\ol\Prym_{\Sigma}^{\,-k})$, one has the isomorphism
\begin{equation} \label{eq Psi}
\Psi^{\calR_C}_{1 \rightarrow 2} \left ( R\breve\nu_*\calF^\bullet \right ) \cong  \mathring\rho_* \left( \imath_{\nu}^{*}\det(\calV_{\Sigma}) \otimes \mathring\nu^* \Psi^{\calR_\Sigma}_{1 \rightarrow 2} (\mathring\tau_{\calL*}\mathring\tau_{k, y_0,*} \calF^\bullet ) \right). 
\end{equation} 
\end{corollary}
\begin{proof} Starting from \eqref{eq relation between Rr s}, the proof is analogous to that of Theorem \ref{tm relation of FM dual}. 
\end{proof}

\begin{remark}
Note that Theorem \ref{co: SL_n transform} does not depend on the choice of $m$-th root $\calL$ of $\calO_X(k\beta_\Sigma(y_0)-\beta_C(\RSing(\nu)))$. Indeed, choosing a different $m$-th root $L \otimes \calL$ for any $L \in \Gamma$, provides
\[
\Psi^{\calR_\Sigma}_{1 \rightarrow 2} (\mathring\tau_{L\otimes\calL,*}\mathring\tau_{k, y_0,*} \calF^\bullet ) \cong \Psi^{\calR_\Sigma}_{1 \rightarrow 2} \left (\mathring\tau_{L,*}(\mathring\tau_{\calL,*}\mathring\tau_{k, y_0,*} \calF^\bullet) \right ) \cong \left [ \calR_\Sigma |_{\{ L \} \times \ol\Jac}  \right ] \otimes \Psi^{\calR_\Sigma}_{1 \rightarrow 2} (\mathring\tau_{\calL,*}\mathring\tau_{k, y_0,*} \calF^\bullet ).
\]
After \eqref{eq pull-back of Poincare_Sigma under jmath}, one has
\[
\calR_\Sigma |_{\{ L \} \times \ol\Jac} \cong \calO_{\ol\Jac} 
\]
naturally equipped with the trivial $\Gamma$-equivariant structure. Hence, 
\[
\Psi^{\calR_\Sigma}_{1 \rightarrow 2} (\mathring\tau_{L\otimes\calL,*}\mathring\tau_{k, y_0,*} \calF^\bullet ) \cong \Psi^{\calR_\Sigma}_{1 \rightarrow 2} (\mathring\tau_{\calL,*}\mathring\tau_{k, y_0,*} \calF^\bullet ),
\]
giving the same description of \eqref{eq Psi}.
\end{remark}

We dedicate the rest of the section to discuss the role of endoscopy. Inspired by \cite{ngo}, \cite{frenkel&witten} and \cite{hausel&pauly}, we say that the finite flat covering $\beta_C : C \to X$ is of {\it endoscopic type} if the action of $\Gamma$ on the corresponding compactified Prym variety is not free. Otherwise, we say that $\beta_C$ is of {\it non-endoscopic type}. By \cite[Lemma 3.3 and Proposition 3.4]{FHR}, whenever $C$ is irreducible, $\beta_C$ is of endoscopic type if and only if its composition with the (full) normalisation of $C$ provides an unramified cover of $X$.

In the framework of \eqref{eq nu and beta commute}, as both $C$ and $\Sigma$ have the same (full) normalisation, one has that $\beta_C$ is of the same type (endoscopic or non-endoscopic) as $\beta_\Sigma$. In the non-endoscopic case, the quotient maps  
\[
\xi_C: \ol\Prym_C \to \ol\Prym_C/\Gamma , \quad \quad 
\xi_\Sigma: \ol\Prym_\Sigma \to \ol\Prym_\Sigma/\Gamma , 
\]
are \'etale, and the quotients are connected reduced projective Gorenstein schemes of finite type \cite[Proposition 3.6]{FHR}. Furthermore, $\calR_C$ and $\calR_\Sigma$ descend along $\xi_C$ and $\xi_\Sigma$ \cite[Corollary 5.1]{FHR}; so there exists sheaves $\calQ_C$ and $\calQ_\Sigma$ over $\ol\Prym_\Sigma\times \ol{\Prym}_C/\Gamma$ and $\mathring\tau_{\calL}^{-1} \left (\ol\Prym_{\Sigma} \right ) \times \ol\Prym_\Sigma/\Gamma$ respectively, such that
\[
\calR_C \cong ( \id \times \xi_C)^{*} \calQ_C, \quad \quad \calR_\Sigma \cong ( \id \times \xi_\Sigma)^{*} \calQ_\Sigma.
\]
Taking $\calQ_C$ and $\calQ_\Sigma$ as integral kernels, we define the integral functors 
\begin{equation} \label{eq Phi^Q}
\Phi^{\calQ_{C}}_{1 \rightarrow 2} : D^b\left (\ol\Prym_C \right ) \to D^b\left (\ol\Prym_C/\Gamma \right) , \quad \quad 
\Phi^{\calQ_{\Sigma}}_{1 \rightarrow 2} : D^b\left (\ol\Prym_{\Sigma} \right ) \to D^b\left (\ol\Prym_\Sigma/\Gamma \right ). 
\end{equation}

\begin{theorem}[Theorem 5.2 of \cite{FHR}] \label{tm Phi^Q equivalence}
Let $C$ be an integral nodal curve with arithmetic genus $g$ and a partial normalisation $\nu : \Sigma \to C$ resolving precisely $k$ nodes. Suppose that $C$ and $\Sigma$ are equipped with degree $m$-coverings $\beta_C : C \to X$ and $\beta_\Sigma : \Sigma \to X$ of non-endoscopic type. The associated integral functors $\Phi^{\calQ_C}_{1 \rightarrow 2}$ and $\Phi^{\calR_\Sigma}_{1 \rightarrow 2}$ are equivalences of categories fitting in the commutative diagrams
\[
\begin{tikzcd}[column sep = huge]
D^b( \ol{\Prym}_C ) \arrow[r, "\Phi^{\calQ_C}_{1 \rightarrow 2}"]  \arrow[rd, "\Psi^{\calR_C}_{1 \rightarrow 2}"'] &  D^b( \ol{\Prym}_C/\Gamma ) \arrow[bend right=30,swap,"{\xi_C^*}" pos=0.56]{d} \\
 & D^b( \ol{\Prym}_C, \Gamma ) \arrow[bend right=30,swap,"{\xi_{C,*}^\Gamma}"]{u} 
\end{tikzcd}, \qquad \begin{tikzcd}[column sep = huge]
D^b\left (\ol\Prym_{\Sigma} \right) \arrow[r, "\Phi^{\calQ_\Sigma}_{1 \rightarrow 2}"]  \arrow[rd, "\Psi^{\calR_\Sigma}_{1 \rightarrow 2}"'] &  D^b( \ol\Prym_\Sigma/\Gamma ) \arrow[bend right=30,swap,"{\xi_\Sigma^*}" pos=0.56]{d} \\
 & D^b( \ol\Prym_\Sigma, \Gamma ) \arrow[bend right=30,swap,"{\xi_{\Sigma,*}^\Gamma}"]{u}  
\end{tikzcd} ,
\]
where $\xi_{C,*}^\Gamma$ and $\xi_{\Sigma,*}^\Gamma$ amount to the $\Gamma$-invariant part of their corresponding push-forwards. 
\end{theorem}

In the non-endoscopic case, the action of $\Gamma$ is faithful. Denote by 
\[ \wt{\xi}_\Sigma:  \PMod_{\nu} \times_{\ol{\Jac}_{C}}\ol{\Prym}_{C} \to \PMod_{\nu} \times_{\ol{\Jac}_{C}} \ol{\Prym}_{C}/\Gamma 
\] the quotient of the Prymian version of the moduli space of parabolic modules by the $m$-torsion points. By Proposition \ref{pr compatible Gamma-structures} $\imath_{\nu}^{*}\det(\calV_{\Sigma})$ can be equipped with a $\Gamma$-equivariant structure. By the Drézet--Narasimhan--Kempf descent criterion \cite[Theorem 2.3]{drezet&narasimhan}, it descends along $\wt\xi_\Sigma$ to a sheaf $[\imath_{\nu}^{*}\det(\calV_{\Sigma})]$ on the quotient $\PMod_{\nu} \times_{\ol{\Jac}_{C}} \ol{\Prym}_{C}/\Gamma$, such that 
\begin{equation} \label{eq descent of det Vv}
\imath_{\nu}^{*}\det(\calV_{\Sigma}) \cong \wt\xi_\Sigma^* [\imath_{\nu}^{*}\det(\calV_{\Sigma})].
\end{equation}
Given all of the above, one can naturally derive the following from Theorem \ref{tm Phi^Q equivalence} and Theorem \ref{co: SL_n transform}.

\begin{corollary} 
\label{co: SL_n transform non-endoscopic}
Let $C$ be an integral nodal curve with arithmetic genus $g$ and partial normalisation $\nu : \Sigma \to C$, which resolves $k$ nodes. Suppose that $C$ and $\Sigma$ are equipped with $m$-coverings $\beta_C : C \to X$ and $\beta_\Sigma : \Sigma \to X$ of non-endoscopic type. Denote by $[\mathring \rho]$, $[\mathring \nu]$, $[\imath_\nu]$ the corresponding morphisms between $\Gamma$-orbits induced by $\mathring \rho$, $\mathring \nu$, $\imath_\nu$. Then, for $\calF^\bullet \in D^b(\ol\Prym_{\Sigma}^{\,-k})$, one has the isomorphism
\[
\Phi^{\calQ_C}_{1 \rightarrow 2} \left ( \breve\nu_*\calF^\bullet \right ) \cong  [\mathring\rho]_* \left( [\imath_{\nu}^{*}\det(\calV_{\Sigma})] \otimes [\mathring\nu]^* \Phi^{\calQ_\Sigma}_{1 \rightarrow 2} (\mathring\tau_{\calL,*}\mathring\tau_{k, y_0,*} \calF^\bullet ) \right). 
\]
\end{corollary}

\printbibliography

\end{document}